\documentclass[12pt]{amsart}
\newtheorem{thm}{Theorem}[section]
\newtheorem{lem}[thm]{Lemma}
\newtheorem{cor}[thm]{Corollary}
\newtheorem{prop}[thm]{Proposition}
\newtheorem{definition}[thm]{Definition}

\newtheorem*{ack*}{Acknowledgment}
\usepackage[bookmarks]{hyperref}
\usepackage{graphicx,amscd,amsthm,amsfonts,amsopn,amssymb,verbatim,enumerate}

\usepackage[letterpaper,left=1in,top=1in,right=1in,bottom=1in]{geometry}

\def\half{ \frac{1}{2}}
\def\D{\partial}
\def\T{{\mathbb T}}
\def\R{{\mathbb R}}

\def\N{{\mathbb N}}
\def\C{{\mathbb C}}
\def\nint{\mathop{\diagup\kern-13.0pt\int}}
\def\Z{{\mathbb Z}}

\def\bas{\begin{align*}}
\def\eas{\end{align*}}
\def\bi{\begin{itemize}}
\def\ei{\end{itemize}}

\def\emph#1{{\it #1}}

\def\FF{{\mathcal F}}
\def\LL{{\mathcal L}}
\def\TT{{\mathcal T}}
\def\MM{{\mathcal M}}
\def\ZZ{{\mathcal Z}}
\def\eps{{\epsilon}}
\def\OO{{\mathcal O}}
\theoremstyle{definition}
\newtheorem{rem}[thm]{Remark}
\numberwithin{equation}{section}

\begin{document}

\begin{abstract}
The gravity water waves equations describe the evolution of the surface of an incompressible, irrotational fluid in the presence of gravity. The classical regularity threshold for the well-posedness of this system requires initial velocity field in $H^s$, with $s > \half + 1$, and can be obtained by proving standard energy estimates. This threshold was improved by additionally proving Strichartz estimates with loss. In this article, we establish the well-posedness result for $s > \half + 1 - \frac{1}{8}$, corresponding to proving lossless Strichartz estimates. This provides the sharp regularity threshold with respect to the approach of combining Strichartz estimates with energy estimates.
\end{abstract}

\title{Low Regularity Solutions for Gravity Water Waves II: The 2D Case}

\author{Albert Ai}
\email{aai@math.berkeley.edu}
\thanks{The author was supported by the Simons Foundation.}
\date{November 24, 2018}

\maketitle

\tableofcontents

\section{Introduction}

We consider the Cauchy problem for the gravity water waves equations in one surface dimension, without surface tension, over an incompressible, irrotational fluid flow.

Precisely, let $\Omega$ denote a time dependent fluid domain contained in a fixed domain $\OO$, located between a free surface and a fixed bottom: 
$$\Omega = \{(t, x, y) \in [0, 1] \times \OO \ ; \ y < \eta(t, x)\}$$
where $\OO \subseteq \R^d \times \R$ is a given connected open set, with $x \in \R^d$ representing the horizontal spatial coordinate and $y \in \R$ representing the vertical spatial coordinate. Our results will concern the case $d = 1$ of one surface dimension, though we often use the notation of general $d$ to provide context.

We assume the free surface
$$\Sigma = \{(t, x, y) \in [0, 1] \times \R^d \times \R: y = \eta(t, x)\}$$
is separated from the fixed bottom $\Gamma = \D\Omega \backslash \Sigma$ by a curved strip of depth $h > 0$:
\begin{align}\label{width}
\{(x, y) \in \R^d \times \R: \eta(t, x) - h < y < \eta(t, x) \} \subseteq \OO.
\end{align}

We consider an incompressible, irrotational fluid flow in the presence of gravity but not surface tension. In this setting the fluid velocity field $v$ may be given by $\nabla_{x, y} \phi$ where $\phi: \Omega \rightarrow \R$ is a harmonic velocity potential,
$$\Delta_{x, y} \phi = 0.$$
The water waves system is then given by
\begin{equation}\label{euler}
\begin{cases}
\D_t \phi + \dfrac{1}{2} |\nabla_{x, y}\phi|^2 + P + gy = 0 \qquad \text{in } \Omega, \\
\D_t \eta = \D_y \phi - \nabla_x \eta \cdot \nabla_x \phi \qquad \text{on } \Sigma, \\
P = 0 \qquad \text{on } \Sigma, \\
\D_\nu \phi = 0 \qquad \text{on } \Gamma,
\end{cases}
\end{equation}
where $g > 0$ is acceleration due to gravity, $\nu$ is the normal to $\Gamma$, and $P$ is the pressure, recoverable from the other unknowns by solving an elliptic equation. Here the first equation is the Euler equation in the presence of gravity, the second is the kinematic condition ensuring fluid particles at the interface remain at the interface, the third indicates no surface tension, and the fourth indicates a solid bottom.

A substantial literature regarding the well-posedness of the system (\ref{euler}) has been produced. We refer the reader to \cite{alazard2014cauchy}, \cite{alazard2014strichartz}, \cite{lannes2013water} for a more complete history and references. In our direction, well-posedness of (\ref{euler}) in Sobolev spaces was first established by Wu \cite{wu1997well}, \cite{wu1999well}. This well-posedness result was improved by Alazard-Burq-Zuily to a lower regularity at the threshold of a Lipschitz velocity field using only energy estimates \cite{alazard2014cauchy}. This was further sharpened to velocity fields with only a BMO derivative by Hunter-Ifrim-Tataru \cite{hunter2016two}.

\
 
It has been known that by taking advantage of an equation's dispersive properties, one can lower its well-posedness regularity threshold below that which is attainable by energy estimates alone. This was first realized in the context of low regularity Strichartz estimates for the nonlinear wave equation, in the works of Bahouri-Chemin \cite{bahouri1999equations}, \cite{bahouri1999equations2}, Tataru \cite{tataru2000strichartz}, \cite{tataru2001strichartz}, \cite{tataru2002strichartz}, Klainerman-Rodnianski \cite{klainerman2003improved}, and Smith-Tataru \cite{smith2005sharp}. Strichartz estimates have been similarly been studied for the water waves equations with surface tension; see \cite{christianson2010strichartz}, \cite{alazard2011strichartz}, \cite{de2015paradifferential}, \cite{de2016strichartz}, \cite{nguyen2017sharp}.

This low regularity Strichartz paradigm was first applied toward gravity water waves by Alazard-Burq-Zuily in \cite{alazard2014strichartz}. The argument proceeds by first establishing a paradifferential formulation of the water waves system (described precisely in Section \ref{paradiffred})
\begin{equation}\label{paralinear}
(\D_t + T_V \cdot \nabla + iT_\gamma)u = f
\end{equation}
where $T_V$ denotes the low-high frequency paraproduct with the vector field $V$, and $T_\gamma$ denotes the low-high paradifferential operator with symbol $\gamma$. Here, $\gamma$ is a real symbol of order $\half$, thus making explicit the dispersive character of the equations. Then by proving a Strichartz estimate for the paradifferential equation (\ref{paralinear}), one can obtain \emph{a priori} estimates for the full nonlinear system. However, to be useful toward the low regularity well-posedness theory, this Strichartz estimate must be proven assuming only a correspondingly low regularity of the coefficient vector field $V$ and symbol $\gamma$. This difficulty is addressed in \cite{alazard2014strichartz} by truncating these coefficients to frequencies below $\lambda^\delta$, where $\delta = 2/3 < 1$ is an optimized constant, essentially regularizing $V$ and $\gamma$.

However, this truncation approach comes at the cost of derivative losses in the Strichartz estimates and corresponding losses in the well-posedness threshold. In fact, it is likely that using only the regularity of the coefficients $V$ and $\gamma$ in (\ref{paralinear}), it is impossible to prove Strichartz estimates without loss. For instance, in the context of the wave equation with Lipschitz metric, counterexamples to sharp Strichartz estimates were provided by Smith-Sogge \cite{smith1994strichartz} and Smith-Tataru \cite{smith2002sharp}. Likewise, the former discusses counterexamples to sharp $L^p$ estimates for eigenfunctions of elliptic operators with Lipschitz coefficients.

Thus, to further improve the well-posedness threshold (at least, through the low regularity Strichartz paradigm), one needs to invoke the additional structure of $V$ and $\gamma$ as solutions to the water waves system. This is partially achieved for general dimensions in \cite{ai2017low}, improving upon the derivative losses. The goal of the present article is to entirely remove the derivative loss for the Strichartz estimates, implying the largest improvement to the well-posedness threshold that can be attained using the low regularity Strichartz paradigm. Because the argument is quite delicate, we consider here only the case $d = 1$ of one surface dimension, where the equations are somewhat simpler.

\subsection{Zakharov-Craig-Sulem formulation}

We first reduce the free boundary problem (\ref{euler}) to a system of equations on the free surface. Following Zakharov \cite{zakharov1968stability} and Craig-Sulem \cite{craig1993numerical}, we have unknowns $(\eta, \psi)$ where $\eta$ is the vertical position of the fluid surface as before, and
$$\psi(t, x) = \phi(t, x, \eta(t, x))$$
is the velocity potential $\phi$ restricted to the surface. Then (henceforth, $\nabla = \nabla_x$):
\begin{equation}\label{zak}
\begin{cases}
\D_t \eta - G(\eta) \psi = 0 \\
\displaystyle \D_t \psi + g\eta + \half |\nabla \psi|^2 - \half \frac{(\nabla \eta \cdot \nabla \psi + G(\eta) \psi)^2}{1 + |\nabla \eta|^2} = 0.
\end{cases}
\end{equation}

Here, $G(\eta)$ is the Dirichlet to Neumann map with boundary $\eta$:
$$(G(\eta)\psi)(t, x) = \sqrt{1 + |\nabla \eta|^2} (\D_n \phi) |_{y = \eta(t, x)}.$$
See \cite{alazard2011water}, \cite{alazard2014cauchy} for a precise construction of $G(\eta)$ in a domain with a general bottom. In addition, it was shown in \cite{alazard2013water} that if a solution $(\eta, \psi)$ of (\ref{zak}) belongs to $C^0([0, T]; H^{s + \half}(\R^d))$ for $T > 0$ and $s > \frac{d}{2} + \half$, then one can define a velocity potential $\phi$ and a pressure $P$ satisfying the Eulerian system (\ref{euler}). We will consider the case $s > \frac{d}{2} + \half$ throughout.

\subsection{Paradifferential reduction}\label{paradiffred}

We establish Strichartz estimates in terms of the paradifferential reduction of the water waves system developed in \cite{alazard2014cauchy}, \cite{alazard2014strichartz}, which expresses the unknowns of the water waves system as the solution to an explicit dispersive equation. We recall the reduction in this subsection.

\

We denote the horizontal and vertical components of the velocity field restricted to the surface $\eta$ by
$$V(t, x) = (\nabla \phi)|_{y = \eta(t, x)}, \quad B(t, x) = (\D_y \phi)|_{y = \eta(t, x)}$$
and the good unknown of Alinhac by
$$U_s = \langle D_x \rangle^s V + T_{\nabla \eta} \langle D_x \rangle^s B.$$
The traces $(V, B)$ can be expressed directly in terms of the unknowns $(\eta, \psi)$ of the Zakharov-Craig-Sulem formulation,
\begin{equation}\label{BVformulas}
B = \frac{\nabla \eta \cdot \nabla \psi + G(\eta) \psi}{1 + |\nabla \eta|^2}, \quad \nabla \psi = V + B \nabla \eta.
\end{equation}

Denote the principal symbol of the Dirichlet to Neumann map
$$\Lambda(t, x, \xi) = \sqrt{(1 + |\nabla \eta|^2)|\xi|^2 - (\nabla \eta \cdot \xi)^2},$$
and the Taylor coefficient 
$$a(t, x) = -(\D_y P)|_{y = \eta(t, x)}.$$
We will assume the now-classical Taylor sign condition, $a(t, \cdot) \geq a_{min} > 0$. This expresses the fact that the pressure increases going from the air to the fluid domain. It is satisfied in the case of infinite bottom \cite{wu1999well} or small perturbations of flat bottoms \cite{lannes2005well}. In the case of one surface dimension, also see \cite{hunter2016two} and \cite{harrop2017finite} for alternative proofs of this fact with infinite bottom and flat bottom, respectively. The water waves system is known to be ill-posed when the condition is not satisfied \cite{ebin1987equations}.

The reduction to a dispersive equation involves a symmetrization and complexification of the Eulerian system. As a result, it is convenient to define the symmetrized symbol
$$\gamma = \sqrt{a \Lambda}$$
and the complex unknown
\begin{equation}\label{complexu}
u = \langle D_x \rangle^{-s} (U_s - iT_{\sqrt{a/\Lambda}} \langle D_x \rangle^s \nabla \eta).
\end{equation}

Lastly, let $\FF$ denote a non-decreasing positive function depending on $h, g$, and $a_{min}$ and define
\begin{align*}
M_s(t) &= \|(\eta, \psi)(t)\|_{H^{s + \half}(\R^d)} + \|(V, B)(t)\|_{H^s(\R^d)} \\
Z_r(t) &= 1 + \|\eta(t)\|_{W^{r + \half, \infty}(\R^d)} + \|(V, B)(t)\|_{W^{r, \infty}(\R^d)}.
\end{align*}
Then we have the following paradifferential equation for $u$ (see Appendix \ref{paracalcnotation} for the definitions and notation of the paradifferential calculus), which holds in general dimension:
\begin{prop}[{\cite[Corollary 2.7]{alazard2014strichartz}}]\label{prop:paralinearization}
Let 
$$0 < T \leq 1,\quad s > \frac{d}{2} + \frac{3}{4}, \quad r > 1.$$ 
Consider a smooth solution $(\eta, \psi) \in C^1([0, T]; H^{s + \half}(\R^d))$ to (\ref{zak}) satisfying, uniformly on $t \in [0, T]$, (\ref{width}) and the Taylor sign condition $a(t, \cdot) \geq a_{min} > 0$. 

Then $u$ given by (\ref{complexu}) satisfies (\ref{paralinear}),
\begin{equation*}
\D_t u + T_V \cdot \nabla u + i T_\gamma u = f,
\end{equation*}
with
$$\|f(t)\|_{H^s(\R^d)} \leq \FF(M_s(t))Z_r(t).$$
\end{prop}

\begin{rem}
The proposition was stated for $s > \frac{d}{2} + \frac{3}{4}$ in \cite{alazard2014strichartz}, but using a more careful elliptic analysis, this can be reduced to $s > \frac{d}{2} + \half.$
\end{rem}

\subsection{Strichartz estimates and local well-posedness}\label{secstrichintro}

In this section we state our main result, a Strichartz estimate for arbitrary smooth solutions $u$ to the equation (\ref{paralinear}). Write, denoting $I = [0, T]$,
\begin{align*}
\MM_s(T) &= \|(\eta, \psi)\|_{L^\infty(I;H^{s + \half}(\R))} + \|(V, B)\|_{L^\infty(I; H^{s}(\R))} \\
\ZZ_r(T) &= 1 + \|\eta\|_{L^2(I;W^{r + \half, \infty}(\R))} + \|(V, B)\|_{L^2(I;W^{r, \infty}(\R))}.
\end{align*}

\begin{thm} \label{p1}
Let $d = 1$ and
$$0 < T \leq 1,\quad s > \frac{d}{2} + \frac{7}{8}, \quad r > 1.$$ 
Consider a smooth solution $(\eta, \psi) \in C^1([0, T]; H^{s + \half}(\R))$ to (\ref{zak}) satisfying, uniformly on $t \in [0, T]$, (\ref{width}) and the Taylor sign condition $a(t, \cdot) \geq a_{min} > 0$. Further, let
$$\sigma \in \R, \quad \sigma' = \sigma + \frac{1}{8}, \quad \eps > 0.$$

Then a smooth solution $u$ on $I$ to (\ref{paralinear}),
$$\D_t u + T_V \cdot \nabla u + i T_\gamma u = f$$
with coefficients $V, \gamma$ obtained from $(\eta, \psi)$, satisfies
$$\|u\|_{L^2(I; W^{\sigma' - \frac{d}{2} - \eps, \infty}(\R))} \leq \FF(\MM_s(T) + \ZZ_r(T))( \|f\|_{L^1(I; H^{\sigma}(\R))} + \|u\|_{L^\infty(I;H^{\sigma}(\R))}).$$
\end{thm}

\begin{rem}
We make the following remarks regarding the Strichartz gain $\mu$:
\begin{itemize}
\item For comparison, solutions to the constant coefficient linearized equation satisfy
$$\|e^{-it|D_x|^\half} u_0\|_{L^4([0, 1]; L^\infty(\R))} \lesssim \|u_0\|_{H^{\frac{3}{8}}(\R)},$$
which also exhibits a gain of $1/8$ derivatives over Sobolev embedding. Note that Theorem \ref{p1} does not achieve the $L^4$ in time; this remains an open question for now.
\item Such an estimate was first established in \cite{alazard2014strichartz} with a loss of $1/12$ derivatives relative to the constant coefficient case, and in \cite {ai2017low} with a loss of $1/40$ derivatives.
\end{itemize}
\end{rem}

Combining Theorem \ref{p1} with Proposition \ref{prop:paralinearization}, it is relatively straightforward to obtain \emph{a priori} H\"older and Sobolev estimates; see \cite[Chapter 3]{alazard2014strichartz} for details. Combined with contraction and limiting arguments likewise detailed in \cite[Chapter 3]{alazard2014strichartz}, we obtain the following local well-posedness result with Strichartz estimates:

\begin{thm}\label{lwp}
Let $d = 1$ and
$$s > \frac{d}{2} + \frac{7}{8}, \quad 1 < r < s + \frac{1}{8} - \frac{d}{2}.$$
Consider initial data $(\eta_0, \psi_0) \in H^{s + \half}(\R)$ satisfying
\begin{enumerate}[i)]
\item $(V_0, B_0) \in H^s(\R)$,
\item the positive depth condition
$$\{(x, y) \in \R \times \R: \eta_0(x) - h < y < \eta_0(x) \} \subseteq \OO,$$
\item the Taylor sign condition
$$a_0(x) \geq a_{min} > 0.$$
\end{enumerate}

Then there exists $T > 0$ such that the system (\ref{zak}) with initial data $(\eta_0, \psi_0)$ has a unique solution $(\eta, \psi) \in C([0, T]; H^{s + \half}(\R))$ such that
\begin{enumerate}[i)]
\item $(\eta, \psi) \in L^2([0, T]; W^{r + \half, \infty}(\R))$,
\item $(V, B) \in C([0, T]; H^s(\R)) \cap L^2([0, T]; W^{r, \infty}(\R))$,
\item the positive depth condition (\ref{width}) holds on $t \in [0, T]$ with $h/2$ in place of $h$,
\item the Taylor sign condition $a(t, x) \geq a_{min}/2$ holds on $t \in [0, T]$.
\end{enumerate}
\end{thm}

Since the contraction and limiting arguments are essentially unchanged from those in \cite[Chapter 3]{alazard2014strichartz}, the remainder of the article discusses the proof of Theorem \ref{p1}.

\subsection{Overview of the argument}

The strategy we use to prove Theorem \ref{p1} relies on a parametrix construction that approximates solutions to (\ref{paralinear}) as a square-summable superposition of discrete wave packets. The Strichartz estimates then follow from a geometric observation capturing the dispersion of the packets, combined with a counting argument. This is in contrast with the Hadamard parametrix used in \cite{alazard2014strichartz}, or the parametrix based on the FBI transform used in \cite{ai2017low}, both of which use dispersive estimates combined with a classical $TT^*$ argument to obtain the Strichartz estimates. The strategy we follow here more closely resembles the parametrix construction and counting argument used for the wave equation in \cite{smith2005sharp}. Also in contrast with both \cite{alazard2014strichartz} and \cite{ai2017low}, we incorporate the transport term of (\ref{paralinear}) directly in the parametrix construction, rather than changing variables to Lagrangian coordinates.

To proceed without derivative losses, we can only truncate the coefficients of (\ref{paralinear}) to frequencies below $\lambda$, in further contrast with \cite{alazard2014strichartz}, \cite{ai2017low} which truncate the coefficients of (\ref{paralinear}) to frequencies $\lambda^\delta$, with $\delta$ a constant below 1. As a result, our wave packet parametrix construction and subsequent geometric analysis must be performed at lower regularity. This presents difficulties on two main fronts.

First, if we observe only the regularity of the symbol of (\ref{paralinear}), the associated Hamilton characteristics along which wave packets travel do not have sufficient regularity to allow a meaningful geometric analysis of the dispersion on the full time interval. To overcome this, we seek an integration structure for the symbol along the Hamilton flow, to show that the characteristics are in fact bilipschitz. We remark that this integration structure is more complex than that in \cite{ai2017low}, as it involves the full symbol $V\xi + \gamma$ rather than $V$ alone.

Second, the natural wave packet scale at frequency $\lambda$ for a dispersive equation of order $1/2$ is
$$\delta x \approx \lambda^{-\frac{3}{4}}, \quad \delta \xi \approx \lambda^{\frac{3}{4}}, \quad \delta t \approx 1.$$
However, the coefficients of (\ref{paralinear}), when truncated only to frequencies below $\lambda$, are approximately constant only on the scale $\delta x \approx \lambda^{-1}$, so we cannot expect a generic bump function localized on the wave packet scale to be a useful approximate solution. We address this by using an exact eikonal phase function for our wave packets, at the cost of losing the exact localization in frequency. In turn, this interferes with the square summability of the wave packets. To gain the additional frequency localization needed for the summability, we use the fact that, along wave packets, the coefficients of (\ref{paralinear}) enjoy extra regularity in the form of dispersive local smoothing estimates on the wave packet scale $\delta x \approx \lambda^{-\frac{3}{4}}$.

\

We outline the remainder of the article. In Section \ref{notation}, we record the notation that we use throughout the article. In Section \ref{sec:freqloc}, we perform a standard reduction of the Strichartz estimate to a dyadic frequency localized form. 

In Section \ref{sec:locsmooth}, we prove local smoothing estimates. In Sections \ref{sec:dirichletprob} and \ref{sec:localdnpar}, we prove local versions of the elliptic estimates for the Dirichlet problem with rough boundary, and of the paralinearization of the Dirichlet to Neumann map, respectively. In Section \ref{sec:taylorlocal}, we likewise establish local Sobolev estimates on the Taylor coefficient $a$. 

In Sections \ref{COVsec} and \ref{sec:regham}, we integrate the symbol of our evolution equation along its Hamilton flow, and use this to establish improved the regularity properties for the flow. In Section \ref{sec:parametrix}, we construct the wave packet parametrix. Finally, in Section \ref{sec:strichartz}, we prove Strichartz estimates for the parametrix, before establishing the estimate for the exact solution.

\subsection{Acknowledgments}

The author would like to thank his advisor, Daniel Tataru, for introducing him to this research area and for many helpful discussions.

\section{Notation}\label{notation}

We record the notation and setting that we will use throughout the article, excluding the appendix.

\subsection{Regularity indices} Let
$$s > \frac{d}{2} + \frac{7}{8}, \quad 1 < r < s + \half - \frac{d}{2}, \quad r \leq \frac{3}{2}$$
where $d = 1$, though we often preserve the notation of general $d$ to provide context. Note that some propositions will obviously hold without the upper bounds on $r$. For any regularity index $m \in \R$, let
$$m' = m + \frac{1}{8}.$$
Also let $0 < \eps < 1 - r$ denote a small number, and
$$m+ = m + \eps, \quad m- = m - \eps.$$
We will allow $m+, m-$, and $\eps$ to vary from line to line.

\subsection{Spaces for parabolic estimates} Define the following spaces for a vertical spatial interval $J \subseteq \R$:
\begin{align*}
X^\sigma(J) &= C_z^0(J;H^{\sigma}) \cap L_z^2(J;H^{\sigma + \half})  \\
Y^\sigma(J) &= L_z^1(J;H^{\sigma}) + L_z^2(J;H^{\sigma - \half}) \\
U^\sigma(J) &= C_z^0(J;C_*^{\sigma}) \cap L_z^2(J;C_*^{\sigma + \half}). 
\end{align*}

\subsection{Quantities from the water waves system} We write 
$$\nabla = \nabla_x = \D_x, \quad \Delta = \Delta_x = \D_x^2, \quad \langle D \rangle = \langle D_x \rangle, \quad | D| = | D_x |.$$
Denote $I = [0, T]$ and let $(\eta, \psi)  \in C^1(I; H^{s + \half}(\R))$ denote a smooth solution to (\ref{zak}) satisfying, uniformly on $t \in I$, the positive depth condition (\ref{width}). From $(\eta, \psi)$ we can extract the harmonic velocity potential $\phi(t, x, y)$ satisfying
$$\Delta_{x, y}\phi = 0, \qquad \phi|_{y = \eta(t, x)} = \psi,$$
the velocity field $v(t, x, y)$,
$$v = \nabla_{x, y} \phi,$$
the pressure $P(t, x, y)$,
$$-P = \D_t \phi + \half |\nabla_{x, y} \phi|^2 + gy,$$
the horizontal and vertical components $(V, B)(t, x)$ of the velocity field restricted to the surface $\eta$,
$$V = (\nabla \phi)|_{y = \eta(t, x)}, \quad B = (\D_y \phi)|_{y = \eta(t, x)},$$
and the Taylor coefficient $a(t, x)$,
$$a = -(\D_y P)|_{y = \eta(t, x)}.$$
We assume the Taylor sign condition $a \geq a_{min} > 0$ as discussed in the introduction. We remark that throughout, we use the estimates on $a$ provided in Proposition \ref{taylorbd}. Lastly, we obtain the principal symbol $\Lambda(t, x, \xi)$ of the Dirichlet to Neumann map associated to $\eta(t, \cdot)$,
$$\Lambda = \sqrt{(1 + |\nabla \eta|^2)|\xi|^2 - (\nabla \eta \cdot \xi)^2},$$
and define the symmetrized symbol
$$\gamma = \sqrt{a\Lambda}.$$
Observe that in the case $d = 1$, we have
$$\Lambda = |\xi|, \quad \gamma = \sqrt{a|\xi|}.$$
Lastly, for brevity we let $L$ denote the vector field
$$L = \D_t + V \cdot \nabla.$$

\subsection{Frequency decomposition} Throughout, we denote frequencies by $\lambda, \mu, \kappa \in 2^\N$. We will fix $\lambda$ by the end of Section \ref{sec:freqloc}, and use $\kappa$ to denote a general frequency parameter. Throughout, we will assume $\mu$ satisfies $\lambda^{\frac{3}{4}} \ll \mu \leq c\lambda$, where the small absolute constant $c$ is discussed further below.

We recall the standard Littlewood-Paley decomposition. Fix $\varphi(\xi) \in C_0^\infty(\R^d)$ with support in $\{|\xi| \leq 2\}$ such that $\varphi \equiv 1$ on $\{|\xi| \leq 1\}$. Then define for $u = u(x)$,
$$\widehat{S_\kappa u}(\xi) := (\varphi(\xi/\kappa) - \varphi(2\xi/\kappa))\widehat{u}(\xi) =: \psi(\xi/\kappa) \widehat{u} =: \psi_\kappa (\xi) \widehat{u}$$
which has support $\{|\xi| \in [\kappa/2, 2\kappa]\}$. Also allow $S_{< \kappa}$, etc., in the natural way, and denote $S_{0} = S_{<1}$. Lastly, let $\tilde{S}_\kappa = \sum_{\kappa/4 \leq \tilde{\kappa} \leq 4\kappa} S_{\tilde{\kappa}}$ denote a frequency projection with widened support so that $\tilde{S}_\kappa S_\kappa = S_\kappa$.

For a small absolute constant $c > 0$, write
$$A \ll B, \quad A \approx B$$
when, respectively,
$$A \leq cB, \quad (1 - c)A \leq B \leq A/(1 - c).$$
Further, we will set an additional absolute constant
$$0 < c_1 \ll c \ll 1,$$
where $c_1\lambda$ will be the frequency truncation for the paradifferential equation on which we prove Strichartz estimates. We denote
$$a_\kappa = S_{\leq c_1 \kappa} a, \quad V_\kappa = S_{\leq c_1 \kappa} V, \quad \eta_\kappa = S_{\leq c_1 \kappa} \eta, \quad \gamma_\kappa = S_{\leq c_1 \kappa} \gamma.$$

\subsection{The Hamiltonian and wave packets} 
We will use the Hamiltonian $H(t, x, \xi)$ given by
$$H = V_\lambda \xi +  \sqrt{a_\lambda |\xi|}.$$
Note we omit notating $\lambda$ on $H$ for brevity since we will fix $\lambda$ by the end of Section \ref{sec:freqloc}.

Recall the Hamilton equations associated to $H$,
\begin{equation}\label{hamilton}\begin{cases}
\dot{x}(t) = H_\xi(t, x(t), \xi(t))\\
\dot{\xi}(t) = -H_x(t, x(t), \xi(t))\\
(x(s), \xi(s)) = (x, \xi).
\end{cases}\end{equation}
We also denote the solution $(x(t), \xi(t))$ to (\ref{hamilton}), at time $t \in I$ with initial data $(x, \xi)$ at time $s \in I$, variously by 
$$(x^t_s(x, \xi), \xi^t_s(x, \xi)) = (x^t(x, \xi), \xi^t(x, \xi)) = (x^t_s, \xi^t_s) = (x^t, \xi^t).$$

We consider the discrete set of phase space indices,
$$\TT = \{T = (x, \xi) \in \lambda^{-\frac{3}{4}}\Z \times \lambda^{\frac{3}{4}}\Z : |\xi| \approx \lambda \},$$
and in particular often write $T = (x, \xi)$ for brevity. 

We let $s_0 \in I$ denote the time chosen as in Proposition \ref{bilipprop}. Once $s_0$ is introduced, we write
$$(x^t, \xi^t) = (x^t_{s_0}, \xi^t_{s_0})$$
unless otherwise indicated. Further, associated to $T = (x, \xi) \in \TT$, we let $s_T \in I$ denote the time chosen as in Lemma \ref{bilipproplocal}.

\subsection{Local smoothing}\label{localnotation}

We define weights to exhibit local smoothing estimates. Fix a Schwartz weight $w \in C^\infty(\R)$ with frequency support $\{|\xi| \leq 1\}$ and $w \geq 1$ on $\{|x| \leq 1\}$. We denote a scaled translate of $w$ by
$$w_{x_0, \kappa}(x) = w(\kappa^{\frac{3}{4}}(x - x_0))$$
for $x_0 \in \R$. We will use throughout that $w_{x_0, \kappa}$ has frequency support $\{|\xi| \leq \kappa^{\frac{3}{4}}\}$.

It will also be convenient to have compactly supported weights. Fix a smooth bump function $\chi \in C_0^\infty(\R)$ supported on $[-1, 1]$, with $\chi \equiv 1$ on $[-\half, \half]$. Also let $\tilde{\chi}$ denote a widened bump function supported on $[-2, 2]$, with $\tilde{\chi} \equiv 1$ on $[-1, 1]$. As with $w$, we also write $\chi_{x_0, \kappa}, \tilde{\chi}_{x_0, \kappa}$ for the scaled translates of $\chi, \tilde{\chi}$.

We define the following local seminorm to measure local smoothing:
$$\|f\|_{LS^\sigma_{x_0, \lambda}} = \sum_{\kappa \leq c\lambda} \|w_{x_0, \kappa} S_\kappa f\|_{H^\sigma}.$$

We will also define a more technical local seminorm which we use to measure local smoothing on products. Let $\xi_0 \in \R$ with $|\xi_0| \in [\lambda/2, 2\lambda]$. First we construct a symmetric $\lambda$-frequency projection $S_{\xi_0, \lambda, \mu}$ with a $c\mu$-width gap at $\xi_0$, as follows: Let 
$$p(\xi) = 1_{[0, \infty)}(\xi + \lambda)\psi_\lambda(\xi + \lambda)(1 - \chi(\mu^{-1} \xi/c))$$
and define
$$S_{\xi_0, \lambda, \mu} = p(D - \xi_0) + p(-D + \xi_0).$$
Throughout, to simplify exposition, we will abuse notation and use $S_{\xi_0, \lambda, \mu}$ to denote the same construction with $\psi_{\lambda'}$ in place of $\psi_\lambda$, where $\lambda$ and $\lambda'$ have bounded ratio.

Then we define the following local seminorm for $\sigma \in \R$, which collects local measurements on low, high, and balanced frequencies, respectively by row:
\begin{align*}
\|f\|_{LS^\sigma_{x_0, \xi_0, \lambda, \mu}} = \ &\sum_{\kappa \in [c\mu, c\lambda]} ( \kappa^{\frac{3}{4}} \mu^{-1} \|S_\kappa f\|_{H^\sigma} +\|w_{x_0, \kappa}S_\kappa f\|_{H^\sigma}) \\
&+ \lambda^{\frac{3}{4}} \mu^{-1}\|S_{\geq c\lambda} f \|_{H^\sigma} + \sum_{\kappa \geq \lambda/c} \|w_{x_0, \lambda} S_\kappa f\|_{H^\sigma} \\
&+ \|w_{x_0, \lambda}S_{\xi_0, \lambda, \mu}f\|_{H^\sigma}.
\end{align*}

\subsection{A priori quantities}

Throughout, $\FF$ will denote a non-decreasing positive function which may change from line to line, and which may depend on the depth $h > 0$, the constant $a_{min} > 0$ of the Taylor sign condition, and the gravitational constant $g > 0$. For brevity, we denote
\begin{align*}
M(t) &= M_s(t) = \|(\eta, \psi)(t)\|_{H^{s + \half}} + \|(V, B)(t)\|_{H^s}, \\
Z(t) &= M_r(t) = 1 + \|\eta(t)\|_{W^{r + \half, \infty}} + \|(V, B)(t)\|_{W^{r, \infty}}, \\
\LL(t, x_0, \xi_0, \lambda) &= \LL_s(t, x_0, \xi_0, \lambda) \\
&= \|(\eta, \psi)(t)\|_{LS^{s' + \half}_{x_0, \lambda}} + \|(V, B)(t)\|_{LS^{s'}_{x_0, \lambda}} \\
&\quad + \max_{\lambda^{\frac{3}{4}} \ll \mu \leq c\lambda} \max_{0 \leq \sigma \leq s} \lambda^{-\sigma} \mu^{s'}(\|(\eta, \psi)(t)\|_{LS^{\sigma  + \half}_{x_0, \xi_0, \lambda, \mu}} + \|(V, B)(t)\|_{LS^{\sigma}_{x_0, \xi_0, \lambda, \mu}}).
\end{align*}
As an abuse of notation, also write for brevity,
$$\FF(T) = \FF(\|(\eta, \psi)\|_{L^\infty(I;H^{s + \half})}, \|(V, B)\|_{L^\infty(I; H^{s})}, \|\eta\|_{L^2(I;W^{r + \half, \infty})}, \|(V, B)\|_{L^2(I;W^{r, \infty})}).$$
Observe that 
$$\|\FF(M(t))\|_{L_t^\infty(I)} \leq \FF(T), \quad \|Z(t)\|_{L_t^2(I)} \leq \FF(T).$$

\section{Frequency Localization}\label{sec:freqloc}

In this section we reduce Theorem \ref{p1} to a frequency localized form. 

\subsection{Dyadic decomposition}

First we reduce Theorem \ref{p1} to the corresponding dyadic frequency estimates:

\begin{prop} \label{prop:p2}
Consider a smooth solution $u_\lambda$ to (\ref{paralinear}) on $I$ where $u = u_\lambda(t, \cdot)$ and $f = f_\lambda$ have frequency support $\{|\xi| \in [\lambda/2, 2\lambda]\}$. Then
\begin{align*}
\|u_\lambda\|_{L^2(I; W^{\frac{1}{8} - \frac{d}{2}-, \infty})} \leq  \FF(T)( \|f_\lambda\|_{L^1(I;L^2)} + \|u_\lambda \|_{L^\infty(I;L^2)}).
\end{align*}
\end{prop}

\begin{proof}[Proof of Theorem \ref{p1}] Given $u$ solving (\ref{paralinear}), $u_\lambda = S_\lambda u$ solves (\ref{paralinear}) with inhomogeneity
$$S_\lambda f + [T_V \cdot \nabla, S_\lambda]u + i[T_\gamma, S_\lambda]u.$$
Note that this has frequency support $\{|\xi| \in [\lambda/4, 4\lambda]\}$ (strictly speaking, we should modify Proposition \ref{prop:p2} to address this wider support, but we neglect this detail throughout for simplicity). By the paradifferential commutator estimate (\ref{sobolevcommutator}),
\begin{align*}
\|[T_V \cdot \nabla, S_\lambda]u\|_{L^1(I;H^{\sigma})} &\lesssim \|V\|_{L^1(I;W^{1,\infty})} \|\tilde{S}_\lambda u\|_{L^\infty(I;H^{\sigma})}.
\end{align*}
Similarly, by (\ref{sobolevcommutator}) and the estimates on the Taylor coefficient as in Corollary \ref{gammabd},
$$\|[T_\gamma, S_\lambda]u\|_{H^{\sigma}} \lesssim M_\half^\half(\gamma) \|\tilde{S}_\lambda u\|_{H^{\sigma}} \leq \FF(M(t))Z(t) \|\tilde{S}_\lambda u\|_{H^{\sigma}}$$
and hence
$$\|[T_\gamma, S_\lambda]u\|_{L^1(I;H^{\sigma})} \leq \FF(T)\|\tilde{S}_\lambda u\|_{L^\infty(I;H^{\sigma})}.$$
We then decompose $u$ into frequency pieces $u_\lambda$ on which we can apply Proposition \ref{prop:p2} (multiplied by $\lambda^{\sigma}$, using the frequency localization):
\begin{align*}
\|u\|_{L^2(I; W^{\sigma' - \frac{d}{2}-, \infty})} &\leq \sum_{\lambda = 0} \|u_\lambda \|_{L^2(I; W^{\sigma' - \frac{d}{2}-, \infty})} \\
&\leq \FF(T)\sum_{\lambda = 0} \lambda^{-\eps}( \|S_\lambda f\|_{L^1(I; H^{\sigma})} + \|\tilde{S}_\lambda u \|_{L^\infty(I;H^{\sigma})}) \\
&\leq \FF(T)( \| f\|_{L^1(I; H^{\sigma})} + \| u \|_{L^\infty(I;H^{\sigma})})
\end{align*}
as desired.
\end{proof}

\subsection{Symbol truncation}

Next, we reduce Proposition \ref{prop:p2} to an estimate with frequency truncated symbols.

\begin{prop} \label{p3}
Consider a smooth solution $u_\lambda$ to
\begin{equation}\label{eqn:pretruncparalinearized}
(\D_t + T_{V_\lambda} \cdot \nabla + iT_{\gamma_\lambda}) u_\lambda = f_\lambda
\end{equation}
on $I$ where $u_\lambda(t, \cdot)$ and $f_\lambda$ have frequency support $\{|\xi| \in [\lambda/2, 2\lambda]\}$. Then
\begin{align*}
\|u_\lambda\|_{L^2(I; W^{\frac{1}{8} - \frac{d}{2}-, \infty})} \leq  \FF(T)(\|f_\lambda\|_{L^1(I; 
L^2)} + \|u_\lambda \|_{L^\infty(I;L^2)}).
\end{align*}
\end{prop}

\begin{proof}[Proof of Proposition \ref{prop:p2}]
This follows from Proposition \ref{p3} by including $(T_V - T_{V_\lambda}) \cdot \nabla u_\lambda$ and $(T_\gamma - T_{\gamma_\lambda})u_\lambda$ as components of the inhomogeneous term, along with the following estimates, proven below:
\begin{align*}
\|(T_V - T_{V_\lambda}) \cdot \nabla u_\lambda \|_{L^1(I; L^2)} &\leq \FF(T)\|u_\lambda \|_{L^\infty(I;L^2)}   \\
\|(T_\gamma - T_{\gamma_\lambda})u_\lambda\|_{L^1(I; L^2)} &\leq \FF(T)\|u_\lambda \|_{L^\infty(I;L^2)}. 
\end{align*}

For the first estimate, observe that 
$$(T_V - T_{V_\lambda}) \cdot \nabla u_\lambda = (S_{c_1\lambda < \cdot \leq \lambda/8}V) \cdot \nabla u_\lambda$$
which satisfies
\begin{align*}
\|(S_{c_1\lambda < \cdot \leq \lambda/8}V) \cdot \nabla u_\lambda\|_{L^1(I; L^2)} &\lesssim \|V\|_{L^1(I;W^{1, \infty})} \|u_\lambda\|_{L^\infty(I; L^2)}.
\end{align*}

The second estimate is similar, using that
$$T_\gamma - T_{\gamma_\lambda} = (T_{\sqrt{a}} - T_{S_{\leq c_1 \lambda} \sqrt{a}})|D|^\half$$
to obtain
\begin{align*}
\|(T_\gamma - T_{\gamma_\lambda}) u_\lambda \|_{L^1(I; L^2)} \lesssim \|\sqrt{a}\|_{L^1(I;W^{\half, \infty})} \|u_\lambda\|_{L^\infty(I; L^2)}.
\end{align*}
Using the estimates on the Taylor coefficient $a$ provided in Proposition \ref{taylorbd} (and the Moser estimate \ref{smoothholder}) yields the desired result.
\end{proof}

\subsection{Pseudodifferential symbol}

To study the Hamilton flow, it is convenient to replace the paradifferential symbol $T_{\gamma_\lambda}$ by the pseudodifferential symbol $\gamma_\lambda = \gamma_\lambda(t, x, \xi)$, and likewise the paraproduct $T_{V_\lambda}$ by $V_\lambda$. It is harmless to do so since $u_\lambda$ is frequency localized. 

Recall that since $d = 1$, $\gamma = \sqrt{a|D|}$. For technical reasons, we further replace $\gamma_\lambda$ by
$\sqrt{a_\lambda |D|}.$

\begin{prop}\label{redstrich}
Consider a smooth solution $u_\lambda$ to 
\begin{equation}\label{truncparalinearized}
(\D_t + V_\lambda \D_x + i \sqrt{a_\lambda |D|}) u_\lambda = f
\end{equation}
on $I$ where $u_\lambda(t, \cdot)$ has frequency support $\{|\xi| \in [\lambda/2, 2\lambda]\}$. Then
\begin{align*}
\|u_\lambda\|_{L^2(I; W^{\frac{1}{8} - \frac{d}{2} -, \infty})} \leq  \FF(T)(\|f_\lambda\|_{L^1(I;
L^2)} + \|u_\lambda \|_{L^\infty(I;L^2)}).
\end{align*}
\end{prop}

\begin{proof}[Proof of Proposition \ref{p3}]
We may apply Proposition \ref{redstrich} with inhomogeneity
$$f_\lambda + (T_{\gamma_\lambda} - \gamma_\lambda)u_\lambda + (T_{V_\lambda} - V_\lambda) \cdot \nabla u_\lambda + (\gamma_\lambda - \sqrt{a_\lambda |D|})u_\lambda.$$
Using for instance \cite[Proposition 2.7]{nguyen2015sharp} and the estimates on the Taylor coefficient as in Corollary \ref{gammabd},
$$\|(T_{\gamma_\lambda} - \gamma_\lambda)u_\lambda\|_{L^2} \lesssim M^\half_\half(\gamma)\|u_\lambda\|_{L^2} \leq \FF(M(t))Z(t) \|u_\lambda\|_{L^2}.$$
Integrating in time, we have
$$\|(T_{\gamma_\lambda} - \gamma_\lambda) u_\lambda\|_{L^1(I;L^2)} \leq \FF(T) \|u_\lambda\|_{L^\infty(I;L^2)}$$
which is controlled by the right hand side of the estimate in Proposition \ref{p3}.

Similarly,
\begin{align*}
\|(T_{V_\lambda} - V_\lambda) \cdot \nabla u_\lambda\|_{L^2} &\lesssim M_1^1(V \cdot \xi) \|u_\lambda\|_{L^2} \lesssim \|V\|_{W^{1, \infty}} \|u_\lambda\|_{L^2}.
\end{align*}
Integrating in time yields the desired estimate.

It remains to estimate $(\gamma_\lambda - \sqrt{a_\lambda |D|})u_\lambda$. First, we have
\begin{align*}
\|(\gamma - \gamma_\lambda) u_\lambda \|_{L^2} &\lesssim \|S_{> c_1\lambda}\sqrt{a}\|_{L^\infty} \|u_\lambda\|_{H^{\frac{1}{2}}} \\
&\leq \|\sqrt{a}\|_{W^{\half, \infty}} \|u_\lambda\|_{L^2}.
\end{align*}
Using the Taylor coefficient estimates of Proposition \ref{taylorbd} and integrating in time, this is controlled by the right hand side of the estimate in Proposition \ref{p3}. We may thus exchange $\gamma_\lambda$ for $\gamma$.

Lastly, it is easy to compute
$$\gamma - \sqrt{a_\lambda |D|} = \frac{a_{> c_1\lambda}}{\sqrt{a} + \sqrt{a_\lambda}} \sqrt{|D|}.$$
Thus, using the Taylor sign condition,
\begin{align*}
\|(\gamma - \sqrt{a_\lambda |D|})u_\lambda \|_{L^2} &\lesssim \frac{1}{\sqrt{a_{min}}} \|a_{> c_1\lambda}\|_{L^\infty} \|u_\lambda\|_{H^{\frac{1}{2}}} \\
&\lesssim \frac{1}{\sqrt{a_{min}}} \|a\|_{W^{\half, \infty}} \|u_\lambda\|_{L^2}.
\end{align*}
Using the usual Taylor coefficient estimates and integrating in time yields the desired result.

\end{proof}

For the remainder of the article, it remains to prove Proposition \ref{redstrich}. In particular, we fix the frequency $\lambda$ throughout.

\section{Local Smoothing Estimates}\label{sec:locsmooth}

In this section we establish local smoothing estimates for the paradifferential dispersive equation (\ref{paralinear}), and then show that these estimates are inherited by $(\eta, \psi, V, B)$. 

Observe that the order 1/2 dispersive term of (\ref{paralinear}) contributes velocity $\mu^{-\half} \ll 1$ to a frequency $\mu$ component of a solution $u$. On the other hand, the transport term of the equation contributes up to a unit scale velocity to all frequency components of $u$, dominating the dispersive velocity, $1 \gg \mu^{-\half}$. As a result, we only expect to see a local smoothing effect by following the transport vector field.

In fact, for our purposes we are interested in local smoothing along $\lambda$-frequency Hamilton characteristics, on packets of width $\delta x \approx \lambda^{-\frac{3}{4}}$. Since at frequency $\lambda$, 
$$\D_\xi^2 |\xi|^\half = \lambda^{-\frac{3}{2}},$$
we expect to see local smoothing effects for frequencies of $u$ with frequency separation greater than 
$$\lambda^{-\frac{3}{4}}/\lambda^{-\frac{3}{2}} = \lambda^{\frac{3}{4}}$$
from the frequency center of a packet.

\subsection{The Hamilton flow}

Recall we denote the Hamiltonian for (\ref{truncparalinearized}) by
$$H(t, x, \xi) = V_\lambda \xi + \sqrt{a_\lambda |\xi|}.$$
Thus, the Hamilton equations (\ref{hamilton}) associated to $H$ are given by
\begin{equation*}\begin{cases}
\dot{x}^t = V_\lambda + \half\sqrt{a_\lambda} |\xi^t|^{-\frac{3}{2}}\xi\\
\dot{\xi}^t = -\D_x V_\lambda \xi^t - \D_x \sqrt{a_\lambda} |\xi^t|^\half.
\end{cases}\end{equation*}

In the following lemma, we observe that $\xi^t$ preserves data satisfying $|\xi| \in [\lambda/2, 2\lambda]$:
\begin{lem}\label{freqpres}
If $|\xi| \in [\lambda/2, 2\lambda]$ and $|t - s|^{\half} \FF(T) \ll 1$, then
$$\|\xi^t_s(x, \xi)\|_{L^\infty_x} \approx |\xi|.$$
\end{lem}
\begin{proof}
We work with $|\xi^t|^{\half}$ instead of $\xi^t$ itself, due to its appearance in $H$. By (\ref{hamilton}),
$$\frac{d}{dt}|\xi^t|^{\half} = -\half|\xi^t|^{-\frac{3}{2}}\xi^t H_x,$$
so we have (using at the end the Taylor sign condition, and the frequency localization)
\begin{align*}
||\xi^t|^{\half} - |\xi|^{\half}| \leq \half \int_s^t  |\xi^\tau|^{-\half}|H_x| \, d\tau &\leq \half \int_s^t |\D_x V_\lambda |\xi^\tau |^{\half}| + |\D_x \sqrt{a_\lambda}| \, d\tau \\
&\leq \half \int_s^t \|\D_x V_\lambda\|_{L_x^\infty} |\xi^\tau |^{\half}+ \|\D_x \sqrt{a_\lambda}\|_{L_x^\infty} \, d\tau \\
&\lesssim \int_s^t \|V\|_{W^{1,\infty}} |\xi^\tau |^{\half} + \lambda^{\frac{1}{2}}\|a_\lambda\|_{W^{\half, \infty}} \, d\tau. \\
\end{align*}
Using the Taylor coefficient estimates of Proposition \ref{taylorbd},
\begin{equation}\label{xigronwall}
|\xi^t|^{\half} \leq |\xi|^{\half} + \int_s^t Z(t) |\xi^\tau |^{\half} + \lambda^{\half}\FF(M(t))Z(t) \, d\tau.
\end{equation}

Applying Gronwall, Cauchy-Schwarz, and the assumption $|\xi|^{\half} \lesssim \lambda^{\half}$,
$$|\xi^t|^{\half} \lesssim \lambda^{\half}(1 + |t - s|^{\half} \FF(T)) \exp (|t - s|^{\half} \FF(T)) \leq 2\lambda^{\half}.$$
Use this in (\ref{xigronwall}) to obtain
$$||\xi^t|^{\half} - |\xi|^{\half} | \leq \int_s^t Z(t) |\xi^\tau |^{\half} + \lambda^{\half} \FF(M(t))Z(t) \, d\tau \leq \lambda^{\half} |t - s|^{\half} \FF(T) \ll \lambda^{\half}.$$
\end{proof}

\begin{rem}
In the remainder of the article, we may assume $T$ is small enough so that the condition
$$|t - s|^{\half} \FF(T) \ll 1$$
of the lemma is satisfied for all $t, s \in I$ (otherwise, we can iterate the argument a number of times depending on $\FF(T)$). In particular, $\xi^t$ maintains separation from 0, so that the flow $(x^t, \xi^t)$ remains defined on $I$.
\end{rem}

\subsection{Estimates on the dispersive equation} In this subsection we establish local smoothing estimates for (\ref{paralinear}) along Hamilton characteristics. Throughout this section, we consider the case of $\xi_0 > 0$, but the corresponding results for $\xi_0 < 0$ are similar.

\begin{prop}\label{local1}
Consider a smooth solution $u$ to
\begin{equation}\label{constwater}
(\D_t + V_\lambda\D_x + i\sqrt{a_\lambda|D|})u = f
\end{equation}
on $I$. Let $(x^t, \xi^t)$ be a solution to (\ref{hamilton}) with initial data $(x_0, \xi_0)$ satisfying $\xi_0 \in [\lambda/2, 2\lambda]$.  Let $u(t, \cdot)$ have frequency support $[c\lambda/2, \xi^t - c\mu]$. Then 
$$\|w_{x^t, \lambda}u\|_{L^2(I;L^2)} \leq \mu^{-\half}\lambda^{\frac{3}{8}}\FF(T)(\|u\|_{L^\infty(I;L^2)} + \|f\|_{L^1(I;L^2)}).$$
The same holds for $u(t, \cdot)$ with frequency supports $[\xi^t + c\mu, 2\lambda/c]$ or $[-2\lambda/c, -c\lambda/2]$. 
\end{prop}

\begin{proof}
We use the positive commutator method. Let $v \in C^\infty(\R)$ satisfy $\sqrt{v'} = w$. It is also convenient to write 
$$p(D) = |D|^{\half}\tilde{S}_\lambda$$
so that we may exchange $|D|^{\half}$ for $p(D)$ in (\ref{constwater}). Lastly, for brevity, write
$$\tilde{w}(t, x) = w_{x^t, \lambda} = w(\lambda^{\frac{3}{4}}(x - x^t)), \quad \tilde{v}(t, x) = v(\lambda^{\frac{3}{4}}(x - x^t)), \quad \tilde{v}' = \D_x \tilde{v}.$$

Using (\ref{constwater}), we have
\begin{align}
\label{poscom}\frac{d}{dt}\langle \tilde{v} u, u \rangle = \ &\langle \tilde{v} \D_t u - \tilde{v}' \dot{x}^t u, u \rangle + \langle \tilde{v} u , \D_t u  \rangle \\
= \ &\langle \tilde{v} \D_t u + \D_x (V_\lambda \tilde{v} u) + ip(D)\sqrt{a_\lambda} \tilde{v} u, u \rangle - \langle \tilde{v}' \dot{x}^t u, u \rangle \nonumber \\
= \ &\langle \tilde{v} f, u \rangle + \langle i[p(D), \sqrt{a_\lambda}] \tilde{v} u, u\rangle  + \langle i\sqrt{a_\lambda}[p(D), \tilde{v}] u, u \rangle \nonumber \\
&+ \langle \tilde{v}'(V_\lambda -  \dot{x}^t) u, u \rangle + \langle (\D_x V_\lambda)\tilde{v} u, u\rangle. \nonumber
\end{align}
We successively consider each term on the right hand side. 

We will obtain positivity from the third term on the right hand side of (\ref{poscom}). Write the kernel of $[p(D), \tilde{v}]$ using a Taylor expansion:
\begin{align*}
\int &e^{i(x - y)\xi} p(\xi)(\tilde{v}(y) - \tilde{v}(x)) \, d\xi \\
&= \int e^{i(x - y)\xi} p(\xi)\left((y - x)\tilde{v}'(x) + (y - x)^2 \int_0^1 \tilde{v}''(h(x - y) + y)h \, dh \right) \, d\xi
\end{align*}
Then integrating by parts, we obtain
\begin{align*}
\int e^{i(x - y)\xi} \left(-ip'(\xi)\tilde{v}'(x) - p''(\xi) \int_0^1 \tilde{v}''(h(x - y) + y)h \, dh \right) \, d\xi \\
=: K_1(x, y) &+ K_2(x, y).
\end{align*}

Observe that $K_1$ is the kernel of 
$$-i\tilde{v}'p'(D) = -i\lambda^{\frac{3}{4}}w^2p'(D)$$
so that $K_1$ will contribute the positive operator. We observe that $\tilde{v}$ has been chosen so that the second order component $K_2$ is integrable, independent of $\lambda$. Indeed, observe that we may write, for a unit-scaled bump function $q$,
$$p(\xi) = \lambda^{\half} q(\lambda^{-1} \xi)$$
so that
$$\lambda^{\frac{3}{2}}\int e^{i(x - y)\xi} p''(\xi) \, d\xi = \lambda \widehat{q''}(\lambda(y - x))$$
is an integrable kernel. Further, since $\lambda^{-\frac{3}{2}}\tilde{v}'' = \widetilde{v''}$ is bounded, we have uniformly in $(x, y)$,
$$\lambda^{-\frac{3}{2}}\left|\int_0^1 \tilde{v}''(h(y - x) + x)h \, dh \right| \lesssim 1.$$
We conclude
$$\int \sqrt{a_\lambda}(x) K_2(x, y)u(y)\overline{u(x)} \, dy dx \lesssim \|a\|_{L^\infty}\|u\|_{L^2}^2.$$

We can similarly estimate the second and last terms on the right hand side of (\ref{poscom}) as errors (observe that $\tilde{v}$ has frequency support below $\lambda^{\frac{3}{4}} \ll \lambda$):
\begin{align*}
\|[p(D), \sqrt{a_\lambda}] \tilde{v} u \|_{L^2} &\lesssim \|\D_x \sqrt{a_\lambda}\|_{L^\infty} \|\tilde{v} u\|_{H^{-\half}}\\
&\leq \lambda^{\half} \FF(M(t))Z(t)\lambda^{-\half}\|u\|_{L^2}, \\
\|(\D_x V_\lambda) \tilde{v} u\|_{L^2} &\leq \FF(M(t))Z(t) \|u\|_{L^2}.
\end{align*}

It remains to consider the fourth term on the right hand side of (\ref{poscom}),
$$\langle \tilde{v}'(V_\lambda -  \dot{x}^t) u, u \rangle = \langle \tilde{v}'(V_\lambda(x) -  V_\lambda(x^t)) u, u \rangle - \half\langle \tilde{v}' \sqrt{a_\lambda} |\xi^t|^{-\frac{3}{2}}\xi^t u, u \rangle.$$
We bound the first of these terms as an error:
\begin{align*}
\|\tilde{v}'(V_\lambda(x) -  V_\lambda(x^t)) u \|_{L^2} \lesssim \ &\|\lambda^{\frac{3}{4}}v'(\lambda^{\frac{3}{4}}(x - x^t))(x - x^t)\|_{L^\infty}\|V_\lambda\|_{C^1}\| u \|_{L^2} \\
\lesssim \ &Z(t) \|u\|_{L^2}.
\end{align*}

We conclude from (\ref{poscom}) that, on the frequency support of $u$, and using the choice of $v$ in terms of $w$,
\begin{align}\label{integratelocal}
\frac{d}{dt}\langle \tilde{v} u, u \rangle = \ &\half\lambda^{\frac{3}{4}}\langle \sqrt{a_\lambda}\tilde{w}^2(|D|^{-\half} - |\xi^t|^{-\half}) u,  u \rangle \\
&+ O((\FF(M(t))Z(t))\|u\|_{L^2}^2 + \|f\|_{L^2}\|u\|_{L^2}). \nonumber
\end{align}

Next, we symmetrize the first term on the right hand side. We write
\begin{align*}
\tilde{w}^2(|D|^{-\half} - |\xi^t|^{-\half}) u = \ &\tilde{w}(|D|^{-\half} - |\xi^t|^{-\half})\tilde{w} u + \tilde{w}[|D|^{-\half}, \tilde{w}] u
\end{align*}
where the commutator is bounded in $L^2_x$ by
\begin{align*}
\|\tilde{w}'\|_{L^\infty}\|u\|_{H^{-\frac{3}{2}}} \leq \lambda^{\frac{3}{4}}\lambda^{-\frac{3}{2}}\|u\|_{L^2}.
\end{align*}
Using additionally the bounds on the Taylor coefficient, this may be absorbed into the error on the right hand side. Similarly, we write, using the frequency localization of $u$,
\begin{align*}
\sqrt{a_\lambda}(|D|^{-\half} - |\xi^t|^{-\half})\tilde{w} u = \ &(|D|^{-\half} - |\xi^t|^{-\half})^\half\sqrt{a_\lambda}(|D|^{-\half} - |\xi^t|^{-\half})^\half\tilde{w} u \\
&- [(|D|^{-\half} - |\xi^t|^{-\half})^\half, \sqrt{a_\lambda}](|D|^{-\half} - |\xi^t|^{-\half})^\half\tilde{w} u.
\end{align*}
Then, observing that $\tilde{w}$ has frequency support below $\lambda^{\frac{3}{4}} \ll \lambda$, we similarly may estimate the commutator by
$$\|\D_x \sqrt{a_\lambda}\|_{L^\infty} \|\tilde{w} u\|_{H^{-\frac{3}{2}}} \leq \lambda^{\half}\FF(M(t))Z(t)\lambda^{-\frac{3}{2}}\|u\|_{L^2}$$
which is better than needed with respect to the power of $\lambda$ to absorb into the right hand side.

After the symmetrization, integrating (\ref{integratelocal}) in time, we obtain
$$\lambda^{\frac{3}{4}}\| a_\lambda^\frac{1}{4}(|D|^{-\half} - |\xi^t|^{-\half})^\half \tilde{w}u\|_{L^2(I;L^2)}^2 \leq \FF(T)\|u\|_{L^\infty(I;L^2)}^2 + \|f\|_{L^1(I;L^2)}^2.$$
Using the lower bound from the Taylor sign condition and taking square roots,
$$\|(|D|^{-\half} - |\xi^t|^{-\half})^\half \tilde{w}u\|_{L^2(I;L^2)} \leq \lambda^{-\frac{3}{8}}\FF(T)(\|u\|_{L^\infty(I;L^2)} + \|f\|_{L^1(I;L^2)}).$$
Lastly, observe that for $\xi$ in the frequency support of $u$ (using Lemma \ref{freqpres}),
$$|\xi|^{-\half} - |\xi^t|^{-\half} \geq |\xi^t - c\mu|^{-\half} - |\xi^t|^{-\half} \gtrsim \lambda^{-\frac{3}{2}}\mu.$$
Using that the frequency support of $\tilde{w}$ at $\lambda^{\frac{3}{4}} \ll \mu \leq c\lambda$ essentially leaves that of $u$ unchanged, we obtain
$$\lambda^{-\frac{3}{4}}\mu^{\half}\|\tilde{w}u\|_{L^2(I;L^2)} \leq \lambda^{-\frac{3}{8}}\FF(T)(\|u\|_{L^\infty(I;L^2)} + \|f\|_{L^1(I;L^2)})$$
as desired. The case of frequency support $[\xi^t + c\mu, 2\lambda/c]$ is similar, and the case $[-2\lambda/c, -c\lambda/2]$ is better than needed.

\end{proof}

Using similar, cruder analyses, we have corresponding estimates for the cases of $u$ with low and high frequencies:

\begin{prop}\label{local1lowhigh}
Consider a smooth solution $u$ to
\begin{equation*}
(\D_t + V_\kappa\D_x + i\sqrt{a_\kappa |D|})u = f
\end{equation*}
on $I$. Let $(x^t, \xi^t)$ be a solution to (\ref{hamilton}) with initial data $(x_0, \xi_0)$ satisfying $\xi_0 \in [\lambda/2, 2\lambda]$. Let $u(t, \cdot)$ have frequency support $\{|\xi| \in [\kappa/2, 2\kappa]\}$. If $\kappa \leq c \lambda$, then 
$$\|w_{x^t, \kappa}u\|_{L^2(I;L^2)} \leq \kappa^{-\frac{1}{8}}\FF(T)(\|u\|_{L^\infty(I;L^2)} + \|f\|_{L^1(I;L^2)}).$$
If $\kappa \geq \lambda/c$, then 
$$\|w_{x^t, \lambda}u\|_{L^2(I;L^2)} \leq \lambda^{-\frac{1}{8}}\FF(T)(\|u\|_{L^\infty(I;L^2)} + \|f\|_{L^1(I;L^2)}).$$
\end{prop}

We can prove paradifferential counterparts, using essentially the same analysis as in Section \ref{sec:freqloc}:

\begin{cor}\label{local2}
Consider a smooth solution $u$ to (\ref{paralinear}) on $I$. Let $(x^t, \xi^t)$ be a solution to (\ref{hamilton}) with initial data $(x_0, \xi_0)$ satisfying $\xi_0 \in [\lambda/2, 2\lambda]$.
\begin{enumerate}[i)]
\item If $u(t, \cdot)$ has frequency support on one of
$$[c\lambda/2, \xi^t - c\mu], \quad [\xi^t + c\mu, 2\lambda/c], \quad [-2\lambda/c, -c\lambda/2],$$
then 
$$\|w_{x^t, \lambda}u\|_{L^2(I;L^2)} \leq \mu^{-\half}\lambda^{\frac{3}{8}}\FF(T)(\|u\|_{L^\infty(I;L^2)} + \|f\|_{L^1(I;L^2)}).$$
\item If $u(t, \cdot)$ has frequency support $\{|\xi| \in [\kappa/2, 2\kappa]\}$ and $\kappa \leq c \lambda$, then 
$$\|w_{x^t, \kappa}u\|_{L^2(I;L^2)} \leq \kappa^{-\frac{1}{8}}\FF(T)(\|u\|_{L^\infty(I;L^2)} + \|f\|_{L^1(I;L^2)}).$$
\item If $u(t, \cdot)$ has frequency support $\{|\xi| \in [\kappa/2, 2\kappa]\}$ and $ \kappa \geq \lambda/c$, then 
$$\|w_{x^t, \lambda}u\|_{L^2(I;L^2)} \leq \lambda^{-\frac{1}{8}}\FF(T)(\|u\|_{L^\infty(I;L^2)} + \|f\|_{L^1(I;L^2)}).$$
\end{enumerate}
\end{cor}

Lastly, we apply Corollary \ref{local2} to frequency localized pieces of a smooth solution $u$ to (\ref{paralinear}), with no assumed frequency localization. Recall that we construct a symmetric $\lambda$-frequency projection $S_{\xi_0, \lambda, \mu}$ with a $c\mu$-width gap at $\xi_0$, as in Section \ref{localnotation}.

\begin{cor}\label{local3}
Consider a smooth solution $u$ to (\ref{paralinear}) on $I$. Let $(x^t, \xi^t)$ be a solution to (\ref{hamilton}) with initial data $(x_0, \xi_0)$ satisfying $\xi_0 \in [\lambda/2, 2\lambda]$.
\begin{enumerate}[i)]
\item We have
$$\|w_{x^t, \lambda}S_{\xi^t, \lambda, \mu} u\|_{L^2(I;H^s)} \leq \mu^{-\frac{3}{2}}\lambda^{\frac{11}{8}}\FF(T)(\|u\|_{L^\infty(I;H^s)} + \|f\|_{L^1(I;H^s)}).$$
\item For $\kappa \leq c \lambda$,
$$\|w_{x^t, \kappa}S_\kappa u\|_{L^2(I;H^s)} \leq \kappa^{-\frac{1}{8}}\FF(T)(\|u\|_{L^\infty(I;H^s)} + \|f\|_{L^1(I;H^s)}).$$
\item For $\kappa \geq \lambda/c$, 
$$\|w_{x^t, \lambda}S_\kappa u\|_{L^2(I;H^s)} \leq \lambda^{-\frac{1}{8}}\FF(T)(\|u\|_{L^\infty(I;H^s)} + \|f\|_{L^1(I;H^s)}).$$
\end{enumerate}
\end{cor}

\begin{proof}
For $(ii)$ and $(iii)$, we may apply Corollary \ref{local2} on $S_\kappa u$ with inhomogeneity
$$S_\kappa f + [T_V \D, S_\kappa]u + i[T_{\sqrt{a}}|D|^\half, S_\kappa] u.$$
The commutators are estimated in the same way as in the proof of Theorem \ref{p1} in terms of Proposition \ref{prop:p2}, though we do not sum the frequency pieces here.

The case of $(i)$ is similar. However, for $p$ as in the definition of the symbol of $S_{\xi^t, \lambda, \mu}$, we observe that $p'$ is not of order $-1$ uniformly in $\mu$ and $\lambda$. Rather, we need to consider instead $\mu \lambda^{-1} p'$. Additionally, we need to estimate
$$[\D_t, p(D - \xi^t)]u = \dot{\xi}^t p'(D - \xi^t) u = (-\D_x V_\lambda \xi^t - \D_x \sqrt{a_\lambda} |\xi^t|^\half) p'(D - \xi^t) u$$
in $L^1(I;L^2)$ as an inhomogeneous term (and likewise for the similar $p(-D + \xi^t)$ case). Using Lemma \ref{freqpres} to see that $\xi^t \approx \xi_0$, we obtain
$$\|(-\D_x V_\lambda \xi^t - \D_x \sqrt{a_\lambda} |\xi^t|^\half) p'(D - \xi^t) u\|_{L^1(I;L^2)} \lesssim \lambda (\|V\|_{L^1(I; C^1)} + \|a\|_{L^1(I; C^\half)}) \mu^{-1} \|u\|_{L^\infty(I;L^2)}.$$
Using the Taylor coefficient estimates of Proposition \ref{taylorbd} yields the desired estimate.

\end{proof}

\subsection{Estimates on the surface and velocity field}

In this section we establish local smoothing estimates on the original unknowns $(\eta, \psi, V, B)$, using Proposition \ref{prop:paralinearization}.

\begin{prop}\label{wwlocalsmoothing}
Let $(x^t, \xi^t)$ be a solution to (\ref{hamilton}) with initial data $(x_0, \xi_0)$ satisfying $\xi_0 \in [\lambda/2, 2\lambda]$.
\begin{enumerate}[i)]
\item We have
\begin{align*}
\|w_{x^t, \lambda}S_{\xi^t, \lambda, \mu}(\eta, \psi)\|_{L^2(I;H^{s + \half})} + \|w_{x^t, \lambda}S_{\xi^t, \lambda, \mu}(V, B)\|_{L^2(I;H^{s})} \leq \mu^{-\frac{3}{2}}\lambda^{\frac{11}{8}}\FF(T).
\end{align*}
\item For $\kappa \leq c \lambda$,
$$ \|w_{x^t, \kappa} S_\kappa(\eta, \psi)\|_{L^2(I;H^{s + \half})} + \|w_{x^t, \kappa} S_\kappa(V, B)\|_{L^2(I;H^s)} \leq \kappa^{-\frac{1}{8}}\FF(T).$$
\item For $ \kappa \geq \lambda/c$,
$$\|w_{x^t, \lambda} S_\kappa(\eta, \psi)\|_{L^2(I;H^{s + \half})} + \|w_{x^t, \lambda} S_\kappa(V, B)\|_{L^2(I;H^s)} \leq \lambda^{-\frac{1}{8}}\FF(T).$$
\end{enumerate}
\end{prop}

\begin{proof}
We establish $(ii)$; the other cases are similar.

By Proposition \ref{prop:paralinearization}, we have that 
$$u = \langle D \rangle^{-s}(\langle D \rangle^s V + T_{\nabla \eta} \langle D \rangle^sB -iT_{\sqrt{a/|D|}} \langle D \rangle^s \nabla \eta)$$
solves 
$$(\D_t + T_{V}\D_x + iT_{\sqrt{a}}\sqrt{|D|})u = f$$
with 
$$\|f\|_{L^1(I;H^s)} \leq \FF(T).$$
Further, using Sobolev embedding and (\ref{ordernorm}), it is easy to see that
$$\|u\|_{L^\infty(I;H^s)} \leq \FF(T).$$
We conclude by Corollary \ref{local3},
$$\| w_{x^t, \kappa} S_\kappa u\|_{L^2(I;H^s)} \leq \kappa^{-\frac{1}{8}}\FF(T).$$
Using the frequency localization of $w_{x^t, \kappa}$ at $\kappa^{-\frac{3}{4}} \ll \kappa$, we may commute $\langle D \rangle^s$ to obtain
$$\| w_{x^t, \kappa}\langle D \rangle^s S_\kappa u\|_{L^2(I;L^2)} \leq \kappa^{-\frac{1}{8}}\FF(T).$$

\emph{Step 1.} First we estimate $\eta$. Taking the imaginary part of $u$,
$$\| w_{x^t, \kappa} S_\kappa T_{\sqrt{a}}|D|^{-\half}  \langle D \rangle^s\nabla \eta\|_{L^2(I;L^2)} \leq \kappa^{-\frac{1}{8}}\FF(T).$$

We may commute $S_\kappa$ with $T_{\sqrt{a}}$ by absorbing the commutator into the right hand side, using (\ref{sobolevcommutator}) with the estimates on the Taylor coefficient $a$ provided in Proposition \ref{taylorbd}:
\begin{align*}
\|[S_\kappa, T_{\sqrt{a}}]|D|^{-\half}\langle D \rangle^s \nabla \tilde{S}_\kappa \eta\|_{L^2(I;L^2)} &\leq \FF(T)\||D|^{-\half}\langle D \rangle^s  \nabla \tilde{S}_\kappa\eta\|_{L^\infty(I;H^{-\half })} \\
&\leq \FF(T)\|\tilde{S}_\kappa \eta\|_{L^\infty(I;H^{s})} \leq \kappa^{-\half}\FF(T)\|\tilde{S}_\kappa \eta\|_{L^\infty(I;H^{s + \half})}
\end{align*}
so that
$$\| w_{x^t, \kappa} T_{\sqrt{a}}|D|^{-\half}  \langle D \rangle^{s} \nabla S_\kappa \eta\|_{L^2(I;L^2)} \leq \kappa^{-\frac{1}{8}}\FF(T).$$
We may also exchange $T_{\sqrt{a}}$ with $\sqrt{a}$, absorbing the error into the right hand side using (\ref{sobolevparaproduct}), (\ref{sobolevparaerror}), and Proposition \ref{taylorbd}:
\begin{align*}
\|T_{|D|^{-\half}  \langle D \rangle^{s} \nabla S_\kappa \eta}\sqrt{a} \|_{L^2(I;L^2)} &\lesssim \||D|^{-\half}  \langle D \rangle^{s} \nabla S_\kappa \eta\|_{L^2(I;C_*^{\half-s})} \|\sqrt{a}\|_{L^\infty(I;H^{s - \half})} \\
&\leq \kappa^{-\half}\FF(T),\\
\|R(|D|^{-\half}  \langle D \rangle^{s} \nabla S_\kappa \eta, \sqrt{a}) \|_{L^2(I;L^2)} &\lesssim \||D|^{-\half}  \langle D \rangle^{s} \nabla S_\kappa \eta\|_{L^2(I;C_*^{\half-s})} \|\sqrt{a}\|_{L^\infty(I;H^{s - \half})} \\
&\leq \kappa^{-\half}\FF(T),
\end{align*}
so that
$$\|\sqrt{a}w_{x^t, \kappa}|D|^{-\half}  \langle D \rangle^{s} \nabla S_\kappa \eta\|_{L^2(I;L^2)} \leq \kappa^{-\frac{1}{8}}\FF(T).$$
Using the Taylor sign condition, $a \geq a_{min} > 0$, 
$$\|w_{x^t, \kappa}|D|^{-\half}  \langle D \rangle^{s} \nabla S_\kappa \eta\|_{L^2(I;L^2)} \leq \kappa^{-\frac{1}{8}}\FF(T).$$
Lastly, recalling the frequency localization of $w_{x^t, \kappa}$, we obtain the desired result.

\

\emph{Step 2.} Next, we estimate $B$. Taking the real part of $u$,
$$\| w_{x^t, \kappa}S_\kappa(\langle D \rangle^s V + T_{\nabla \eta} \langle D \rangle^s B)\|_{L^2(I;L^2)} \leq \kappa^{-\frac{1}{8}}\FF(T).$$
Similar to Step 1, we may commute $S_\kappa$ with $T_{\nabla \eta}$, obtaining
$$\| w_{x^t, \kappa}(\langle D \rangle^s \D_x^{-1}  S_\kappa\D_x V + T_{\nabla \eta} \langle D \rangle^s S_\kappa B)\|_{L^2(I;L^2)} \leq \kappa^{-\frac{1}{8}}\FF(T).$$
Recalling from Proposition 4.5 of \cite{alazard2014cauchy},
$$\D_x V = -G(\eta)B - \Gamma_y.$$
We may thus exchange $\D_x V$ with $-G(\eta)B$, absorbing the $\Gamma_y$ by using Proposition \ref{roughbottomesttame} to estimate
$$\|\langle D \rangle^s \D_x^{-1}  S_\kappa \Gamma_y\|_{L^2(I;L^2)} \lesssim \kappa^{-\half} \|\Gamma_y\|_{L^2(I;H^{s - \half})} \leq \kappa^{-\half} \FF(T).$$
Further, we can exchange $G(\eta) B$ with $|D| B$ by using Proposition \ref{DNparalineartame} to estimate
\begin{align*}
\|\langle D \rangle^s \D_x^{-1}  S_\kappa (G(\eta)B - |D| B)\|_{L^2(I;L^2)} &\lesssim \kappa^{-\half}\|G(\eta)B - |D| B\|_{L^2(I;H^{s - \half})} \leq \kappa^{-\half}\FF(T).
\end{align*}
We conclude
$$\| w_{x^t, \kappa}(-\D_x^{-1}|D| + T_{\nabla \eta})\langle D \rangle^{s}S_\kappa B\|_{L^2(I;L^2)} \leq \kappa^{-\frac{1}{8}}\FF(T).$$

Similar to Step 1, we may exchange $T_{\nabla \eta}$ with $\nabla \eta$, and using as usual the frequency localization of $w_{x^t, \kappa}$, commute $w_{x^t, \kappa}$ with $\D_x^{-1}|D|$:
$$\|(-\D_x^{-1}|D| + (\nabla \eta))w_{x^t, \kappa}\langle D \rangle^{s}S_\kappa B\|_{L^2(I;L^2)} \leq \kappa^{-\frac{1}{8}}\FF(T).$$

We can restore the paralinearization by estimating (using the frequency localization of $w_{x^t, \kappa}$ and $S_\kappa B$)
\begin{align*}
\|T_{w_{x^t, \kappa}\langle D \rangle^{s}S_\kappa B} \nabla \eta\|_{L^2(I;L^2)} &\lesssim \|w_{x^t, \kappa}\langle D \rangle^{s}S_\kappa B\|_{L^2(I;C_*^{\half - s})} \|\nabla \eta\|_{L^\infty(I;H^{s - \half})} \\
& \leq \kappa^{-\half}\FF(T).
\end{align*}
and similarly for the balanced-frequency term. We conclude
$$\|T_{i\xi^{-1}|\xi|+ \nabla \eta}w_{x^t, \kappa}\langle D \rangle^{s}S_\kappa B\|_{L^2(I;L^2)} \leq \kappa^{-\frac{1}{8}}\FF(T).$$
Lastly, by (\ref{ordernorm}), 
$$\|T_{(i\xi^{-1}|\xi| + \nabla \eta)^{-1}}T_{i\xi^{-1}|\xi| + \nabla \eta}w_{x^t, \kappa}\langle D \rangle^{s}S_\kappa B\|_{L^2(I;L^2)} \leq \kappa^{-\frac{1}{8}}\FF(T).$$
which we may exchange for the desired estimate by (\ref{sobolevcommutator}).

\

\emph{Step 3.} We estimate $V$. Using the estimate on $B$ from Step 2 along with an analysis similar to that in Step 1 to commute $T_{\nabla \eta}$, we have
$$\| w_{x^t, \kappa}S_\kappa T_{\nabla \eta} \langle D \rangle^s B\|_{L^2(I;L^2)} \leq \kappa^{-\frac{1}{8}}\FF(T).$$
Recalling the estimate on the real part of $u$ at the beginning of Step 2, we may absorb this into the right hand side. The remaining term is the desired estimate.

\

\emph{Step 4.} Lastly, we estimate $\psi$ using the formula
$$\nabla \psi = V + B \nabla \eta.$$
Note that it suffices to show
$$\|w_{x^t, \kappa} \langle D \rangle^{s - \frac{3}{8}} S_\kappa \nabla \psi\|_{L^2(I;L^2)} \leq \FF(T).$$
We easily have
$$\|w_{x^t, \kappa} \langle D  \rangle^{s - \frac{3}{8}} S_\kappa V\|_{L^2(I;L^2)} \leq \FF(T)$$
so it remains to show
$$\|w_{x^t, \kappa} \langle D \rangle^{s - \frac{3}{8}} S_\kappa B\nabla \eta\|_{L^2(I;L^2)} \leq \FF(T).$$
We may commute $B$ with $ \langle D \rangle^{s - \frac{3}{8}} S_\kappa $ by estimating
\begin{align*}
\|[\langle D \rangle^{s - \frac{3}{8}} S_\kappa, B]\nabla \eta\|_{L^2(I;L^2)} \lesssim \ &\|B\|_{L^2(I; C^1)} \|\eta\|_{L^2(I;H^{s - \frac{3}{8}})} + \|B\|_{L^2(I;H^{s - \frac{3}{8}})}\|\nabla \eta||_{L^\infty} \leq \FF(T).
\end{align*}
Then it remains to show
$$\|Bw_{x^t, \kappa}\langle D \rangle^{s - \frac{3}{8}}S_\kappa \nabla \eta\|_{L^2(I;L^2)} \leq \FF(T),$$
which is immediate from the local estimate on $\eta$.

\end{proof}

As a straightforward consequence of case $(ii)$ in Proposition \ref{wwlocalsmoothing}, we have

\begin{cor}\label{wwlsspaceeasy}
Let $(x^t, \xi^t)$ be a solution to (\ref{hamilton}) with initial data $(x_0, \xi_0)$ satisfying $\xi_0 \in [\lambda/2, 2\lambda]$. Then
$$\|(\eta, \psi)(t)\|_{L^2(I; LS^{s' + \half-}_{x^t, \lambda})} + \|(V, B)(t)\|_{L^2(I; LS^{s'-}_{x^t, \lambda})} \leq \FF(T).$$
\end{cor}

The other cases will be used in considering local Sobolev estimates on products in the following subsection.

\subsection{Local smoothing on products}

Recall that we define the following local seminorm for $\sigma \in \R$ (see Section \ref{localnotation}):
\begin{align*}
\|f\|_{LS^\sigma_{x_0, \xi_0, \lambda, \mu}} = \ &\sum_{\kappa \in [c\mu, c\lambda]} ( \kappa^{\frac{3}{4}} \mu^{-1} \|S_\kappa f\|_{H^\sigma} +\|w_{x_0, \kappa}S_\kappa f\|_{H^\sigma}) \\
&+ \lambda^{\frac{3}{4}} \mu^{-1}\|S_{\geq c\lambda} f \|_{H^\sigma} + \sum_{\kappa \geq \lambda/c} \|w_{x_0, \lambda} S_\kappa f\|_{H^\sigma} \\
&+ \|w_{x_0, \lambda}S_{\xi_0, \lambda, \mu}f\|_{H^\sigma}.
\end{align*}
This definition is motivated by the following balanced-frequency product estimate:

\begin{prop}\label{bilinear}
Let $\alpha + \beta = \alpha' + \beta' > 0$. Then
\begin{align*}
\|w_{x_0, \lambda}S_\mu R(f, g)\|_{H^{\alpha + \beta}} \lesssim \ &\|f\|_{LS^\alpha_{x_0, \xi_0, \lambda, \mu}}\|g\|_{C_*^\beta} + \|S_\lambda f\|_{C_*^{\alpha'}}\|w_{x_0, \lambda}S_{\xi_0, \lambda, \mu}g\|_{H^{\beta'}}.
\end{align*}
In the case $\alpha = \beta = 0$, we have
\begin{align*}
\|w_{x_0, \lambda}S_\mu R(f, g)\|_{L^2} \lesssim \ &\|f\|_{LS^0_{x_0, \xi_0, \lambda, \mu}}\|g\|_{C_*^{0+}} + \|S_\lambda f\|_{C_*^{\alpha'}}\|w_{x_0, \lambda}S_{\xi_0, \lambda, \mu}g\|_{H^{\beta'}}.
\end{align*}
\end{prop}

\begin{proof}
For simplicity, we set $\alpha = \beta = \alpha' = \beta' = 0$ and accordingly use $L^\infty$ in place of $C_*^\alpha$; the generalization is easily obtained. First write
\begin{align}
\label{prersum} R(f, g) = \ &R(S_{\leq c\lambda} f, S_{\leq c\lambda} g) + R(S_{c\lambda \leq \cdot \leq \lambda/c} f, S_{c\lambda \leq \cdot \leq \lambda/c} g) + R(S_{\geq \lambda/c} f, S_{\geq \lambda/c} g). 
\end{align}

From the first term of (\ref{prersum}), we have
$$w_{x_0, \lambda} S_\mu  R(S_{\leq c\lambda} f, S_{\leq c\lambda} g) = w_{x_0, \lambda} S_\mu R(S_{c\mu \leq \cdot \leq c\lambda} f, S_{c\mu \leq \cdot \leq c\lambda} g).$$
A term of this sum takes the form
$$w_{x_0, \lambda} S_\mu ((S_\kappa f)(S_\kappa g))$$
with $\kappa \in [c\mu, c\lambda]$. Observe that 
$$\|w_{x_0, \lambda} S_\mu ((S_\kappa f)(S_\kappa g))\|_{L^2} \lesssim \|w_{x_0, \kappa} S_\mu ((S_\kappa f)(S_\kappa g))\|_{L^2}.$$
We commute
\begin{align*}
\|[w_{x_0, \kappa}, S_\mu] ((S_\kappa f)(S_\kappa g))\|_{L^2} &= \|[w_{x_0, \kappa}, S_\mu] \tilde{S}_\mu((S_\kappa f)(S_\kappa g))\|_{L^2} \\
&\lesssim \kappa^{\frac{3}{4}} \mu^{-1}\|\tilde{S}_\mu((S_\kappa f)(S_\kappa g))\|_{L^2} \\
&\lesssim \kappa^{\frac{3}{4}} \mu^{-1}\|S_\kappa f\|_{L^2}\|g\|_{L^\infty}.
\end{align*}
Thus it remains to consider
$$S_\mu w_{x_0, \kappa} ((S_\kappa f)(S_\kappa g))$$
which we estimate
$$\|S_\mu w_{x_0, \kappa}((S_\kappa f)(S_\kappa g))\|_{L^2} \lesssim \|w_{x_0, \kappa}S_\kappa f\|_{L^2} \|g\|_{L^\infty}.$$

For the latter two terms of (\ref{prersum}), we commute
\begin{align*}
\|[w_{x_0, \lambda}, S_\mu] R(S_{\geq c\lambda} f, S_{\geq c\lambda} g)\|_{L^2} &= \|[w_{x_0, \lambda}, S_\mu]\tilde{S}_\mu R(S_{\geq c\lambda} f, S_{\geq c\lambda} g)\|_{L^2} \\
&\lesssim \lambda^{\frac{3}{4}} \mu^{-1}\|\tilde{S}_\mu R(S_{\geq c\lambda} f, S_{\geq c\lambda} g)\|_{L^2} \\
&\lesssim \lambda^{\frac{3}{4}} \mu^{-1}\|S_{\geq c\lambda} f \|_{L^2}\| S_{\geq c\lambda} g\|_{C_*^{0+}}
\end{align*}
so it remains to consider 
\begin{equation}\label{rsum}
S_\mu w_{x_0, \lambda} (R(S_{c\lambda \leq \cdot \leq \lambda/c} f, S_{c\lambda \leq \cdot \leq \lambda/c} g) + R(S_{\geq \lambda/c} f, S_{\geq \lambda/c} g)).
\end{equation}

Consider a term of the latter sum in (\ref{rsum}), which takes the form
$$S_\mu w_{x_0, \lambda}((S_\kappa f)(S_\kappa g))$$
with $\kappa \geq \lambda/c$. We easily estimate
$$\|S_\mu w_{x_0, \lambda}((S_\kappa f)(S_\kappa g)) \|_{L^2} \lesssim \|w_{x_0, \lambda} S_\kappa f\|_{L^2} \|g\|_{L^\infty}.$$

It remains to consider (an absolute number of) terms of the former sum in (\ref{rsum}),
$$S_\mu w_{x_0, \lambda}((S_\lambda f)(S_\lambda g)).$$
First recall that $w_{x_0, \lambda}$ has frequency support below $\lambda^{\frac{3}{4}} \ll \mu$ and so essentially does not disturb the frequency pieces of $S_\lambda f$ and $S_\lambda g$ at the $c\mu$-scale. Then observe that the outermost $S_\mu$ eliminates $c\mu$-width frequency pieces of $S_\lambda f$ and $S_\lambda g$ that are not separated in absolute value by at least $\mu/8$.

Thus, write 
$$\psi_\lambda(\xi) = p(\xi- \xi_0) + p(-\xi + \xi_0) + q(\xi- \xi_0) + q(-\xi + \xi_0)$$
where $p$ was constructed in the definition of $S_{\xi_0, \lambda, \mu}$ and $q$ has support $\{|\xi| \leq c\mu\}$. Using the previous observation, we may write
\begin{align*}
S_\mu w_{x_0, \lambda}((S_\lambda f)(S_\lambda g)) = \ &S_\mu w_{x_0, \lambda}((S_{\xi_0, \lambda, \mu}f)(S_\lambda g) \\
&+((q(D - \xi_0) + q(-D + \xi_0))f)(S_{\xi_0, \lambda, \mu} g)).
\end{align*}
Then
\begin{align*}
\|S_\mu w_{x_0, \lambda}(S_\lambda f)(S_\lambda g)\|_{L^2} \lesssim \ &\|w_{x_0, \lambda}S_{\xi_0, \lambda, \mu}f\|_{L^2}\|S_\lambda g\|_{L^\infty} + \|S_\lambda f\|_{L^\infty}\|w_{x_0, \lambda}S_{\xi_0, \lambda, \mu}g\|_{L^2}.
\end{align*}

\end{proof}

In the special case where at least one of $f, g$ is truncated to low frequencies $\{|\xi| \leq c\lambda\}$, we see from the proof of Proposition \ref{bilinear} that we can use instead the simpler local seminorm,
$$\|f\|_{LS^\sigma_{x_0, \lambda}} = \sum_{\kappa \leq c\lambda} \|w_{x_0, \kappa} S_\kappa f\|_{H^\sigma}.$$

\begin{cor}\label{bilinearsimple}
Let $\alpha + \beta \geq 0$, and $f = S_{\leq c\lambda}f$ or $g = S_{\leq c\lambda}g$. Then
\begin{align*}
\|w_{x_0, \lambda}S_\mu R(f, g)\|_{H^{\alpha + \beta}} \lesssim \ &(\|f\|_{LS^\alpha_{x_0, \lambda}} + \lambda^{\frac{3}{4}}\mu^{-1} \|S_{\leq c\lambda} f\|_{H^\alpha})\|g\|_{C_*^\beta}.
\end{align*}
\end{cor}

Lastly, using Proposition \ref{wwlocalsmoothing}, we observe that we estimate $(\eta, \psi, V, B)$ in the full local seminorm:

\begin{cor}\label{wwlsspace}
Let $0 \leq \sigma \leq s$. Let $(x^t, \xi^t)$ be a solution to (\ref{hamilton}) with initial data $(x_0, \xi_0)$ satisfying $\xi_0 \in [\lambda/2, 2\lambda]$. Then
$$\lambda^{-\sigma}\mu^{s'}\|(\eta, \psi)\|_{L^2(I;LS^{\sigma + \half}_{x^t, \xi^t, \lambda, \mu})} + \lambda^{-\sigma}\mu^{s'}\|(V, B)\|_{L^2(I;LS^{\sigma}_{x^t, \xi^t, \lambda, \mu})} \leq \FF(T).$$
\end{cor}

\begin{proof}
We consider $V$; the other terms are similar. 

First consider the terms of $LS^{\sigma}_{x^t, \xi^t, \lambda, \mu}$ that do not have local weights. We have
$$\lambda^{-\sigma}\kappa^{\frac{3}{4}} \mu^{-1} \|S_\kappa V\|_{H^\sigma} \lesssim \kappa^{\frac{3}{4} - s} \mu^{-1} \|S_\kappa V\|_{H^s} \leq  \kappa^{-\frac{1}{8}} \mu^{-s'} \FF(T).$$
Summing geometrically with respect to $\kappa$ yields the desired estimate. Similarly, 
$$\lambda^{-\sigma}\lambda^{\frac{3}{4}} \mu^{-1}\|S_{\geq c\lambda} f \|_{H^\sigma} \lesssim \lambda^{\frac{3}{4} - s} \mu^{-1}\|S_{\geq c\lambda} f \|_{H^s} \leq \lambda^{-\frac{1}{8}} \mu^{-s'} \FF(T)$$
which is better than needed.

It remains to consider the three terms of $LS^{\sigma}_{x^t, \xi^t, \lambda, \mu}$ with local weights. First consider the low frequency sum, using case $(ii)$ of Proposition \ref{wwlocalsmoothing}:
$$\lambda^{-\sigma}\sum_{\kappa \in [c\mu, c\lambda]} \|w_{x^t, \kappa}S_\kappa V\|_{L^2(I;H^\sigma)} \leq \sum_{\kappa \in [c\mu, c\lambda]} \kappa^{-s - \frac{1}{8}} \FF(T) \leq  \mu^{-s'} \FF(T).$$
For the high frequency sum, use case $(iii)$:
$$\lambda^{-\sigma}\sum_{\kappa \geq \lambda/c} \|w_{x^t, \lambda} S_\kappa V\|_{L^2(I;H^\sigma)} \leq \sum_{\kappa \geq \lambda/c} \kappa^{-(s - \sigma)} \lambda^{-\sigma - \frac{1}{8}} \FF(T) \leq \lambda^{-s'} \FF(T)$$
which is better than needed.

Lastly, for the $\lambda$-frequency term, use case $(i)$:
$$\lambda^{-\sigma}\|w_{x^t, \lambda}S_{\xi^t, \lambda, \mu}V\|_{L^2(I;H^\sigma)} \lesssim \lambda^{-s}\mu^{-\frac{3}{2}} \lambda^{\frac{11}{8}} \FF(T) \leq \mu^{-s'}\FF(T).$$

\end{proof}

\section{Local Estimates for the Dirichlet Problem}\label{sec:dirichletprob}

The goal of this section is to recall and establish various estimates for the elliptic Dirichlet problem with rough boundary. This is satisfied by the velocity potential $\phi$ and pressure $P$, and is the primary object in the Dirichlet to Neumann map. This problem was studied in \cite{alazard2014cauchy}, which discusses Sobolev estimates, and \cite{alazard2014strichartz}, \cite{ai2017low}, which discuss H\"older estimates. Here, our goal is to establish local Sobolev counterparts.

Note that the results of this section are time independent, so we often omit the variable $t$.

\subsection{Flattening the boundary}\label{ssec:flattening}

Consider the Dirichlet problem with rough boundary. Denote the strip of constant depth $h$ along the surface by
$$\Omega_1 = \{(x, y) \in \R^{d + 1}: \eta(x) - h < y < \eta(x)\},$$
and on $\Omega_1$, let $\theta$ satisfy
\begin{equation}\label{eqn:origellip}
\Delta_{x, y} \theta(x, y) = F, \qquad \theta|_{y = \eta(x)} = f.
\end{equation}

We would like estimates on $\theta$ and its derivatives. To address the rough boundary $\eta$ and corresponding rough domain, we change variables to a problem with flat boundary and domain, following \cite{lannes2005well}, \cite{alazard2014cauchy}. Denote the flat strip by
$$\tilde{\Omega}_1 = \{(x, z) \in \R^{d + 1}: z \in (-1, 0)\}.$$
Then define the Lipschitz diffeomorphism $\rho: \tilde{\Omega}_1 \rightarrow \Omega_1$ by
$$\rho(x, z) = (1 + z)e^{\delta z \langle D \rangle}\eta(x) - z(e^{-(1 + z)\delta\langle D \rangle} \eta(x) - h)$$
where $\delta$ is chosen small as in \cite[Lemma 3.6]{alazard2014cauchy}. 
Lastly, for $u: \Omega_1 \rightarrow \R$, let $\tilde{u}: \tilde{\Omega}_1 \rightarrow \R$ denote
$$\tilde{u}(x, z) = u(x, \rho(x, z)).$$

Then on the flat domain $\tilde{\Omega}_1$, $\tilde{\theta}$ satisfies
\begin{equation}\label{eqn:flatelliptic}
(\D_z^2 + \alpha \Delta + \beta \cdot \nabla \D_z - \gamma \D_z)\tilde{\theta} = F_0, \qquad \tilde{\theta} |_{z = 0} = f
\end{equation}
where
$$\alpha = \frac{(\D_z \rho)^2}{1 + |\nabla \rho|^2}, \quad \beta = -2 \frac{\D_z \rho \nabla \rho}{1 + |\nabla \rho|^2}, \quad \gamma = \frac{\D_z^2 \rho + \alpha \Delta \rho + \beta \cdot \nabla \D_z \rho}{\D_z \rho}$$
and
$$F_0(x, z) = \alpha \tilde{F}.$$

The diffeomorphism $\rho$ has essentially the same regularity as $\eta$, with additional smoothing when averaged over $z$. We recall the precise estimates below.

\begin{prop}\label{prop:diffeoest}
Let $J \subseteq [-1, 0]$. The diffeomorphism $\rho$ satisfies
\begin{align*}
\|(\D_z \rho)^{-1}\|_{C^0(J; C_*^{r - 1})} + \|\nabla_{x,z} \rho\|_{C^0(J; C_*^{r - 1})} &\lesssim \|\eta\|_{H^{s + \half}} \\
\|(\D_z \rho - h, \nabla \rho)\|_{X^{s - \half}(J)} + \|\D_z^2 \rho\|_{X^{s - \frac{3}{2}}(J)} &\lesssim \|\eta\|_{H^{s + \half}} \\
\|\nabla_{x,z} \rho\|_{U^{r - \half}(J)} + \|\D_z^2 \rho\|_{U^{r - \frac{3}{2}}(J)}&\lesssim 1 + \|\eta\|_{W^{r + \half, \infty}} \\
\|\nabla_{x,z} \rho\|_{L^1(J; H^{s + \half})} + \|\D_z^2 \rho\|_{L^1(J; H^{s - \half})}&\lesssim \|\eta\|_{H^{s + \half}} \\
\|\nabla_{x,z} \rho\|_{L^1(J; C_*^{r + \half})} + \|\D_z^2 \rho\|_{L^1(J; C_*^{r - \half})}&\lesssim 1 + \|\eta\|_{W^{r + \half, \infty}}.
\end{align*}
\end{prop}
\begin{proof}
The first estimate is a consequence of \cite[Lemmas 3.6, 3.7]{alazard2014cauchy} combined with Sobolev embedding. The first terms of the second and third estimates are from \cite[Lemma 3.7]{alazard2014cauchy} and \cite[Lemma B.1]{alazard2014strichartz} respectively. The corresponding estimates on $\D^2_z \rho$ are proven similarly from the definition of $\rho$. Lastly, the proofs of the fourth and fifth estimates are similar to those of the second and third respectively, by using $L^1_z$ in place of $L^2_z$.
\end{proof}

Then the above coefficients $\alpha, \beta, \gamma$, expressed in terms of $\rho$, satisfy similar estimates with proofs straightforward from paraproduct estimates:
\begin{cor}\label{cor:abgest}
Let $J \subseteq [-1, 0]$. For $\alpha, \beta$, and $\gamma$ defined as above,
\begin{align*}
\|(\alpha - h^2, \beta)\|_{X^{s - \half}(J)} + \|\gamma\|_{X^{s - \frac{3}{2}}(J)} &\leq \FF(\|\eta\|_{H^{s + \half}}) \\
\|(\alpha, \beta)\|_{U^{r - \half}(J)} + \|\gamma\|_{L^2(J; C_*^{r - 1})} &\leq \FF(\|\eta\|_{H^{s + \half}})(1 + \|\eta\|_{W^{r + \half, \infty}}) \\
 \|(\alpha, \beta)\|_{L^1(J; C_*^{r + \half})}  + \|\gamma\|_{L^1(J; C_*^{r - \frac{1}{2}})}  &\leq \FF(\|\eta\|_{H^{s + \half}})(1 + \|\eta\|_{W^{r + \half, \infty}}).
\end{align*}
\end{cor}

To establish local elliptic estimates, we would like a local counterpart to (\ref{eqn:flatelliptic}), which we obtain by simply commuting:

\begin{prop}\label{localflatelliptic}
Let $0 \leq \sigma \leq s - \half$ and $z_0 \in [-1, 0]$, $J = [z_0, 0]$. Consider $\tilde{\theta}$ solving (\ref{eqn:flatelliptic}). Denote $wS = w_{x_0, \lambda} S_{\xi_0, \lambda, \mu}$ or $wS = w_{x_0, \kappa}S_\kappa$ with $\kappa \leq c\lambda$. Then we can write
\begin{equation}\label{eqn:localflatelliptic}
(\D_z^2 + \alpha \Delta + \beta \cdot \nabla \D_z - \gamma \D_z)wS\tilde{\theta} = wSF_0 + F_4, \qquad (wS\tilde{\theta}) |_{z = 0} = wSf
\end{equation}
where
\begin{align*}
\|F_4\|_{Y^{\sigma}(J)} &\leq \FF(\|\eta\|_{H^{s + \half}})(1 + \|\eta\|_{W^{r + \half, \infty}})\|\nabla_{x,z} \tilde{\theta}\|_{X^{\sigma - \frac{1}{4}}(J)}.
\end{align*}
\end{prop}

\begin{proof}

First observe that by Sobolev embedding, 
\begin{align*}
\|\nabla_{x,z} \tilde{\theta}\|_{L^\infty(J; C_*^{\sigma - s + \half+})} &\lesssim \|\nabla_{x,z} \tilde{\theta}\|_{L^\infty(J; H^{\sigma - \frac{1}{4}})}, \\
\|\nabla_{x,z} \tilde{\theta}\|_{L^2(J; C_*^{\sigma - s + 1+})} &\lesssim \|\nabla_{x,z} \tilde{\theta}\|_{L^2(J; H^{\sigma + \frac{1}{4}})}
\end{align*}
so we may freely use the H\"older norms on the right hand side of our error estimate.

We have
$$F_4 = [\alpha \Delta  + \beta \cdot \nabla \D_z - \gamma \D_z, wS] \tilde{\theta}.$$
We discuss the term $\alpha \Delta$ of the commutator below; the term $\beta \cdot \nabla \D_z$ is similar. First, we exchange $\alpha$ with $T_\alpha$ as follows. Estimate
$$\|T_{\Delta wS \tilde{\theta}}\alpha\|_{H^{\sigma - \half}} \lesssim \|\Delta wS \tilde{\theta}\|_{C_*^{\sigma - s}}\|\alpha\|_{H^{s - \half}} \lesssim \|\nabla \tilde{\theta}\|_{C_*^{\sigma - s + 1}}\|\alpha\|_{H^{s - \half}}$$
followed by integrating in $z$ and using the estimates on $\alpha$ in Corollary \ref{cor:abgest}. The balanced-frequency term is similar. We likewise have
$$\|wST_{\Delta \tilde{\theta}} \alpha\|_{H^{\sigma - \half}} \lesssim \|\Delta \tilde{\theta}\|_{C_*^{\sigma - s}}\|\alpha\|_{H^{s - \half}} \lesssim \|\nabla \tilde{\theta}\|_{C_*^{\sigma - s + 1}}\|\alpha\|_{H^{s - \half}}.$$
Then it remains to estimate
$$[T_\alpha \Delta, wS] \tilde{\theta} = T_\alpha [\Delta, w]S \tilde{\theta} + [T_\alpha, wS]\Delta \tilde{\theta}.$$

For the first commutator, we have, estimating $\alpha$ via Corollary \ref{cor:abgest},
\begin{align*}
\|T_\alpha [\Delta, w]S\tilde{\theta}\|_{L^2(J;H^{\sigma - \half})} &\lesssim \|\alpha \|_{L^\infty(J;L^\infty)}\|[\Delta, w]S\tilde{\theta}\|_{L^2(J;H^{\sigma - \half})} \\
&\leq \FF(\|\eta\|_{H^{s + \half}})\|[\Delta, w]S\tilde{\theta}\|_{L^2(J;H^{\sigma - \half})}.
\end{align*}
Then, using $\kappa$ to denote the frequency of the projection $S$,
\begin{align*}
\|[\Delta, w]S \tilde{\theta}\|_{H^{\sigma - \half}} &\lesssim \|\nabla w\|_{L^\infty}\|\nabla S\tilde{\theta}\|_{H^{\sigma - \half}} + \|\Delta w\|_{L^\infty}\|S\tilde{\theta}\|_{H^{\sigma - \half}} \\
&\lesssim \kappa^{\frac{3}{4}}\|\nabla S\tilde{\theta}\|_{H^{\sigma - \half}} + \kappa^{\frac{3}{2}}\|S\tilde{\theta}\|_{H^{\sigma - \half}} \\
&\lesssim \|\nabla \tilde{\theta}\|_{H^{\sigma + \frac{1}{4}}}.
\end{align*}
Integrating in $z$, we obtain the desired estimate.

A typical term in the sum defining the second commutator is
$$[(S_{\leq \kappa/8} \alpha) S_\kappa, w S] \Delta \tilde{\theta} = (S_{\leq \kappa/8} \alpha)[S_\kappa, w]S \Delta \tilde{\theta} + w [(S_{\leq \kappa/8} \alpha), S]S_\kappa \Delta \tilde{\theta}$$
with $\kappa$ comparable to the frequency of the projection $S$. The first of these commutators is estimated in the same way as the previous paragraph, putting one derivative on $w$. For the second commutator, using the frequency localization, we may drop the weight $w$. In the case $wS = w_{x_0, \kappa}S_\kappa$,
$$\|[(S_{\leq \kappa/8} \alpha), S]S_\kappa \Delta \tilde{\theta}\|_{H^{\sigma - \half}} \lesssim \lambda^{\frac{3}{2} - r} \|\alpha\|_{C_*^{r - \half}}\|S_k \Delta \tilde{\theta}\|_{H^{\sigma - \frac{3}{2}}} \lesssim (1 + \|\eta\|_{W^{r + \half, \infty}})\|\nabla \tilde{\theta}\|_{H^{\sigma}}.$$
In the case $wS = w_{x_0, \lambda} S_{\xi_0, \lambda, \mu}$, commuting induces a $\mu^{-1}\lambda \leq \lambda^{\frac{1}{4}}$ loss:
$$\|[(S_{\leq \lambda/8} \alpha), S]S_\lambda \Delta \tilde{\theta}\|_{H^{\sigma - \half}} \lesssim \lambda^{\frac{3}{2} - r + \frac{1}{4}} \|\alpha\|_{C_*^{r - \half}}\|S_\lambda \Delta \tilde{\theta}\|_{H^{\sigma - \frac{3}{2}}} \lesssim (1 + \|\eta\|_{W^{r + \half, \infty}})\|\nabla \tilde{\theta}\|_{H^{\sigma + \frac{1}{4}}}.$$

It remains to consider 
$$[\gamma \D_z, wS] \tilde{\theta}.$$
Here we do not use the commutator, estimating the two terms separately. We decompose into paraproducts,
$$\gamma \D_z \tilde{\theta} = T_{\gamma} \D_z \tilde{\theta} + T_{\D_z \tilde{\theta}} \gamma + R(\gamma, \D_z \tilde{\theta}).$$
We have by (\ref{sobolevparaproduct}), (\ref{sobolevparaerror}), and Corollary \ref{cor:abgest},
\begin{align*}
\|T_{\gamma} \D_z \tilde{\theta} \|_{L^1(J; H^{\sigma})} + \|R(\gamma, \D_z \tilde{\theta})\|_{L^1(J; H^\sigma)} &\lesssim \|\gamma\|_{L^2(J; C_*^{r - 1})}\|\D_z \tilde{\theta}\|_{L^2(J;H^\sigma)}  \\
&\leq \FF(\|\eta\|_{H^{s + \half}})(1 +  \|\eta\|_{W^{r + \half, \infty}})\|\D_z \tilde{\theta}\|_{L^2(J; H^\sigma)}.
\end{align*}
Similarly,
\begin{align*}
\|T_{\D_z \tilde{\theta}} \gamma \|_{L^2(J; H^{\sigma - \half})} &\lesssim \|\gamma\|_{L^2(J; H^{s - 1})}\|\D_z \tilde{\theta}\|_{L^\infty(J; C_*^{\sigma - s + \half+})} \leq \FF(\|\eta\|_{H^{s + \half}})\|\D_z \tilde{\theta}\|_{L^\infty(J; C_*^{\sigma - s + \half+})}
\end{align*}
\end{proof}

\subsection{Elliptic estimates}\label{partoell}

We recall the (global) Sobolev elliptic estimate for (\ref{eqn:flatelliptic}): 

\begin{prop}\cite[Proposition 3.16]{alazard2014cauchy}\label{sobolevinductive}
Let $-\half \leq \sigma \leq s - \half$ and $-1 < z_1 < z_0 < 0$. Denote $J_0 = [z_0, 0], J_1 = [z_1, 0]$. Consider $\tilde{\theta}$ solving (\ref{eqn:flatelliptic}). Then
$$\|\nabla_{x,z} \tilde{\theta}\|_{X^\sigma(J_0)} \leq \FF(\|\eta\|_{H^{s + \half}})(\|f\|_{H^{\sigma + 1}} + \|F_0\|_{Y^\sigma(J_1)} + \|\nabla_{x,z} \tilde{\theta}\|_{X^{-\half}(J_1)}).$$
\end{prop}

Applying this to (\ref{eqn:localflatelliptic}), we have the following local Sobolev counterpart:
\begin{prop}\label{localinduction}
Let $0 \leq \sigma \leq s - \half$ and $-1 < z_1 < z_0 < 0$. Denote $J_0 = [z_0, 0], J_1 = [z_1, 0]$. Consider $\tilde{\theta}$ solving (\ref{eqn:flatelliptic}). Denote $wS = w_{x_0, \lambda} S_{\xi_0, \lambda, \mu}$ or $wS = w_{x_0, \kappa}S_\kappa$ with $\kappa \leq c\lambda$. Then
\begin{align*}
\|\nabla_{x, z} wS\tilde{\theta}\|_{X^{\sigma}(J_0)} \leq \ &\FF(\|\eta\|_{H^{s + \half}})(\|wS f\|_{H^{\sigma + 1}} + \|wSF_0 \|_{Y^{\sigma}(J_1)} \\
&+ (1 + \|\eta\|_{W^{r + \half, \infty}})(\|f\|_{H^{\sigma + \frac{3}{4}}} + \|F_0\|_{Y^{\sigma - \frac{1}{4}}(J_1)} + \|\nabla_{x,z} \tilde{\theta}\|_{X^{-\half}(J_1)})). 
\end{align*}
\end{prop}

\begin{proof}
Applying Proposition \ref{sobolevinductive} to (\ref{eqn:localflatelliptic}), we have
$$\|\nabla_{x,z} wS \tilde{\theta}\|_{X^\sigma(J_0)} \leq \FF(\|\eta\|_{H^{s + \half}})(\|wSf\|_{H^{\sigma + 1}} + \|wSF_0 + F_4\|_{Y^\sigma(J_1)} + \|\nabla_{x,z} \tilde{\theta}\|_{X^{-\half}(J_1)}).$$

By Proposition \ref{localflatelliptic},
\begin{align*}
\|F_4\|_{Y^{\sigma}(J_1)} &\leq \FF(\|\eta\|_{H^{s + \half}})((1 + \|\eta\|_{W^{r + \half, \infty}})\|\nabla_{x,z} \tilde{\theta}\|_{L^2(J_1; H^{\sigma + \frac{1}{4}})}).
\end{align*}
Then apply Proposition \ref{sobolevinductive} on 
$$\|\nabla_{x,z} \tilde{\theta}\|_{L^2(J_1; H^{\sigma + \frac{1}{4}})}.$$
\end{proof}

\subsection{Estimates in the harmonic case}

In the special case of (\ref{eqn:origellip}) where $\theta$ satisfies 
\begin{equation}\label{specialelliptic}
\Delta_{x, y} \theta(x, y) = 0, \qquad \theta|_{y = \eta(x)} = f, \qquad \D_n \theta |_\Gamma = 0,
\end{equation}
as is the case for instance when defining the Dirichlet to Neumann map, we have the following ``base case'' estimate:
\begin{prop}\cite[Remark 3.15]{alazard2014cauchy}\label{basecase}
Consider $\theta$ solving (\ref{specialelliptic}). Then
$$\|\nabla_{x, z}\tilde{\theta}\|_{X^{-\half}([-1, 0])} \leq \FF(\|\eta\|_{H^{s + \half}})\|f\|_{H^\half}.$$
\end{prop}

Combined with Proposition \ref{sobolevinductive}, this yields:
\begin{prop}\label{originalcoordestsobolev}
Let $-\half \leq \sigma \leq s - \half$ and $z_0 \in (-1, 0]$, $J = [z_0, 0]$. Consider $\theta$ solving (\ref{specialelliptic}). Then
$$\|\nabla_{x,z} \tilde{\theta}\|_{X^\sigma(J)} \leq \FF(\|\eta\|_{H^{s + \half}})\|f\|_{H^{\sigma + 1}}.$$
\end{prop}

We also have a local counterpart, combining with Corollary \ref{localinduction}:

\begin{prop}\label{localinductioncorinducted}
Let $0 \leq \sigma \leq s - \half$ and $z_0 \in (-1, 0]$, $J = [z_0, 0]$. Consider $\theta$ solving (\ref{specialelliptic}). Denote $wS = w_{x_0, \lambda} S_{\xi_0, \lambda, \mu}$ or $wS = w_{x_0, \kappa}S_\kappa$ with $\kappa \leq c\lambda$. Then
\begin{align*}
\|\nabla_{x, z} wS\tilde{\theta}\|_{X^{\sigma}(J)} \leq \ &\FF(\|\eta\|_{H^{s + \half}})(\|wS f\|_{H^{\sigma + 1}} + (1 + \|\eta\|_{W^{r + \half, \infty}})\|f\|_{H^{\sigma + \frac{3}{4}}}). 
\end{align*}
\end{prop}

Lastly, for later use, we recall a H\"older counterpart:
\begin{prop}\label{oldholderellipticest}
Let $0 \leq \sigma < r - \half$ and $z_0 \in (-1, 0]$, $J = [z_0, 0]$. Consider $\tilde{\theta}$ solving (\ref{specialelliptic}). Then
\begin{align*}
\|\nabla_{x,z} \tilde{\theta}\|_{U^{\sigma}(J)} \leq \ &\FF(\|\eta\|_{H^{s + \half}})(1 + \|\eta\|_{W^{r + \half, \infty}})\|f\|_{H^{\sigma + 1 +}} + \|f\|_{C_*^{\sigma + 1 +}}.
\end{align*}
\end{prop}

\begin{proof}
First, as a straightforward consequence of Proposition A.7 from \cite{ai2017low},
\begin{align*}
\|\nabla_{x,z} \tilde{\theta}\|_{U^{\sigma}(J)} \leq \ &\FF(\|\eta\|_{H^{s + \half}})(1 + \|\eta\|_{W^{r + \half, \infty}})\|\nabla_{x, z} \tilde{\theta}\|_{U^{\sigma - \half +}(J_1)} + \|f\|_{C_*^{\sigma + 1 +}}.
\end{align*}
Then applying Sobolev embedding and Proposition \ref{originalcoordestsobolev},
$$\|\nabla_{x, z} \tilde{\theta}\|_{U^{\sigma - \half +}(J_1)} \lesssim \|\nabla_{x, z} \tilde{\theta}\|_{X^{\sigma +}(J_1)} \leq \FF(\|\eta\|_{H^{s + \half}})\|f\|_{H^{\sigma + 1 +}}.$$
\end{proof}

\section{Local Dirichlet to Neumann Paralinearization}\label{sec:localdnpar}

In this section we paralinearize the Dirichlet to Neumann map, with a local Soboleev estimate on the error. Recall the Dirichlet to Neumann map is given by solving (\ref{specialelliptic}),
$$\Delta_{x, y} \theta = 0, \qquad \theta|_{y = \eta(x)} = f, \qquad \D_n \theta |_\Gamma = 0,$$
and setting
$$(G(\eta)f)(x) = \sqrt{1 + |\nabla \eta|^2} (\D_n \theta) |_{y = \eta(x)} = ((\D_y - \nabla \eta \cdot \nabla)\theta) |_{y = \eta(x)}.$$

In the flattened coordinates discussed in Section \ref{ssec:flattening}, we have the homogeneous counterpart to (\ref{eqn:flatelliptic}) (recall that we write $\tilde{u}(x, z) = u(x, \rho(x, z))$ where $\rho$ is the diffeomorphism that flattens the boundary defined by the graph of $\eta$),
\begin{equation}\label{homogellipticeqn}
(\D_z^2 + \alpha \Delta + \beta \cdot \nabla \D_z - \gamma \D_z)\tilde{\theta} = 0, \qquad \tilde{\theta} |_{z = 0} = f,
\end{equation}
and we may write the Dirichlet to Neumann map as
$$G(\eta)f = \left. \left(\frac{1 + |\nabla \rho|^2}{\D_z \rho} \D_z \tilde{\theta} - \nabla \rho \cdot \nabla \tilde{\theta}\right)\right|_{z = 0}.$$

To paralinearize the latter term $\nabla \rho \cdot \nabla \tilde{\theta}$ with errors in local Sobolev norm, we will require local counterparts to the estimates on the diffeomorphism $\rho$ from Section \ref{ssec:flattening}. We establish these in the next two subsections. To paralinearize the former $\D_z \tilde{\theta}$ term, we factor (\ref{homogellipticeqn}) as the product of forward and backward paralinearized parabolic evolutions,
\begin{equation}\label{factoreqn}
(\D_z - T_a)(\D_z - T_A)\tilde{\theta} \approx 0
\end{equation}
(in a sense to be made precise) where we define
$$a = \half (-i\beta \cdot \xi - \sqrt{4\alpha |\xi|^2 - (\beta \cdot \xi)^2}), \quad A = \half (-i\beta \cdot \xi + \sqrt{4\alpha |\xi|^2 - (\beta \cdot \xi)^2}).$$
A single parabolic estimate then provides the paralinearization $\D_z \tilde{\theta} \approx T_A\tilde{\theta}$.

\subsection{Commutator and product estimates}

Before addressing the diffeomorphism $\rho$, we observe some general commutator and product estimates regarding the local weights used in Section \ref{sec:locsmooth}. First, we observe that the various weights of the form $wS$ essentially commute with paraproducts $T_a$:

\begin{prop}\label{localparaproduct}
Let $m \in \R$ and $\rho > 0$.
\begin{enumerate}[i)]
\item We have
\begin{align*}
\|w_{x_0, \lambda}S_{\xi_0, \lambda, \mu} T_a u \|_{H^{m'}} &\lesssim \|a\|_{C_*^{-\rho}} (\|w_{x_0, \lambda}\tilde{S}_{\xi_0, \lambda, \mu} u\|_{H^{m' + \rho}} + \lambda^{\half} \mu^{-\half}\|\tilde{S}_\mu u\|_{H^{m + \rho}}), \\
\|w_{x_0, \lambda} S_\mu T_a u \|_{H^{m'}} &\lesssim \|a\|_{C_*^{-\rho}} (\|w_{x_0, \lambda} \tilde{S}_\mu u\|_{H^{m' + \rho}} + \lambda^{\half} \mu^{-\half}\|\tilde{S}_\mu u\|_{H^{m + \rho}}).
\end{align*}
\item For $\kappa \leq c \lambda$,
$$\|w_{x_0, \kappa} S_\kappa T_a u \|_{H^{m'}} \lesssim \|a\|_{C_*^{-\rho}} (\|w_{x_0, \kappa} \tilde{S}_\kappa u\|_{H^{m' + \rho}} + \|u\|_{H^{m' + \rho - \frac{1}{4}}}).$$
\item For $ \kappa \geq \lambda/c$,
$$\|w_{x_0, \lambda} S_\kappa T_a u \|_{H^{m'}} \lesssim \|a\|_{C_*^{-\rho}} (\|w_{x_0, \lambda} \tilde{S}_\kappa u\|_{H^{m' + \rho}} + \|u\|_{H^{m' + \rho - \frac{1}{4}}}).$$
\end{enumerate}
In the case $\rho = 0$, the same estimates hold with $L^\infty$ in the place of $C_*^{-\rho}$. 
\end{prop}

\begin{proof}
We first prove case $(iii)$, noting that $(ii)$ may be proven in the same way. Using the frequency support of $w_{x_0, \lambda}$ at $\lambda^{\frac{3}{4}}$, we may write $S_\kappa w_{x_0, \lambda} T_a u$ as a sum of finitely many terms of the form 
$$S_\kappa w_{x_0, \lambda} (S_{\leq \tilde{\kappa}/8}a) S_{\tilde{\kappa}} u =  S_\kappa (S_{\leq \tilde{\kappa}/8}a) w_{x_0, \lambda} S_{\tilde{\kappa}} u$$
where $\tilde{\kappa} \in [\kappa/8, 8\kappa]$. Then estimate
$$\|S_\kappa (S_{\leq \tilde{\kappa}/8}a) w_{x_0, \lambda} S_{\tilde{\kappa}} u\|_{H^{m'}} \lesssim \kappa^{m}\|S_{\leq \tilde{\kappa}/8}a\|_{L^\infty} \|w_{x_0, \lambda} S_{\tilde{\kappa}} u\|_{L^2} \lesssim \|a\|_{C_*^{-\rho}} \|w_{x_0, \lambda} S_{\tilde{\kappa}} u\|_{H^{m' + \rho}}.$$
It then remains to commute, using the frequency support of $w_{x_0, \lambda}$:
$$\|[w_{x_0, \lambda}, S_\kappa] T_a u \|_{H^{m'}} \lesssim \kappa^{-\frac{1}{4}}\|T_a \tilde{S}_\kappa u \|_{H^{m'}} \lesssim \|a\|_{C_*^{-\rho}}\|\tilde{S}_\kappa u \|_{H^{m' + \rho - \frac{1}{4}}}.$$

The second estimate of $(i)$ is similar, except that when commuting, 
$$\|[w_{x_0, \lambda}, S_\mu] T_a u \|_{H^{m'}} \lesssim \lambda^{\frac{3}{4}} \mu^{-1} \|T_a \tilde{S}_\mu u \|_{H^{m'}} \lesssim \lambda^{\frac{3}{4}} \mu^{-\frac{7}{8}} \|a\|_{C_*^{-\rho}}\|\tilde{S}_\mu u \|_{H^{m + \rho}}.$$
Then observe that
$$\lambda^{\frac{3}{4}} \mu^{-\frac{7}{8}} \leq \lambda^{\half} \mu^{-\half}.$$
The first estimate of $(i)$ is similar, differing only in the commutator.
\end{proof}

As a simple consequence, we have the following local product estimates:

\begin{cor}\label{localproduct}
Let $m \in \R$ and $\rho > 0$.
\begin{enumerate}[i)]
\item We have
\begin{align*}
\|w_{x_0, \lambda}S_{\xi_0, \lambda, \mu}(uv)\|_{H^{m'}} &\lesssim (\|w_{x_0, \lambda}\tilde{S}_{\xi_0, \lambda, \mu} u\|_{H^{m'}} + \lambda^{\half} \mu^{-\half}\| u\|_{H^{m}})\|v\|_{L^\infty} \\
&\quad + \|u\|_{C_*^{-\rho}}(\|w_{x_0, \lambda}\tilde{S}_{\xi_0, \lambda, \mu} v\|_{H^{m' + \rho}} + \lambda^{\half} \mu^{-\half}\|v\|_{H^{m + \rho}}) \\
&\quad + \|v\|_{C_*^{\frac{1}{8}}}\|u\|_{H^{m}},
\end{align*}
\begin{align*}
\|w_{x_0, \lambda} S_\mu(uv)\|_{H^{m'}} \lesssim \ &(\|w_{x_0, \lambda} \tilde{S}_\mu u\|_{H^{m'}} + \lambda^{\half} \mu^{-\half}\|u\|_{H^m})\|v\|_{L^\infty} \\
&+ \|u\|_{C_*^{-\rho}}(\|w_{x_0, \lambda} \tilde{S}_\mu v\|_{H^{m' + \rho}} + \lambda^{\half} \mu^{-\half}\|v\|_{H^{m + \rho}}) \\
&+ \|v\|_{C_*^{\frac{1}{8}}}\|u\|_{H^{m}},
\end{align*}
\item For $\kappa \leq c \lambda$,
\begin{align*}
\|w_{x_0, \kappa} S_\kappa(uv)\|_{H^{m'}} \lesssim \ &(\|w_{x_0, \kappa} \tilde{S}_\kappa u\|_{H^{m'}} + \|u\|_{H^m})\|v\|_{L^\infty} \\
&+ \|u\|_{C_*^{-\rho}}(\|w_{x_0, \kappa} \tilde{S}_\kappa v\|_{H^{m' + \rho}} + \|v\|_{H^{m + \rho}}) \\
&+ \|v\|_{C_*^{\frac{1}{8}}}\|u\|_{H^{m}},
\end{align*}
\item For $ \kappa \geq \lambda/c$,
\begin{align*}
\|w_{x_0, \lambda} S_\kappa(uv)\|_{H^{m'}} \lesssim \ &(\|w_{x_0, \lambda} \tilde{S}_\kappa u\|_{H^{m'}} + \|u\|_{H^m})\|v\|_{L^\infty} \\
&+ \|u\|_{C_*^{-\rho}}(\|w_{x_0, \lambda} \tilde{S}_\kappa v\|_{H^{m' + \rho}} + \|v\|_{H^{m + \rho}}) \\
&+ \|v\|_{C_*^{\frac{1}{8}}}\|u\|_{H^{m}},
\end{align*}
\end{enumerate}
In the case $\rho = 0$, the same estimates hold with $L^\infty$ in the place of $C_*^{-\rho}$. 
\end{cor}

\begin{proof}
Write 
$$wS(uv) = wS(T_u v + T_v u + R(u, v)).$$
For the third term, we may use (\ref{sobolevparaerror}),
$$\| R(u, v) \|_{H^{m + \frac{1}{8}}} \lesssim \|u\|_{H^{m}}\|v\|_{C_*^{\frac{1}{8}}}.$$
Then the first and second terms are estimated by Proposition \ref{localparaproduct}.

\end{proof}

We will also need a generalization to paradifferential operators:

\begin{prop}\label{localparadiff}
Let $m, \tilde{m}, k \in \R$, $\rho, \tilde{\rho} \in [0, 1]$, $a \in \Gamma^m_{\tilde{\rho}}$, $b \in \Gamma_{\tilde{\rho}}^{\tilde{m}}$, and $\kappa^{\frac{3}{4}} \ll \mu$. Then
\begin{align*}
\|[w_{x_0, \kappa}  S_\mu, T_a] u \|_{H^k} &\lesssim M_0^m(a)\kappa^{\frac{3\rho}{4}} \mu^{-\frac{3\rho}{4}}\| \tilde{S}_\mu u\|_{H^{k + m - \frac{\rho}{4}}}, \\
\| w_{x_0, \kappa}  S_\mu T_a u \|_{H^k} &\lesssim M_0^m(a)(\| w_{x_0, \kappa} S_\mu u \|_{H^{k + m}} + \kappa^{\frac{3\rho}{4}} \mu^{-\frac{3\rho}{4}}\| \tilde{S}_\mu u\|_{H^{k + m - \frac{\rho}{4}}}), \\
\| w_{x_0, \kappa}  S_\mu(T_a T_b - T_{ab}) u \|_{H^k} &\lesssim (M_{\tilde{\rho}}^m(a) M_0^{\tilde{m}}(b) + M_0^m(a) M_{\tilde{\rho}}^{\tilde{m}}(b)) \\
&\quad \cdot (\| w_{x_0, \kappa} S_\mu u \|_{H^{k + m + \tilde{m} - \tilde{\rho}}} + \kappa^{\frac{3\rho}{4}} \mu^{-\frac{3\rho}{4}}\| \tilde{S}_\mu u\|_{H^{k + m + \tilde{m} - \tilde{\rho} - \frac{\rho}{4}}}).
\end{align*}
If additionally $a$ is homogeneous in $\xi$, we have
\begin{align*}
\|w_{x_0, \kappa}  S_\mu[T_a, \D_t + T_V \cdot \nabla] u(t)\|_{H^k} &\lesssim (M_0^m(a)\|V(t)\|_{W^{r, \infty}} + M_0^m((\D_t + V \cdot \nabla)a)) \\
&\quad \cdot (\|w_{x_0, \kappa}  S_\mu u(t)\|_{H^{k + m}} + \kappa^{\frac{3\rho}{4}} \mu^{-\frac{3\rho}{4}}\|u(t)\|_{H^{k + m - \frac{\rho}{4}}}).
\end{align*}
\end{prop}

\begin{proof}
We immediately have by (\ref{ordernorm}),
$$\|T_a  w_{x_0, \kappa}  S_\mu u \|_{H^k} \lesssim M_0^m(a)\| w_{x_0, \kappa} S_\mu u \|_{H^{k + m}}.$$
This implies the second estimate once we prove the first estimate.

First observe that by the frequency localization of $w_{x_0, \kappa}$ at $\kappa^{\frac{3}{4}} \ll \mu$, we may replace $w_{x_0, \kappa}$ with $T_{w_{x_0, \kappa}}$. Then, we may apply (\ref{sobolevcommutator}),
\begin{align*}
\|[w_{x_0, \kappa} S_\mu, T_a] u\|_{H^k} &\lesssim (M_\rho^m(a)M_0^0(w_{x_0, \kappa}) + M_0^m(a)M_\rho^0(w_{x_0, \kappa}))\|\tilde{S}_\mu u\|_{H^{k + m - \rho}} \\
&\lesssim (M_\rho^m(a) + M_0^m(a)\kappa^{\frac{3\rho}{4}})\|\tilde{S}_\mu  u\|_{H^{k + m - \rho}}.
\end{align*}
However, also observe that since $w_{x_0, \kappa} S_\mu$ is order 0, a sharper analysis (see for instance \cite[Theorem 3.4.A]{taylor2008pseudodifferential}) shows that in fact
\begin{align*}
\|[w_{x_0, \kappa} S_\mu, T_a] u\|_{H^k} &\lesssim M_0^m(a)M_\rho^0(w_{x_0, \kappa})\|\tilde{S}_\mu u\|_{H^{k + m - \rho}} \lesssim M_0^m(a)\kappa^{\frac{3\rho}{4}}\|\tilde{S}_\mu  u\|_{H^{k + m - \rho}}.
\end{align*}

For the third estimate, use that the paradifferential calculus essentially forms an algebra (see the remarks preceding Corollary 3.4.G in \cite{taylor2008pseudodifferential}), with
$$T_a T_b - T_{ab} = T_A \in OP\Gamma_0^{m + \tilde{m} - \tilde{\rho}}.$$
Then apply the second estimate, with $A$ in the place of $a$ and $m + \tilde{m} - \tilde{\rho}$ in the place of $m$:
$$\| w_{x_0, \kappa}  S_\mu T_A u \|_{H^k} \lesssim M_0^{m + \tilde{m} - \tilde{\rho}}(A)(\| w_{x_0, \kappa} S_\mu u \|_{H^{k + m + \tilde{m} - \tilde{\rho}}} + \kappa^{\frac{3\rho}{4}} \mu^{-\frac{3\rho}{4}}\| \tilde{S}_\kappa u\|_{H^{k + m + \tilde{m} - \tilde{\rho} - \frac{\rho}{4}}}).$$
Then the third estimate follows from (\ref{sobolevcommutator}). One similarly obtains the fourth estimate, using \cite[Lemma 2.15]{alazard2014cauchy} in place of (\ref{sobolevcommutator}).
\end{proof}

\subsection{Local estimates on the diffeomorphism}

In this subsection we establish local counterparts to Proposition \ref{prop:diffeoest}. For brevity, it is convenient to let $wS$ denote any of the localizing operators found in the $LS^\sigma_{x_0, \lambda}$ or $LS^\sigma_{x_0, \xi_0, \lambda, \mu}$ seminorms. When using this notation, let $\kappa$ denote the frequency of $S$, so that $S = S\tilde{S}_\kappa$.

\begin{prop}\label{localdiffeo}
Let $J \subseteq [-1, 0]$. The diffeomorphism $\rho$ satisfies
$$\|wS(\D_z \rho, \nabla \rho)\|_{X^{s' - \half}(J)} \leq \FF(\|\eta\|_{H^{s + \half}})(1 + \|wS\eta\|_{H^{s' + \half}}).$$
\end{prop}

\begin{proof}
Consider first $\nabla = \nabla_x$. Recall that on $z \in (-1, 0)$,
$$\rho(x, z) = (1 + z)e^{\delta z \langle D \rangle} \eta(x) - z(e^{-(1 + z)\delta \langle D \rangle} \eta(x) - h)$$
so without loss of generality, we consider $e^{\delta z \langle D \rangle} \eta$. Since
$$\nabla e^{\delta z \langle D \rangle} wS \eta$$
may be estimated as in the proof of Proposition \ref{prop:diffeoest} with $wS \eta$ in the place of $\eta$, it suffices to estimate the commutator
$$[wS, \nabla e^{\delta z \langle D \rangle}] \eta = [w, \nabla e^{\delta z \langle D \rangle}] S \eta$$
in $X^{s' - \half}(J)$.

Since $\nabla e^{\delta z \langle D \rangle}$ is an order 1 operator uniformly in $z \in (-1, 0)$,
$$\|[w, \nabla e^{\delta z \langle D \rangle}]S \eta\|_{H^{s' - \half}} \lesssim \kappa^{-\frac{1}{4}}\|S \eta\|_{H^{s' + \half}} \lesssim \|\eta\|_{H^{s + \frac{3}{8}}}.$$
This yields the $L_z^\infty(J; H^{s' - \half})$ estimate.

For the $L_z^2$ estimate, we also have that $|z|^{\half}\langle D \rangle^{-\half} \nabla e^{\delta z \langle D \rangle}$ is an order 0 operator uniformly in $z$, so that
\begin{align*}
|z|^{-\half}\|\langle D \rangle^{\half}[wS, |z|^{\half}\langle D \rangle^{-\half} \nabla e^{\delta z \langle D \rangle}]\eta\|_{H^{s'}} &\lesssim |z|^{-\half}\| [w, |z|^{\half}\langle D \rangle^{-\half} \nabla e^{\delta z \langle D \rangle}]S\eta\|_{H^{s' + \half}}\\
&\lesssim \kappa^{-\frac{1}{4}} |z|^{-\half}\|S\eta\|_{H^{s' + \half}} \lesssim |z|^{-\half}\|\eta\|_{H^{s' + \frac{1}{4}}}.
\end{align*}
Integrating in $z$ yields the desired estimate in $L_z^2(J; H^{s})$, provided we also estimate the commutator
\begin{align*}
|z|^{-\half}\|[\langle D \rangle^{\half}, wS] |z|^{\half}\langle D \rangle^{-\half} \nabla e^{\delta z \langle D \rangle}\eta\|_{H^{s'}} &\lesssim \kappa^{-\frac{1}{4}}|z|^{-\half}\||z|^{\half}\langle D \rangle^{-\half} \nabla e^{\delta z \langle D \rangle}S\eta\|_{H^{s' + \frac{1}{2}}} \\
&\lesssim |z|^{-\half}\|\eta\|_{H^{s' + \frac{1}{4}}}
\end{align*}
which we again integrate in $z$.

For the estimate on $\D_z \rho$, recall that on $z \in (-1, 0)$,
\begin{equation}\label{dzrho}
\D_z\rho = h + e^{\delta z \langle D \rangle} \eta - e^{-(1 + z)\delta \langle D \rangle} \eta + (1 + z)\delta \langle D \rangle e^{\delta z \langle D \rangle} \eta + z \delta \langle D \rangle e^{-(1 + z)\delta \langle D \rangle} \eta.
\end{equation}
Without loss of generality, we consider $e^{\delta z \langle D \rangle} \eta$ and $\langle D \rangle e^{\delta z \langle D \rangle} \eta.$ The former is easy to estimate, as $e^{\delta z \langle D \rangle}$ is uniformly bounded on $H^s$. For the latter, as before, it suffices to estimate the commutator
$$[wS, \langle D \rangle e^{\delta z \langle D \rangle}] \eta.$$
This is estimated in the same way as in the previous analysis of $\nabla \rho$, with $\langle D \rangle$ in the place of $\nabla$.
\end{proof}

We also have the analogous estimates on the second derivatives:

\begin{cor}\label{cor:d2rho}
Let $J \subseteq [-1, 0]$. The diffeomorphism $\rho$ satisfies
$$\|wS \nabla_{x, z}^2 \rho\|_{X^{s' - \frac{3}{2}}(J)} \leq \FF(\|\eta\|_{H^{s + \half}})(1 + \|wS\eta\|_{H^{s' + \half}}).$$
\end{cor}

\begin{proof}

First consider $\nabla \nabla_{x, z} \rho$. We easily see by commuting $\nabla$ that
$$\|wS\nabla_x \nabla_{x, z} \rho\|_{X^{s' - \frac{3}{2}}(J)} \lesssim \|wS\nabla_{x, z} \rho\|_{X^{s' - \frac{1}{2}}(J)} + \|\nabla_{x, z} \rho\|_{X^{s' - \frac{3}{4}}(J)}$$
on which we can apply the estimates on $\rho$ from Propositions \ref{localdiffeo} and \ref{prop:diffeoest}.

The estimate on $\D_z^2$ is similarly elementary, after computing $\D_z^2 \rho$ from definition.
\end{proof}

Next, we establish estimates on the coefficients of (\ref{eqn:flatelliptic}). First, we require local estimates on reciprocals. Notably, we use an algebraic argument instead of a Moser estimate:

\begin{prop}\label{recip}
Let $J \subseteq [-1, 0]$. The diffeomorphism $\rho$ satisfies
\begin{align*}
\|wS (\D_z \rho)^{-1}\|_{X^{s' - \frac{1}{2}}(J)} +& \|wS (1 + |\nabla \rho|^2)^{-1}\|_{X^{s' - \frac{1}{2}}(J)}\\
&\leq \FF(\|\eta\|_{H^{s + \half}})(\lambda^{\half} \mu^{-\half} + \|\eta\|_{W^{r + \half, \infty}} + \|wS\eta\|_{H^{s' + \half}}).
\end{align*}
\end{prop}

\begin{proof}
We consider the first term on the left hand side, and the $L_z^\infty(J; H^{s' - \half})$ case. The other cases are obtained with the appropriate modifications. 

First, by commuting $w$ and $S$ as in the proof of Proposition \ref{localparaproduct}, it suffices to consider
$$\|Sw(\D_z \rho)^{-1}\|_{H^{s' - \half}}.$$
Then we have, using that $\D_z \rho \geq \min(h/2, 1)$ by Lemma 3.6 of \cite{alazard2014cauchy}, 
$$\|Sw (\D_z \rho)^{-1}\|_{H^{s' - \half}} \lesssim \|(\D_z \rho) \langle D \rangle^{s' - \half} Sw (\D_z \rho)^{-1}\|_{L^2},$$
so it suffices to estimate
$$\|[(\D_z \rho), \langle D \rangle^{s' - \half} S] w(\D_z \rho)^{-1}\|_{L^2}.$$

As with the commutator and product estimates of the previous subsection, the main burden is to absorb the extra $1/8$ derivatives in either a local Sobolev or H\"older norm. For this, we reduce to paraproducts. The balanced-frequency terms are easily estimated using $C_*^{\frac{1}{8}}$ on one term, so we consider only the low-high terms. First, we have the commutator, in the case that $S$ is a typical frequency projection $S_\kappa$,
$$\|[T_{\D_z \rho}, \langle D \rangle^{s' - \half} S] w(\D_z \rho)^{-1}\|_{L^2} \lesssim \|\D_z \rho\|_{C_*^{\half}} \|(\D_z \rho)^{-1}\|_{H^{s' - 1}}$$
which suffices by using the $\rho$ estimates of Proposition \ref{prop:diffeoest}. In the case $S = S_{\xi_0, \lambda, \mu}$, we have a $\mu^{-1}\lambda \leq \lambda^{\frac{1}{4}}$ loss, which is still better than needed. We also have
$$\|T_{\langle D \rangle^{s' - \half} S w(\D_z \rho)^{-1}} \D_z \rho\|_{L^2} \lesssim \|\langle D \rangle^{s' - \half} S w(\D_z \rho)^{-1}\|_{C_*^{\half - s}} \|\D_z \rho\|_{H^{s - \half}} \lesssim \|(\D_z \rho)^{-1}\|_{C_*^{\frac{1}{8}}} \|\D_z \rho \|_{H^{s - \half}}.$$

It remains to consider the term
$$ST_{ w(\D_z \rho)^{-1}} \D_z \rho$$
in $H^{s' - \half}$. First observe that by (\ref{sobolevcommutator}),
\begin{align*}
\|S(T_{ w(\D_z \rho)^{-1}} - &T_{(\D_z \rho)^{-1}} T_w) \D_z \rho\|_{H^{s' - \half}} \\
&\lesssim (M_\half^0(w)M_0^0((\D_z \rho)^{-1}) + M_0^0(w)M_\half^0((\D_z \rho)^{-1})) \| \tilde{S}_\kappa  \D_z \rho\|_{H^{s' - 1}} \\
&\leq\FF(\|\eta\|_{H^{s + \half}})(\kappa^{\frac{3}{8}} + \|\eta\|_{W^{r + \half, \infty}}) \| \tilde{S}_\kappa  \D_z \rho\|_{H^{s' - 1}}
\end{align*}
so that we may consider 
$$ST_{(\D_z \rho)^{-1}}T_w \D_z \rho = ST_{(\D_z \rho)^{-1}}w \D_z \rho.$$
Similar to before, we may commute $[S, T_{(\D_z \rho)^{-1}}]$ to
$$T_{(\D_z \rho)^{-1}}S w \D_z \rho$$
with the additional but acceptable loss in the case $S = S_{\xi_0, \lambda, \mu}$. We commute $S$ and $w$ as in the proof of Proposition \ref{localparaproduct} (and as at the beginning of this proof). Lastly, observe that
$$\|T_{(\D_z \rho)^{-1}} wS \D_z \rho\|_{H^{s' - \half}} \lesssim \|(\D_z \rho)^{-1}\|_{L^\infty} \|wS \D_z \rho \|_{H^{s' - \half}}$$
on which we can apply the estimates on $\rho$ from Propositions \ref{localdiffeo} and \ref{prop:diffeoest}.

\end{proof}

\begin{cor}\label{localabgest}
Let $J \subseteq [-1, 0]$. For $\alpha, \beta$, and $\gamma$ defined as above,
\begin{align*}
\|wS(\alpha, \beta)\|_{X^{s' - \frac{1}{2}}(J)} + \|wS\gamma\|_{X^{s' - \frac{3}{2}}(J)} &\leq \FF(\|\eta\|_{H^{s + \half}})(\lambda^{\half}\mu^{-\half} + \|\eta\|_{W^{r +\half,\infty}} + \|wS\eta\|_{H^{s' + \half}}).
\end{align*}
\end{cor}

\begin{proof}
For $\alpha$ and $\beta$, apply the product and reciprocal estimates, Propositions \ref{localproduct} and \ref{recip} respectively, with Proposition \ref{prop:diffeoest}. For the product estimates, use $\rho = 0$.

For $\gamma$, use the same estimates with $\rho = 1/2$.

\end{proof}

Lastly, we estimate the coefficients in the local seminorm:

\begin{cor}\label{abgls}
Let $J \subseteq [-1, 0]$ and $0 \leq \sigma \leq s - \half$. For $\alpha, \beta$, and $\gamma$ defined as above,
\begin{align*}
\mu^{\sigma'}\|(\alpha, \beta)&\|_{L^\infty(J;LS^{s - \sigma - \half}_{x_0, \xi_0, \lambda, \mu}) \cap L^2(J;LS^{s - \sigma}_{x_0, \xi_0, \lambda, \mu})} + \mu^{\sigma'}\|\gamma\|_{L^\infty(J;LS^{s - \sigma - \frac{3}{2}}_{x_0, \xi_0, \lambda, \mu}) \cap L^2(J;LS^{s - \sigma - 1}_{x_0, \xi_0, \lambda, \mu})} \\
&\leq \FF(\|\eta\|_{H^{s + \half}})(\lambda^{\half}\mu^{-\half} +\|\eta\|_{W^{r +\half,\infty}} + \mu^{\sigma'}\|\eta\|_{LS^{s - \sigma + \half}_{x_0, \xi_0, \lambda, \mu}}).
\end{align*}
\end{cor}
\begin{proof}
We consider the $L^2_z$ case; the $L^\infty_z$ case is similar. By Corollary \ref{localabgest}, we have in all three cases of $wS$,
$$\|wS(\alpha, \beta)\|_{L^2(J;H^{s - \sigma})} \leq \FF(\|\eta\|_{H^{s + \half}})(\kappa^{-\sigma'}\lambda^{\half}\mu^{-\half} + \kappa^{-\sigma'}\|\eta\|_{W^{r +\half,\infty}} + \|wS\eta\|_{H^{s - \sigma + \half}}).$$
Thus, we have
$$\|(\alpha, \beta)\|_{L^2(J;LS^{s - \sigma}_{x_0, \xi_0, \lambda, \mu})} \leq \FF(\|\eta\|_{H^{s + \half}})(\mu^{-\sigma'}\lambda^{\half}\mu^{-\half} + \mu^{-\sigma'}\|\eta\|_{W^{r +\half,\infty}} + \|\eta\|_{LS^{s - \sigma + \half}_{x_0, \xi_0, \lambda, \mu}}).$$
The $\gamma$ estimate is similar using Corollary \ref{localabgest}. 

\end{proof}

\subsection{Factoring the elliptic equation}

In this subsection we estimate the error on the right hand side of (\ref{factoreqn}). In preparation, first we record estimates on the symbols $a$ and $A$ in (\ref{factoreqn}). Define
$$\MM_\rho^m(a) = \sup_{z \in J} M_\rho^m(a(z)), \qquad \MM_\rho^{m, 2}(a) = \| M_\rho^m(a(z))\|_{L^2_z(J)}.$$

\begin{prop}\label{aabds}
For $a, A$ defined as above,
\begin{align*}
\MM_{0}^1(a) + \MM_{0}^1(A) &\leq \FF(\|\eta\|_{H^{s + \half}}) \\
\MM_{\half}^1(a) + \MM_{\half}^1(A) + \MM^1_{-\frac{1}{2}}(\D_z A) &\leq \FF(\|\eta\|_{H^{s + \half}})(1 + \|\eta\|_{W^{r + \half, \infty}}) \\
\MM_{1}^{1, 2}(a) + \MM_{r}^{1,2}(A) + \MM^{1,2}_{0}(\D_z A) &\leq \FF(\|\eta\|_{H^{s + \half}})(1 + \|\eta\|_{W^{r + \half, \infty}}) 
\end{align*}
\end{prop}

\begin{proof}
The first estimate is from \cite[Lemma 3.22]{alazard2014cauchy}. The second estimate is from \cite[(B.45)]{alazard2014strichartz}. The third estimate is similar, but is proven using the $L^2_z$ part of Proposition \ref{prop:diffeoest}, instead of $C_z^0$.
\end{proof}

\begin{prop} \label{inhomogbd}
Let $0 \leq \sigma \leq s - \half$ and $z_0 \in [-1, 0]$, $J = [z_0, 0]$. Consider $\theta$ solving (\ref{specialelliptic}). Then we can write
$$(\D_z - T_a)(\D_z - T_A)\tilde{\theta} = F_1 + F_2 + F_3$$
where for $i \geq 1$,
$$\|F_i\|_{Y^{\sigma}(J)} \leq \FF(\|\eta\|_{H^{s + \half}})((1 + \|\eta\|_{W^{r + \half, \infty}})\|\nabla_{x,z} \tilde{\theta}\|_{L^2(J; H^{\sigma})} + \|\nabla_{x,z} \tilde{\theta}\|_{U^{\sigma - s + \half+}(J)}),$$
\begin{align*}
\|w_{x_0,\lambda}&S_\mu F_i\|_{Y^{\sigma'}(J)} \leq \FF(\|\eta\|_{H^{s + \half}})((\lambda^{\half}\mu^{-\half} +\|\eta\|_{W^{r +\half,\infty}} + \mu^{\sigma'}\|\eta\|_{LS^{s - \sigma + \half}_{x_0, \xi_0, \lambda, \mu}})\|\nabla_{x, z} \tilde{\theta}\|_{U^{\sigma - s + \half+}(J)} \\
&+ (1 +\|\eta\|_{W^{r +\half,\infty}})(\mu^{\sigma'}\|\nabla_{x, z} w_{x_0, \lambda}S_{\xi_0, \lambda, \mu} \theta\|_{L^2(J;L^2)} + \|\nabla  w_{x_0,\mu}S_\mu\tilde{\theta}\|_{L^2(J;H^{\sigma'})} \\
&+ \|\nabla_{x, z} \tilde{\theta}\|_{L^2(J;H^\sigma)})).
\end{align*}
\end{prop}

\begin{proof}

Here we have used $F_1, F_2$, and $F_3$ to represent the errors arising from, respectively, the first order term on the left hand side of (\ref{eqn:flatelliptic}), the paralinearization errors, and the lower order terms from applying the symbolic calculus. 

We remark that in contrast to the Sobolev estimates on $F_i$ in \cite{alazard2014cauchy}, and similar to the estimates on $F_i$ in \cite{alazard2014strichartz}, \cite{ai2017low}, we use H\"older estimates whenever appropriate to obtain an absolute gain of $1/2$ derivatives (or the equivalent integral gain in $L_z^p$) on $F_i$ when compared to $\nabla^2 \tilde{\theta}$. Also contrast with Proposition \ref{localflatelliptic}, which obtains only a 1/4 derivative gain on $F_4$.

First, we observe that 
$$\|w_{x_0,\lambda}S_\mu F_i\|_{Y^{\sigma'}(J)} \lesssim \|w_{x_0,\mu}S_\mu F_i\|_{Y^{\sigma'}(J)}$$
so we may exchange $w_{x_0,\lambda}$ with $w_{x_0,\mu}$ when necessary.

We begin with $F_3$. Factor
\begin{align*}
(\D_z - T_a)(\D_z - T_A)\tilde{\theta} &= \D_z^2 \tilde{\theta} + T_\alpha \Delta \tilde{\theta} + T_\beta \cdot \nabla \D_z \tilde{\theta} + (T_aT_A - T_{aA})\tilde{\theta} - T_{\D_z A}\tilde{\theta}.
\end{align*}
For the first error term, by the commutator estimate (\ref{sobolevcommutator}) and the estimates on $a, A$ from Proposition \ref{aabds},
\begin{align*}
\|(T_aT_A - T_{aA})\tilde{\theta}\|_{L^1(J;H^\sigma)} &= \|(T_aT_A - T_{aA})S_{\geq 1/10}\tilde{\theta}\|_{L^1(J;H^\sigma)} \\
&\lesssim (\MM_{1}^{1,2}(a)\MM_0^1(A) + \MM_0^1(a)\MM_{1}^{1,2}(A)) \|S_{\geq 1/10}\tilde{\theta}\|_{L^2(J;H^{\sigma + 2 - 1})} \\
&\leq \FF(\|\eta\|_{H^{s + \half}})(1 + \|\eta\|_{W^{r + \half, \infty}}) \|\nabla \tilde{\theta}\|_{L^2(J;H^{\sigma})}.
\end{align*}
Further, for the local estimate, we have by the third local commutator estimate of Proposition \ref{localparadiff},
\begin{align*}
\|w_{x_0,\mu}S_\mu(T_aT_A - T_{aA})\tilde{\theta}\|_{L^1(J;H^{\sigma'})} &\lesssim (\MM_{1}^{1,2}(a)\MM_0^1(A) + \MM_0^1(a)\MM_{1}^{1,2}(A)) \\
&\quad \cdot (\|w_{x_0,\mu}S_\mu\tilde{\theta}\|_{L^2(J;H^{\sigma' + 2 - 1})} + \|\tilde{S}_\mu \tilde{\theta}\|_{L^2(J;H^{\sigma + 2 - 1})}) \\
&\leq \FF(\|\eta\|_{H^{s + \half}})(1 + \|\eta\|_{W^{r + \half, \infty}}) \\
&\quad \cdot (\|\nabla  w_{x_0,\mu}S_\mu\tilde{\theta}\|_{L^2(J;H^{\sigma'})} + \|\nabla \tilde{\theta}\|_{L^2(J;H^{\sigma})}).
\end{align*}

Similarly, using instead (\ref{ordernorm}),
\begin{align*}
\|T_{\D_z A}\tilde{\theta}\|_{L^1(J;H^{\sigma})} &\lesssim \MM_{0}^{1,2}(\D_z A) \|S_{> 1/10} \tilde{\theta}\|_{L^2(J; H^{\sigma + 1})} \\ 
&\leq \FF(\|\eta\|_{H^{s + \half}}) (1 + \|\eta\|_{W^{r + \half, \infty}}) \|\nabla \tilde{\theta}\|_{L^2(J; H^{\sigma})},
\end{align*}
and the second estimate of Proposition \ref{localparadiff} for the local estimate,
\begin{align*}
\|w_{x_0,\mu}S_\mu T_{\D_z A}\tilde{\theta}\|_{L^1(J;H^{\sigma'})} &\lesssim \MM_{0}^{1,2}(\D_z A) (\|w_{x_0,\mu}S_\mu  \tilde{\theta}\|_{L^2(J; H^{\sigma' + 1})} + \|\tilde{S}_\mu \tilde{\theta}\|_{L^2(J;H^{\sigma + 1})}) \\ 
&\leq \FF(\|\eta\|_{H^{s + \half}}) (1 + \|\eta\|_{W^{r + \half, \infty}}) \\
&\quad \cdot (\|\nabla w_{x_0,\mu}S_\mu \tilde{\theta}\|_{L^2(J; H^{\sigma'})} + \|\nabla \tilde{\theta}\|_{L^2(J;H^{\sigma})}).
\end{align*}

We hence have 
$$(\D_z - T_a)(\D_z - T_A)\tilde{\theta} = \D_z^2 \tilde{\theta} + T_\alpha \Delta \tilde{\theta} + T_\beta \cdot \nabla \D_z \tilde{\theta} + F_3$$
where 
\begin{align*}
\|F_3\|_{L^1(J;H^\sigma)} \leq \ &\FF(\|\eta\|_{H^{s + \half}}) (1 + \|\eta\|_{W^{r + \half, \infty}}) \|\nabla \tilde{\theta}\|_{L^2(J; H^\sigma)}, \\
\|w_{x_0,\lambda}S_\mu F_3\|_{L^1(J;H^{\sigma'})} \leq \ &\FF(\|\eta\|_{H^{s + \half}}) (1 + \|\eta\|_{W^{r + \half, \infty}}) \\
&\cdot(\|\nabla  w_{x_0,\mu}S_\mu\tilde{\theta}\|_{L^2(J;H^{\sigma'})} + \|\nabla \tilde{\theta}\|_{L^2(J;H^{\sigma})}).
\end{align*}

Next we estimate the error $F_2$ consisting of errors from paralinearization. Write 
$$(\D_z - T_a)(\D_z - T_A)\tilde{\theta} = \D_z^2 \tilde{\theta} + \alpha \Delta \tilde{\theta} + \beta \cdot \nabla \D_z \tilde{\theta} + (T_\alpha - \alpha) \Delta \tilde{\theta} + (T_\beta - \beta)\cdot \nabla \D_z \tilde{\theta} + F_3$$
and expand
$$(T_\alpha - \alpha) \Delta \tilde{\theta} + (T_\beta - \beta)\cdot \nabla \D_z \tilde{\theta} = -(T_{\Delta \tilde{\theta}} \alpha + R(\alpha, \Delta \tilde{\theta}) + T_{\nabla \D_z \tilde{\theta}} \cdot \beta + R(\beta, \nabla \D_z \tilde{\theta})).$$
By (\ref{sobolevparaproduct}) and (\ref{sobolevparaerror}),
\begin{align*}
\|T_{\Delta \tilde{\theta}} \alpha\|_{L^1(J;H^\sigma)} + \| R(\alpha, \Delta \tilde{\theta})\|_{L^1(J;H^\sigma)} &\lesssim \|\Delta \tilde{\theta}\|_{L^2(J;C_*^{\sigma - s})} \|\alpha\|_{L^2(J;H^{s})} \\
\|T_{\nabla \D_z \tilde{\theta}} \cdot \beta\|_{L^1(J;H^\sigma)} + \|R(\beta, \nabla \D_z \tilde{\theta})\|_{L^1(J;H^\sigma)} &\lesssim \|\nabla \D_z \tilde{\theta}\|_{L^2(J;C_*^{\sigma - s})} \|\beta\|_{L^2(J;H^{s})}.
\end{align*}
We estimate $\alpha$ and $\beta$ via Corollary \ref{cor:abgest} while
$$\|\Delta \tilde{\theta}\|_{C_*^{\sigma - s}} \lesssim \|\nabla \tilde{\theta} \|_{C_*^{\sigma - s + 1}}, \quad \|\nabla \D_z \tilde{\theta}\|_{C^{\sigma - s}_*} \lesssim \|\D_z \tilde{\theta}\|_{C^{\sigma - s + 1}_*}.$$

For the local estimates, we consider the $\alpha$ terms; the $\beta$ terms are similar. By Proposition \ref{localparaproduct},
$$\|w_{x_0,\mu}S_\mu T_{\Delta \tilde{\theta}} \alpha\|_{H^{\sigma'}} \lesssim \|\Delta \tilde{\theta}\|_{C_*^{\sigma - s}}(\|w_{x_0, \mu}S_\mu \alpha\|_{H^{s'}} + \|\alpha\|_{H^s}).$$
After integrating in $z$ and using Cauchy-Schwarz, we may use Corollary \ref{localabgest} and Corollary \ref{cor:abgest} respectively to estimate the $\alpha$ terms.

Next we consider $w_{x_0,\lambda}S_\mu R(\alpha, \Delta \tilde{\theta})$. By Proposition \ref{bilinear},
\begin{align*}
\|w_{x_0, \lambda}S_\mu R(\alpha, \Delta \tilde{\theta})\|_{H^{\sigma'}} &\lesssim \mu^{\sigma'}\|w_{x_0, \lambda}S_\mu R(\alpha, \Delta \tilde{\theta})\|_{L^2} \\
&\lesssim \mu^{\sigma'}(\|\alpha\|_{LS^{s - \sigma}_{x_0, \xi_0, \lambda, \mu}}\|\Delta \tilde{\theta}\|_{C_*^{\sigma - s+}} + \|S_\lambda \alpha\|_{C_*^{r}}\|w_{x_0, \lambda}S_{\xi_0, \lambda, \mu}\Delta \tilde{\theta}\|_{H^{-1}}).
\end{align*}
We use Corollary \ref{abgls} and Corollary \ref{cor:abgest} to estimate the $\alpha$ terms, concluding
\begin{align*}
\|w_{x_0, \lambda}&S_\mu R(\alpha, \Delta \tilde{\theta})\|_{L^1(J;H^{\sigma'})}\\
\leq \ &\FF(\|\eta\|_{H^{s + \half}})((\lambda^{\half}\mu^{-\half} +\|\eta\|_{W^{r +\half,\infty}} + \mu^{\sigma'}\|\eta\|_{LS^{s - \sigma + \half}_{x_0, \xi_0, \lambda, \mu}})\|\nabla \tilde{\theta}\|_{L^2(J;C_*^{\sigma - s + 1 +})}\\
&+ (1 +\|\eta\|_{W^{r +\half,\infty}})\mu^{\sigma'}(\|\nabla w_{x_0, \lambda}S_{\xi_0, \lambda, \mu} \tilde{\theta}\|_{L^2(J;L^2)} + \|S_{\xi_0, \lambda, \mu} \tilde{\theta}\|_{L^2(J;H^{1-\frac{1}{4}})})).
\end{align*}

We hence have 
$$(\D_z - T_a)(\D_z - T_A)\tilde{\theta} = \D_z^2 \tilde{\theta} + \alpha \Delta \tilde{\theta} + \beta \cdot \nabla \D_z \tilde{\theta} + F_2 + F_3 = F_0 + \gamma\D_z \tilde{\theta} + F_2 + F_3$$
where 
\begin{align*}
\|F_2\|_{L^1(J;H^\sigma)} &\leq \FF(\|\eta\|_{H^{s + \half}})\|\nabla_{x, z} \tilde{\theta}\|_{L^2(J; C_*^{\sigma - s + 1})},
\end{align*}
\begin{align*}
\|w_{x_0,\mu} &S_\mu F_2\|_{L^1(J;H^{\sigma'})} \\
\leq \ &\FF(\|\eta\|_{H^{s + \half}})((\lambda^{\half}\mu^{-\half} +\|\eta\|_{W^{r +\half,\infty}} + \mu^{\sigma'}\|\eta\|_{LS^{s - \sigma + \half}_{x_0, \xi_0, \lambda, \mu}})\|\nabla \tilde{\theta}\|_{L^2(J;C_*^{\sigma - s + 1 +})}\\
&+ (1 +\|\eta\|_{W^{r +\half,\infty}})(\mu^{\sigma'}\|\nabla w_{x_0, \lambda}S_{\xi_0, \lambda, \mu} \tilde{\theta}\|_{L^2(J;L^2)} + \|\nabla \tilde{\theta}\|_{L^2(J;H^\sigma)})).
\end{align*}

Lastly, we estimate the first order term $F_1 = \gamma\D_z \tilde{\theta}$. We decompose into paraproducts,
$$\gamma \D_z \tilde{\theta} = T_{\gamma} \D_z \tilde{\theta} + T_{\D_z \tilde{\theta}} \gamma + R(\gamma, \D_z \tilde{\theta}).$$
We have by (\ref{sobolevparaproduct}), (\ref{sobolevparaerror}), and Corollary \ref{cor:abgest},
\begin{align*}
\|T_{\gamma} \D_z \tilde{\theta} \|_{L^1(J; H^{\sigma})} + \|R(\gamma, \D_z \tilde{\theta})\|_{L^1(J; H^\sigma)} &\lesssim \|\gamma\|_{L^2(J; C_*^{r - 1})}\|\D_z \tilde{\theta}\|_{L^2(J;H^\sigma)}  \\
&\leq \FF(\|\eta\|_{H^{s + \half}})(1 +  \|\eta\|_{W^{r + \half, \infty}})\|\D_z \tilde{\theta}\|_{L^2(J; H^\sigma)}.
\end{align*}
Similarly,
\begin{align*}
\|T_{\D_z \tilde{\theta}} \gamma \|_{L^2(J; H^{\sigma - \half})} &\lesssim \|\gamma\|_{L^2(J; H^{s - 1})}\|\D_z \tilde{\theta}\|_{L^\infty(J; C_*^{\sigma - s + \half+})} \leq \FF(\|\eta\|_{H^{s + \half}})\|\D_z \tilde{\theta}\|_{L^\infty(J; C_*^{\sigma - s + \half+})}.
\end{align*}

For the local estimates, by Proposition \ref{localparaproduct},
\begin{align*}
\|w_{x_0,\mu}S_\mu T_{\gamma} \D_z \tilde{\theta} \|_{H^{\sigma'}} &\lesssim \|\gamma\|_{C_*^{r - 1}}(\|w_{x_0, \mu}S_\mu \D_z \tilde{\theta}\|_{H^{\sigma'}} + \|\D_z \tilde{\theta}\|_{H^{\sigma}}), \\
\|w_{x_0,\mu}S_\mu T_{\D_z \tilde{\theta}} \gamma \|_{H^{\sigma' - \half}} &\lesssim \|\D_z \tilde{\theta}\|_{C_*^{\sigma - s + \half+}}(\|w_{x_0, \mu}S_\mu \gamma\|_{H^{s' - 1}} + \|\gamma\|_{H^{s - 1}}).
\end{align*}
After integrating in $z$ and using Cauchy-Schwarz for the first estimate, and simply integrating in $L^2_z$ for the second, we use Corollary \ref{localabgest} and Corollary \ref{cor:abgest} to estimate the $\gamma$ terms.

Lastly, by Proposition \ref{bilinear},
\begin{align*}
\|w_{x_0, \lambda}S_\mu R(\gamma, \D_z \tilde{\theta})\|_{H^{\sigma'}} &\lesssim \mu^{\sigma'}\|w_{x_0, \lambda}S_\mu R(\gamma, \D_z \tilde{\theta})\|_{L^2} \\
&\lesssim \mu^{\sigma'}(\|\gamma\|_{LS^{s - \sigma - 1}_{x_0, \xi_0, \lambda, \mu}}\|\D_z \tilde{\theta}\|_{C_*^{\sigma - s + 1+}} + \|S_\lambda \gamma\|_{C_*^{r - 1}}\|w_{x_0, \lambda}S_{\xi_0, \lambda, \mu}\D_z \tilde{\theta}\|_{L^2}).
\end{align*}
We use Corollary \ref{abgls} and Corollary \ref{cor:abgest} to estimate the $\gamma$ terms, concluding
\begin{align*}
\|w_{x_0, \lambda}&S_\mu R(\gamma, \D_z \tilde{\theta})\|_{L^1(J;H^{\sigma'})}\\
\leq \ &\FF(\|\eta\|_{H^{s + \half}})((\lambda^{\half}\mu^{-\half} +\|\eta\|_{W^{r +\half,\infty}} + \mu^{\sigma'}\|\eta\|_{LS^{s - \sigma + \half}_{x_0, \xi_0, \lambda, \mu}})\|\D_z \tilde{\theta}\|_{L^2(J;C_*^{\sigma - s + 1+})}\\
&+ (1 +\|\eta\|_{W^{r +\half,\infty}})\mu^{\sigma'}\|\D_z w_{x_0, \lambda}S_{\xi_0, \lambda, \mu}\tilde{\theta}\|_{L^2(J;L^2)}).
\end{align*}
\end{proof}

\subsection{Paralinearization of $\D_z$}

Now that we have estimates on the inhomogeneous error for the factored equation (\ref{factoreqn}), we can apply parabolic estimates, which we now recall:
\begin{prop}\cite[Proposition 2.18]{alazard2014cauchy}\label{parabolics}
Let $r \in \R$, $\rho \in (0, 1)$, $J = [z_0, z_1] \subseteq \R$, and let $p \in \Gamma_\rho^1(J \times \R^d)$ satisfy
$$\text{Re} \, p(z, x, \xi) \geq c |\xi|$$
for some positive constant $c$. Then for any $f \in Y^r(J)$ and $v_0 \in H^r(\R^d)$, there exists $v \in X^r(J)$ solving the parabolic evolution equation
$$(\D_z + T_p)v = f, \qquad v|_{z = z_0} = v_0,$$
satisfying
$$\|v\|_{X^r(J)} \lesssim \|v_0\|_{H^r} + \|f\|_{Y^r(J)}$$
with a constant depending only on $r, \rho, c,$ and $\MM_\rho^1(p)$. 
\end{prop}

First, we make $\D_z \approx T_A$ precise in a (global) Sobolev norm, on which we will induct for the local counterpart:

\begin{lem}\label{wlems}
Let $0 \leq \sigma \leq s - \half$ and $-1 < z_1 < z_0 < 0$. Denote $J_0 = [z_0, 0], J_1 = [z_1, 0]$. Consider $\theta$ solving (\ref{specialelliptic}). Then
\begin{align*}
\|(\D_z - T_A)\tilde{\theta}\|_{X^{\sigma}(J_0)} \leq \ &\FF(\|\eta\|_{H^{s + \half}}, \|f\|_{H^{\sigma + \half}})(1 + \|\eta\|_{W^{r + \half, \infty}}+ \|\nabla_{x,z} \tilde{\theta}\|_{U^{\sigma - s + \half+}(J_1)}).
\end{align*}
\end{lem}

\begin{proof}

We insert a smooth vertical cutoff $\chi(z)$ vanishing on $[-1, z_1]$ with $\chi = 1$ on $J_0 = [z_0, 0] \subseteq (-1, 0]$, so that we have
$$(\D_z - T_a)\chi(\D_z - T_A)\tilde{\theta} = \chi(F_1 + F_2 + F_3) + \chi'(\D_z - T_A)\tilde{\theta} =: F'.$$
We estimate $\chi'(\D_z - T_A)\tilde{\theta}$ directly, obtaining
\begin{align*}
\|\chi'(\D_z - T_A)\tilde{\theta}\|_{L^1(J_1;H^{\sigma})} &\leq \FF(\|\eta\|_{H^{s + \half}}) \|\nabla_{x, z} \tilde{\theta}\|_{L^2(J_1;H^{\sigma})}.
\end{align*}
Combining this estimate with the Sobolev estimates of Proposition \ref{inhomogbd}, we have
\begin{align*}
\|F'\|_{Y^{\sigma}(J_1)} \leq \FF(\|\eta\|_{H^{s + \half}})((1 + \|\eta\|_{W^{r + \half, \infty}})\|\nabla_{x,z} \tilde{\theta}\|_{L^2(J_1; H^{\sigma})} + \|\nabla_{x,z} \tilde{\theta}\|_{U^{\sigma - s + \half+}(J_1)}).
\end{align*}
By Proposition \ref{originalcoordestsobolev} with $\sigma - \half$ in the place of $\sigma$,
$$\|\nabla_{x,z} \tilde{\theta}\|_{L^2(J_1; H^{\sigma})} \leq \FF(\|\eta\|_{H^{s + \half}})\|f\|_{H^{\sigma + \half}}.$$

Lastly, applying the Sobolev parabolic estimate of Proposition \ref{parabolics}, we obtain the result.
\end{proof}

We are now ready to establish a local Sobolev counterpart to Lemma \ref{wlems}:

\begin{lem}\label{wlemslocal}
Let $0 \leq \sigma \leq s - \half$ and $z_0 \in (-1, 0]$, $J = [z_0, 0]$. Consider $\theta$ solving (\ref{specialelliptic}). Then
\begin{align*}
\|w_{x_0, \lambda}&S_\mu(\D_z - T_A)\tilde{\theta}\|_{X^{\sigma'}(J)} \leq \FF(\|\eta\|_{H^{s + \half}}, \|f\|_{H^{\sigma + \half}}, \|f\|_{H^{1+}})(1 + \|\eta\|_{W^{r + \half, \infty}} + \|f\|_{C_*^{r}}) \\
&\cdot(\lambda^{\half}\mu^{-\half} +\|\eta\|_{W^{r +\half,\infty}} + \mu^{\sigma'}\|\eta\|_{LS^{s - \sigma + \half}_{x_0, \xi_0, \lambda, \mu}} + \mu^{\sigma'}\|f\|_{LS^\half_{x_0, \xi_0, \lambda, \mu}}).
\end{align*}
\end{lem}

\begin{proof}

By Proposition \ref{inhomogbd}, we have
$$w_{x_0, \lambda}S_\mu (\D_z - T_a)(\D_z - T_A)\tilde{\theta} = w_{x_0, \lambda}S_\mu(F_1 + F_2 + F_3).$$

We will commute
$$[w_{x_0, \lambda}S_\mu, T_a](\D_z - T_A)\tilde{\theta}.$$
As in the proof of Lemma \ref{wlems}, we also insert a smooth vertical cutoff $\chi(z)$ vanishing on $[-1, z_1]$ with $\chi = 1$ on $J = [z_0, 0] \subseteq (-1, 0]$. Writing $W = (\D_z - T_A)\tilde{\theta}$, we then have
\begin{align}\label{localparaboliceqn}
(\D_z - T_a)\chi w_{x_0, \lambda}S_\mu W = \ &\chi w_{x_0, \lambda}S_\mu(F_1 + F_2 + F_3) + \chi [w_{x_0, \lambda}S_\mu, T_a]W \\
&+ \chi' w_{x_0, \lambda}S_\mu W =: F'. \nonumber
\end{align}

Our goal is to estimate $F' \in Y^{\sigma'}(J_1)$ to apply the parabolic estimate. First, we consider the commutator on the right hand side of (\ref{localparaboliceqn}). By Proposition \ref{localparadiff},
\begin{align*}
\|[w_{x_0, \lambda}  S_\mu, T_a] W \|_{H^{\sigma' - \half}} &\lesssim M_0^1(a)\lambda^{\frac{3}{8}} \mu^{-\frac{3}{8}}\|W\|_{H^{\sigma' - \half + 1 - \frac{1}{8}}} \leq \FF(\|\eta\|_{H^{s + \half}})\lambda^{\frac{3}{8}} \mu^{-\frac{3}{8}}\|W\|_{H^{\sigma + \half}}.
\end{align*}
Applying Lemma \ref{wlems} on $W$, we conclude
\begin{align*}
\|[w_{x_0, \lambda}  S_\mu, T_a] W \|_{L^2(J_1;H^{\sigma' - \half})} \leq \ &\FF(\|\eta\|_{H^{s + \half}}, \|f\|_{H^{\sigma + \half}})\lambda^{\frac{3}{8}} \mu^{-\frac{3}{8}} \\
&\cdot (1 + \|\eta\|_{W^{r + \half, \infty}}+ \|\nabla_{x,z} \tilde{\theta}\|_{U^{\sigma - s + \half+}(J_1)}).
\end{align*}

Likewise, for the term $\chi' w_{x_0, \lambda}S_\mu W$ on the right hand side of (\ref{localparaboliceqn}), we again apply Lemma \ref{wlems} (without using any of the localization), obtaining a bound with the same right hand side.

Collecting these estimates with the local Sobolev estimates of Proposition \ref{inhomogbd}, we have
\begin{align*}
\|F'\|_{Y^{\sigma'}(J_1)} \leq \ &\FF(\|\eta\|_{H^{s + \half}}, \|f\|_{H^{\sigma + \half}})(\lambda^{\frac{3}{8}} \mu^{-\frac{3}{8}}(1 + \|\eta\|_{W^{r + \half, \infty}}+ \|\nabla_{x,z} \tilde{\theta}\|_{U^{\sigma - s + \half+}(J_1)}) \\
&+ (\lambda^{\half}\mu^{-\half} +\|\eta\|_{W^{r +\half,\infty}} + \mu^{\sigma'}\|\eta\|_{LS^{s - \sigma + \half}_{x_0, \xi_0, \lambda, \mu}})\|\nabla_{x, z} \tilde{\theta}\|_{U^{\sigma - s + \half+}(J_1)} \\
&+ (1 +\|\eta\|_{W^{r +\half,\infty}})(\mu^{\sigma'}\|\nabla_{x, z} w_{x_0, \lambda}S_{\xi_0, \lambda, \mu} \tilde{\theta}\|_{L^2(J_1;L^2)} + \|\nabla  w_{x_0,\mu}S_\mu\tilde{\theta}\|_{L^2(J_1;H^{\sigma'})} \\
&\quad + \|\nabla_{x, z} \tilde{\theta}\|_{L^2(J_1;H^\sigma)})).
\end{align*}

On the left hand side, applying the parabolic estimate of Proposition \ref{parabolics} to (\ref{localparaboliceqn}) on $J_1$, we may exchange $\|F'\|$ with
$$\|\chi w_{x_0, \lambda}S_\mu(\D_z - T_A)\tilde{\theta}\|_{X^{\sigma'}(J)}$$
as desired. Here we may drop the $\chi$, as on on $J$, we have $\chi \equiv 1$.

We may bound the right hand side using Propositions \ref{originalcoordestsobolev} and \ref{localinductioncorinducted} to estimate the Sobolev norms, and Proposition \ref{oldholderellipticest} for the H\"older norms,
\begin{align*}
\FF(&\|\eta\|_{H^{s + \half}}, \|f\|_{H^{\sigma + \half}}, \|f\|_{H^{1+}})(\lambda^{\frac{3}{8}} \mu^{-\frac{3}{8}}(1 + \|\eta\|_{W^{r + \half, \infty}} + \|f\|_{C_*^{r}}) \\
&+ (\lambda^{\half}\mu^{-\half} +\|\eta\|_{W^{r +\half,\infty}} + \mu^{\sigma'}\|\eta\|_{LS^{s - \sigma + \half}_{x_0, \xi_0, \lambda, \mu}})(1 + \|\eta\|_{W^{r + \half, \infty}} + \|f\|_{C_*^{r}}) \\
&+ (1 +\|\eta\|_{W^{r +\half,\infty}})(1 +\|\eta\|_{W^{r +\half,\infty}} + \mu^{\sigma'}\|w_{x_0, \lambda}S_{\xi_0, \lambda, \mu} f\|_{H^{\half}} + \| w_{x_0,\mu}S_\mu f\|_{H^{\sigma' + \half}})).
\end{align*}
Further rearranging, and recalling the definition of $LS^\half_{x_0, \xi_0, \lambda, \mu}$, yields the desired estimate.
\end{proof}

\subsection{Local Sobolev paralinearization}

In this subsection we perform a paralinearization of the Dirichlet to Neumann map, measuring the error locally in $H^{s' - \half}$. For comparison, the error was measured (globally) in $H^{s - \half}$ in \cite{alazard2014cauchy}, \cite{alazard2014strichartz}.

For convenience, we record Lemma \ref{wlemslocal} with $\sigma = s - \half$, and Proposition \ref{oldholderellipticest} with $\sigma = 0+$:

\begin{cor}\label{wlemslocalinducted}
Let $z_0 \in (-1, 0]$, $J = [z_0, 0]$. Consider $\theta$ solving (\ref{specialelliptic}). Then
\begin{align*}
\|w_{x_0, \lambda}&S_\mu(\D_z - T_A)\tilde{\theta}\|_{X^{s' - \half}(J)} \leq \FF(\|\eta\|_{H^{s + \half}}, \|f\|_{H^{s}})(1 + \|\eta\|_{W^{r + \half, \infty}} + \|f\|_{C_*^{r}}) \\
&\cdot(\lambda^{\half}\mu^{-\half} +\|\eta\|_{W^{r +\half,\infty}} + \mu^{s' - \half}\|\eta\|_{LS^1_{x_0, \xi_0, \lambda, \mu}} + \mu^{s' - \half}\|f\|_{LS^\half_{x_0, \xi_0, \lambda, \mu}}),
\end{align*}
$$\|\nabla_{x,z} \tilde{\theta}\|_{L^\infty(J; C_*^{0+})} \leq \FF(\|\eta\|_{H^{s + \half}}, \|f\|_{H^s})(1 +  \|\eta\|_{W^{r + \half, \infty}} + \|f\|_{C_*^{r}}).$$
\end{cor}

Recall that we write $\Lambda$ for the principal symbol of the Dirichlet to Neumann map. 
\begin{prop}\label{localparalinearize}
Write
$$\Lambda(t, x, \xi) = \sqrt{(1 + |\nabla \eta|^2)|\xi|^2 - (\nabla \eta \cdot \xi)^2} = |\xi|.$$
Then
\begin{align*}
\|w_{x_0, \lambda}&S_\mu ( G(\eta) - T_\Lambda) f\|_{H^{s' - \half}}  \leq \FF(\|\eta\|_{H^{s + \half}}, \|f\|_{H^{s}})(1 + \|\eta\|_{W^{r + \half, \infty}} + \|f\|_{C_*^{r}}) \\
&\cdot(\lambda^{\half}\mu^{-\half} +\|\eta\|_{W^{r +\half,\infty}} + \mu^{s' - \half}\|\eta\|_{LS^1_{x_0, \xi_0, \lambda, \mu}} + \mu^{s' - \half}\|f\|_{LS^\half_{x_0, \xi_0, \lambda, \mu}}).
\end{align*}
\end{prop}

\begin{proof}
Recall
$$G(\eta)f = \left. \left(\frac{1 + |\nabla \rho|^2}{\D_z \rho} \D_z \tilde{\theta} - \nabla \rho \cdot \nabla \tilde{\theta}\right)\right|_{z = 0} =: \left. \left(\zeta \D_z \tilde{\theta} - \nabla \rho \cdot \nabla \tilde{\theta}\right)\right|_{z = 0}.$$

First we reduce to paraproducts. Write
$$\zeta \D_z \tilde{\theta} - \nabla \rho \cdot \nabla \tilde{\theta} = T_\zeta \D_z \tilde{\theta} - T_{\nabla \rho} \nabla \tilde{\theta} + T_{\D_z \tilde{\theta}} \zeta - T_{\nabla \tilde{\theta}} \cdot \nabla \rho + R(\zeta, \D_z \tilde{\theta}) - R(\nabla \rho, \nabla \tilde{\theta}).$$

First we estimate $T_{\D_z \tilde{\theta}} \zeta$ (note $T_{\nabla \tilde{\theta}} \cdot \nabla \rho$ is similar, with the same estimates). By Proposition \ref{localparaproduct},
\begin{align*}
\|w_{x_0, \lambda}S_\mu T_{\D_z \tilde{\theta}} \zeta\|_{H^{s' - \half}} \lesssim \|w_{x_0, \mu}S_\mu T_{\D_z \tilde{\theta}} \zeta\|_{H^{s' - \half}} &\lesssim \|\D_z \tilde{\theta}\|_{L^\infty} (\|w_{x_0, \mu}S_\mu \zeta\|_{H^{s' - \half}} + \|\zeta\|_{H^{s - \half}}).
\end{align*}
Use Corollary \ref{wlemslocalinducted} to estimate the $\D_z \tilde{\theta}$ term. Also note that $\zeta$ satisfies the same local estimate as $\alpha$ in Corollary \ref{localabgest} by using the same argument,
$$\|w_{x_0, \mu}S_\mu \zeta\|_{X^{s' - \half}(J)} \leq \FF(\|\eta\|_{H^{s + \half}}) (\lambda^\half \mu^{-\half} + \|\eta\|_{W^{r + \half, \infty}} + \|w_{x_0, \mu}S_\mu \eta\|_{H^{s' + \half}}).$$
Likewise, $\zeta$  satisfies the same estimate as $\alpha$ in Corollary \ref{cor:abgest}, and in particular $\zeta \in X^{s - \half}(J)$. Using the $L_z^\infty$ component of the $X^{s - \half}(J)$ norms, the left hand sides bound the corresponding terms evaluated at $z = 0$. 

Second, we consider the $R(\cdot, \cdot)$ terms. We again consider only the typical $R(\zeta, \D_z \tilde{\theta})$. By Proposition \ref{bilinear},
\begin{align*}
\|w_{x_0, \lambda}S_\mu R(\zeta, \D_z \tilde{\theta})\|_{H^{s' - \half}} &\lesssim \mu^{s' - \half}\|w_{x_0, \lambda}S_\mu R(\zeta, \D_z \tilde{\theta})\|_{L^2} \\
&\lesssim \mu^{s' - \half}(\|\zeta\|_{LS^{0}_{x_0, \xi_0, \lambda, \mu}}\|\D_z \tilde{\theta}\|_{C_*^{0+}} + \|S_\lambda \zeta\|_{C_*^{\half}}\|w_{x_0, \lambda}S_{\xi_0, \lambda, \mu}\D_z \tilde{\theta}\|_{H^{-\half}}).
\end{align*}
Again noting that $\zeta$ satisfies the same estimate as $\alpha$ in Corollaries \ref{abgls} and \ref{cor:abgest} by using the same arguments, and using Corollary \ref{wlemslocalinducted} for $\D_z \tilde{\theta}$, we conclude
\begin{align*}
\|w_{x_0, \lambda}&S_\mu R(\zeta, \D_z \tilde{\theta})\|_{L^\infty(J;H^{s' - \half})} \leq \FF(\|\eta\|_{H^{s + \half}}, \|f\|_{H^s}) \\
&\cdot((\lambda^{\half}\mu^{-\half} +\|\eta\|_{W^{r +\half,\infty}} + \mu^{s' - \half}\|\eta\|_{LS^1_{x_0, \xi_0, \lambda, \mu}})(1 +  \|\eta\|_{W^{r + \half, \infty}} + \|f\|_{C_*^{r}})\\
&+ (1 +\|\eta\|_{W^{r +\half,\infty}})\mu^{s' - \half}(\|\D_z w_{x_0, \lambda}S_{\xi_0, \lambda, \mu} \tilde{\theta}\|_{L^\infty(J;H^{-\half})})).
\end{align*}
Lastly, estimate using Proposition \ref{localinductioncorinducted},
\begin{align*}
\|\D_z &w_{x_0, \lambda}S_{\xi_0, \lambda, \mu} \tilde{\theta}\|_{L^\infty(J;H^{-\half})} \lesssim \lambda^{\half - s'}\|\D_z w_{x_0, \lambda}S_{\xi_0, \lambda, \mu} \tilde{\theta}\|_{L^\infty(J;H^{s' - 1})} \\
&\leq \lambda^{\half - s'}\FF(\|\eta\|_{H^{s + \half}}, \|f\|_{H^s})(\|w_{x_0, \lambda}S_{\xi_0, \lambda, \mu} f\|_{H^{s'}} + 1 + \|\eta\|_{W^{r + \half, \infty}}) \\
&\leq \FF(\|\eta\|_{H^{s + \half}}, \|f\|_{H^s})(\|w_{x_0, \lambda}S_{\xi_0, \lambda, \mu} f\|_{H^{\half}} + \lambda^{\half - s'}(1 + \|\eta\|_{W^{r + \half, \infty}}))
\end{align*}
which estimates $R(\zeta, \D_z \tilde{\theta})$ by the desired right hand side.


Third, we may replace the vertical derivative $\D_z \tilde{\theta}$ with $T_A \tilde{\theta}$ as a consequence of Corollary \ref{wlemslocalinducted}. Note that to handle the multiplication of the paralinearization error with $\zeta$, we use a straightforward generalization of Proposition \ref{localparaproduct}:
\begin{align*}
\|  w_{x_0, \lambda}  S_\mu T_\zeta (\D_z - T_A)\tilde{\theta} \|_{H^{s' - \half}} \lesssim \ &\|\nabla \eta\|_{L^\infty} (\|w_{x_0, \lambda}  S_\mu(\D_z - T_A)\tilde{\theta}\|_{H^{s' - \half}} \\
&+ \lambda^{\half} \mu^{-\half}\|(\D_z - T_A)\tilde{\theta}\|_{H^{s - \half}}).
\end{align*}
Here, the last term on the right may be estimated using Lemma \ref{wlems}, as in the proof of Lemma \ref{wlemslocal}. Thus 
$$G(\eta)f = (T_\zeta T_A \tilde{\theta} - T_{\nabla \rho} \cdot \nabla \tilde{\theta})|_{z = 0} + R,$$
with the error $R$ satisfying the desired estimate. 

Lastly, by Proposition \ref{localparadiff}, Propositions \ref{prop:diffeoest} and \ref{aabds}, and Proposition \ref{originalcoordestsobolev} and Corollary \ref{localinductioncorinducted},
\begin{align*}
\|w_{x_0, \mu}S_\mu(T_\zeta T_A - T_{\zeta A})\tilde{\theta}\|_{H^{s' - \half}} &\lesssim (M_{0}^0(\zeta)M_{\half}^1(A) + M_{\half}^0(\zeta)M_{0}^1(A))(\|w_{x_0, \mu}S_\mu\tilde{\theta}\|_{H^{s'}} + \|\tilde{S}_\mu \tilde{\theta}\|_{H^s}) \\
&\leq \FF(\|\eta\|_{H^{s + \half}}) (1 +\|\eta\|_{W^{r + \half, \infty}}) \\
&\quad \cdot (\|\nabla w_{x_0, \mu}S_\mu\tilde{\theta}\|_{H^{s' - 1}} + \|\nabla \tilde{\theta}\|_{H^{s - 1}}) \\
&\leq \FF(\|\eta\|_{H^{s + \half}}, \|f\|_{H^s})(1 +\|\eta\|_{W^{r + \half, \infty}}) \\
& \quad \cdot(1 +\|\eta\|_{W^{r + \half, \infty}} +\|w_{x_0, \mu}S_\mu f\|_{H^{s'}}).
\end{align*}

We thus may exchange $T_\zeta T_A \tilde{\theta}$ for $T_{\zeta A}\tilde{\theta}$ in the expression for $G(\eta)f$, with an error $R'$ satisfying the same estimates as $R$. We conclude, using that $\tilde{\theta}(0) = f$, 
$$G(\eta)f = T_{\zeta A - i\nabla \rho \cdot \xi}f + R + R'.$$
A routine computation shows that $(\zeta A - i\nabla \rho \cdot \xi)|_{z = 0} = \Lambda$ as desired.
\end{proof}

\section{Local Estimates on the Pressure and Taylor Coefficient}\label{sec:taylorlocal}

In this section, we establish local Sobolev estimates on the Taylor coefficient 
$$a = -(\D_y P)|_{y = \eta(t, x)}.$$
Recall that the pressure is given by
$$-P = \D_t \phi + \half |\nabla_{x, y} \phi|^2 + gy.$$

This immediately implies an identity involving a derivative along the velocity field $v = \nabla_{x, y} \phi$,
\begin{equation}\label{pressuredef}
(\D_t + v\cdot \nabla_{x, y})\D_y\phi = -\D_y P - g,
\end{equation}
as well as an elliptic equation,
\begin{equation}\label{pressureeqn}
\Delta_{x, y}P = -\nabla_{x, y}^2\phi \cdot \nabla_{x, y}^2\phi =: F, \qquad P|_{y = \eta(x)} = 0.
\end{equation}

\subsection{Elliptic estimates on the pressure}

First we recall standard elliptic estimates on the pressure $P$:

\begin{prop}\label{firstpressureest}
Let $z_0 \in (-1, 0]$, $J = [z_0, 0]$. Then
$$\|\nabla_{x,z} \tilde{P}\|_{X^{s - \half}(J)} \leq \FF(M(t)).$$
\end{prop}

\begin{proof}

Apply Proposition \ref{sobolevinductive} to (\ref{pressureeqn}) with $\sigma = s - \half$:
$$\|\nabla_{x,z} \tilde{P}\|_{X^{s - \half}(J_0)} \leq \FF(\|\eta\|_{H^{s + \half}})(\|\alpha \tilde{F}\|_{Y^{s - \half}(J_1)} + \|\nabla_{x,z} \tilde{P}\|_{X^{-\half}(J_1)}).$$
Then, using \cite[(4.22)]{alazard2014cauchy}, we have 
$$\|\nabla_{x,z} \tilde{P}\|_{X^{-\half}(J_1)} \leq \FF(M(t)).$$
Further, by \cite[Lemma 4.7]{alazard2014cauchy} (which one easily checks only requires $s > \frac{d}{2} + \half$),
$$\|\alpha F\|_{Y^{s  - \frac{1}{2}}(J_1)} \leq \FF(M(t)).$$
\end{proof}

\subsection{Local Elliptic estimates on the pressure}

Here we establish local Sobolev estimates on the pressure $P$. First, we apply Proposition \ref{inhomogbd} to $P$:
\begin{prop}\label{pinhomogbd}
Let $z_0 \in [-1, 0]$, $J = [z_0, 0]$. Then we can write
$$(\D_z - T_a)(\D_z - T_A)\tilde{P} = F_0 + F_1 + F_2 + F_3$$
where for $i \geq 1$,
\begin{align*}
\|F_i\|_{Y^{s - \half}(J)} &\leq \FF(M(t))(1 + \|\eta\|_{W^{r + \half, \infty}}),\\
\|w_{x_0,\lambda}S_\mu F_i\|_{Y^{s' - \half}(J)} &\leq \FF(M(t))(\lambda^{\half}\mu^{-\half} +\|\eta\|_{W^{r +\half,\infty}} + \mu^{s' - \half}\|\eta\|_{LS^1_{x_0, \xi_0, \lambda, \mu}}).
\end{align*}
\end{prop}

\begin{proof}
Apply Proposition \ref{inhomogbd} except with $F_0 \neq 0$ and $\sigma = s - \half$:
$$\|F_i\|_{Y^{s - \half}(J)} \leq \FF(\|\eta\|_{H^{s + \half}})((1 + \|\eta\|_{W^{r + \half, \infty}})\|\nabla_{x,z} \tilde{P}\|_{L^2(J; H^{s - \half})} + \|\nabla_{x,z} \tilde{P}\|_{U^{0+}(J)}),$$
\begin{align*}
\|w_{x_0,\lambda}&S_\mu F_i\|_{Y^{s' - \half}(J)} \leq \FF(\|\eta\|_{H^{s + \half}})((\lambda^{\half}\mu^{-\half} +\|\eta\|_{W^{r +\half,\infty}} + \mu^{s' - \half}\|\eta\|_{LS^1_{x_0, \xi_0, \lambda, \mu}})\|\nabla_{x, z} \tilde{P}\|_{U^{0+}(J)} \\
&+ (1 +\|\eta\|_{W^{r +\half,\infty}})(\mu^{s' - \half}\|\nabla_{x, z} w_{x_0, \lambda}S_{\xi_0, \lambda, \mu} P\|_{L^2(J;L^2)} + \|\nabla  w_{x_0,\mu}S_\mu\tilde{P}\|_{L^2(J;H^{s' - \half})} \\
& + \|\nabla_{x, z} \tilde{P}\|_{L^2(J;H^{s - \half})})).
\end{align*}
Then applying Sobolev embedding along with Proposition \ref{firstpressureest} to the right hand sides (without using the localization), one obtains the estimate.
\end{proof}

Next, we obtain the counterparts of Lemmas \ref{wlems} and \ref{wlemslocal}:

\begin{lem}\label{pwlems}
Let $z_0 \in (-1, 0]$, $J = [z_0, 0]$. Then
\begin{align*}
\|(\D_z - T_A)\tilde{P}\|_{X^{s - \half}(J_0)} \leq \FF(M(t)).
\end{align*}
\end{lem}

\begin{proof}
This is immediate from Proposition \ref{firstpressureest} by estimating $\D_z\tilde{P}$ and $T_A \tilde{P}$ separately. 
\end{proof}

\begin{lem}\label{pwlemslocal}
Let $z_0 \in (-1, 0]$, $J = [z_0, 0]$. Then
\begin{align*}
\|w_{x_0, \lambda}S_\mu(\D_z - T_A)\tilde{P}\|_{X^{s' - \half}(J)} \leq \ &\FF(M(t))(\lambda^{\half}\mu^{-\half} +\|\eta\|_{W^{r +\half,\infty}} + \mu^{s' - \half}\|\eta\|_{LS^1_{x_0, \xi_0, \lambda, \mu}}).
\end{align*}
\end{lem}

\begin{proof}
The proof is the same as that of Lemma \ref{wlemslocal}, except in place of the inhomogeneous bounds of Proposition \ref{inhomogbd}, we apply Proposition \ref{pinhomogbd}; in place of the global counterpart Lemma \ref{wlems}, apply Lemma \ref{pwlems}; in place of the elliptic estimates of Propositions \ref{originalcoordestsobolev} and \ref{localinductioncorinducted}, apply Proposition \ref{firstpressureest}. For the H\"older estimates, Sobolev embedding suffices.
\end{proof}

Finally, we perform a second parabolic estimate to obtain the full elliptic estimate:

\begin{prop}\label{localpressure}
Let $z_0 \in (-1, 0]$, $J = [z_0, 0]$. Then
\begin{align*}
\|\nabla_{x, z} w_{x_0, \lambda}S_\mu \tilde{P}\|_{X^{s' - \half}(J)} \leq \ &\FF(M(t))(\lambda^{\half}\mu^{-\half} +\|\eta\|_{W^{r +\half,\infty}} + \mu^{s' - \half}\|\eta\|_{LS^1_{x_0, \xi_0, \lambda, \mu}}).
\end{align*}
\end{prop}

\begin{proof}
First we estimate the commutator using Proposition \ref{localparadiff},
\begin{align*}
\|[w_{x_0, \lambda} S_\mu, T_A]\tilde{P} \|_{H^{s'}} &\lesssim M_0^1(A)\lambda^{\frac{3}{8}} \mu^{-\frac{3}{8}}\|\tilde{S}_\mu\tilde{P}\|_{H^{s' + 1 - \frac{1}{8}}} \leq \FF(\|\eta\|_{H^{s + \half}})\lambda^{\frac{3}{8}} \mu^{-\frac{3}{8}}\|\nabla \tilde{P}\|_{H^{s}}.
\end{align*}
Then integrating $L_z^2$ and applying Proposition \ref{firstpressureest}, we see the commutator satisfies the same estimate as $w_{x_0, \lambda}S_\mu(\D_z - T_A)\tilde{P} \in Y^{s' + \half}(J)$ obtained by Lemma \ref{pwlemslocal}.

We conclude 
\begin{equation}\label{paraboliceqnp}
(\D_z - T_A)w_{x_0, \lambda}S_\mu\tilde{P} = F'
\end{equation}
where 
$$\|F'\|_{Y^{s' + \half}(J)} \leq \FF(M(t))(\lambda^{\half}\mu^{-\half} +\|\eta\|_{W^{r +\half,\infty}} + \mu^{s' - \half}\|\eta\|_{LS^1_{x_0, \xi_0, \lambda, \mu}}).$$
Then applying the parabolic estimate of Proposition \ref{parabolics},
$$\|w_{x_0, \lambda}S_\mu\tilde{P}\|_{X^{s' + \half}(J)} \leq \FF(M(t))(\lambda^{\half}\mu^{-\half} +\|\eta\|_{W^{r +\half,\infty}} + \mu^{s' - \half}\|\eta\|_{LS^1_{x_0, \xi_0, \lambda, \mu}}).$$
We thus obtain the desired estimate on $\nabla w_{x_0, \lambda}S_\mu\tilde{P}$. Then using the equation (\ref{paraboliceqnp}) and applying our estimate on $F'$, we also obtain the estimate on $\D_z w_{x_0, \lambda}S_\mu\tilde{P}$. 
\end{proof}


\subsection{Local smoothing estimates on the Taylor coefficient}

To conclude this section, we state a local Sobolev estimate on the Taylor coefficient. This is almost immediate from the previous subsection, but we need to restore the change of coordinates:
\begin{cor}\label{taylorsmooth}
We have
\begin{align*}
\|w_{x_0, \lambda}S_\mu a\|_{H^{s' - \half}} \leq \ &\FF(M(t))(\lambda^{\half}\mu^{-\half} +\|\eta\|_{W^{r +\half,\infty}} + \mu^{s' - \half}\|\eta\|_{LS^1_{x_0, \xi_0, \lambda, \mu}}).
\end{align*}
\end{cor}

\begin{proof}
Observe that
$$\widetilde{\D_y \theta} = \frac{1}{\D_z \rho} \D_z \tilde{\theta}.$$
Thus, it suffices to apply the product estimate, Proposition \ref{localproduct}, combined with Proposition \ref{localpressure} and the reciprocal estimate Proposition \ref{recip} (it suffices to use Sobolev embedding on the $L^\infty$ estimates).
\end{proof}

\section{Integration Along the Hamilton Flow}\label{COVsec}

To motivate the results of this section, recall the symbol of the operator of our evolution is
$$H(t, x, \xi) = V_\lambda \xi +  \sqrt{a_\lambda |\xi|}$$
with associated Hamilton equations
\begin{equation*}\begin{cases}
\dot{x} = H_\xi(t, x, \xi) = V_\lambda + \half\sqrt{a_\lambda} |\xi|^{-\frac{3}{2}}\xi\\
\dot{\xi} = -H_x(t, x, \xi) = -\D_x V_\lambda \xi - \D_x \sqrt{a_\lambda} |\xi|^\half.
\end{cases}\end{equation*}
To construct a useful wave packet parametrix for the evolution $D_t + H(t, x, D)$, we require the Hamilton flow of $H$ to be bilipschitz. In turn, this relies on the regularity of $\D_x^2 H$. To circumvent this in our low regularity setting, we will write $\D_x^2 H$ as a derivative of some $F = F(t, x, \xi)$ along the flow $(x(t), \xi(t))$:
\begin{equation}\label{desiredint}
\D_x^2 H \approx (\D_t + \dot{x} \D_x + \dot{\xi} \D_\xi) F
\end{equation}
in a sense to be made precise.

By a straightforward computation,
$$\D_x^2 H = \D_x^2 V_\lambda \xi + \frac{1}{2}\left(-\half \frac{(\D_x a_\lambda)^2}{\sqrt{a_\lambda^3}} + \frac{\D_x^2 a_\lambda}{\sqrt{a_\lambda}} \right)|\xi|^\half.$$
Then (\ref{desiredint}) becomes (omitting ``lower order terms'' from both sides)
\begin{equation}\label{desiredint2}
\D_x^2 V_\lambda \xi + \frac{1}{2}\frac{\D_x^2 a_\lambda}{\sqrt{a_\lambda}}|\xi|^\half \approx \left(\D_t + V_\lambda \D_x + \half \sqrt{a_\lambda} |\xi|^{-\frac{3}{2}}\xi \D_x\right) F.
\end{equation}

We will see below that in fact, for $F_1 = F_1(t, x)$ to be determined,
\begin{equation}\label{dispint}
\D_x^2 V_\lambda\xi \approx (\D_t + V_\lambda \D_x) F_1 \xi,\quad \frac{\D_x^2 a_\lambda}{\sqrt{a_\lambda}}|\xi|^\half \approx (\sqrt{a_\lambda} |\xi|^{-\frac{3}{2}}\xi \D_x) F_1 \xi
\end{equation}
in a sense to be made precise.

\subsection{Vector field identities}\label{vfid}

First we recall identities involving the vector field $L = \D_t + V \cdot \nabla$.

Recall that the traces of the velocity field on the surface $(V, B)$ can be expressed directly in terms of $\eta, \psi$ via the formulas (\ref{BVformulas}). Further, as a simple consequence of these formulas and the first equation of the system (\ref{zak}),
\begin{equation}\label{etatoB}
L\eta = B.
\end{equation}
Also recall from \cite[Propositions 4.3, 4.5]{alazard2014cauchy},
\begin{equation}\label{LandTaylor}
\begin{cases}
LB = a - g,\\
LV = -a \nabla \eta, \\
L\nabla \eta = G(\eta) V + \nabla \eta G(\eta) B + \Gamma_x + \nabla \eta \Gamma_y,
\end{cases}
\end{equation}
\begin{equation}\label{structure}
G(\eta)B = -\nabla \cdot V - \Gamma_y.
\end{equation}
Here, $\Gamma_x, \Gamma_y$ arise only in the case of finite bottom, and is described with estimates in Appendix \ref{bottomsection}. In the remainder of the section, we will consider the case of infinite bottom, so that 
$$\Gamma_x = \Gamma_y = 0.$$
This assumption is only for convenience; the appropriate local estimates on $\Gamma_x, \Gamma_y$ may be obtained by using elliptic arguments similar to those in Section \ref{sec:taylorlocal}.

\subsection{Integrating the vector field}

Observe that two of the identities from (\ref{LandTaylor}) and (\ref{structure}) imply
\begin{equation}\label{vstructurefinal}
L\nabla \eta = (G(\eta) - (\nabla \eta) \text{div}) V,
\end{equation}
which is the basis for our integration of $\D_x^2 V_\lambda$ below. Recall we write $\Lambda$ for the principal symbol of the Dirichlet to Neumann map.

\begin{prop}\label{integrate}
We have
\begin{equation}\label{integratedeqn}
\D_x^2 V_\lambda = (\D_t + V_\lambda \cdot \nabla) T_{q^{-1}} \D_x^2 \nabla \eta_\lambda + G_V
\end{equation}
where $q = \Lambda - i\nabla \eta \xi$ is a symbol of order $1$, and $G_V$ satisfies
\begin{align*}
\|G_V\|_{L^\infty} &\leq \lambda^{\frac{1}{2} -}\FF(M(t))Z(t)^2, \\
\|w_{x_0, \lambda} S_\mu G_V\|_{H^{\half}} &\leq\lambda^{\frac{1}{2} -}\FF(M(t))Z(t)(Z(t) + \LL(t, x_0, \xi_0, \lambda)).
\end{align*}
\end{prop}

\begin{proof}
Starting from (\ref{vstructurefinal}), we incrementally paralinearize, apply the symbolic calculus, frequency truncate, and commute. First, we observe that 
$$\|w_{x_0,\lambda}S_\mu f\|_{H^\sigma} \lesssim \|w_{x_0,\mu}S_\mu f\|_{H^\sigma}$$
so we may exchange $w_{x_0,\lambda}$ with $w_{x_0,\mu}$ when necessary.

\emph{Step 1. Paralinearization.} Paralinearize the terms
\begin{align*}
G(\eta)V &= T_\Lambda V + R(\eta)V, \\
(\nabla \eta) \nabla \cdot V &=  T_{\nabla \eta} \nabla \cdot V + T_{\nabla \cdot V} \nabla \eta + R(\nabla \eta, \nabla \cdot V)
\end{align*}
of (\ref{vstructurefinal}). Then rearranging,
$$T_q V = (T_\Lambda - T_{\nabla \eta}\text{div}) V = L\nabla \eta - R(\eta)V  + T_{\nabla \cdot V} \nabla \eta + R(\nabla \eta, \nabla \cdot V).$$
We estimate the error terms on the right hand side. By (\ref{DNparalinear}) (appropriately sharpened for the $\eps$ gain), (\ref{ordernorm}), and (\ref{holderparaerror}) respectively,
\begin{align*}
\|R(\eta)V\|_{W^{\half +, \infty}} &\leq \FF(\|\eta\|_{H^{s + \half}}, \|V\|_{H^s})(1 + \|\eta\|_{W^{r + \half, \infty}})(1 + \|\eta\|_{W^{r + \half, \infty}}+ \|V\|_{W^{r, \infty}})  \\
\|T_{\nabla \cdot V} \nabla \eta \|_{W^{\half +, \infty}} &\lesssim \|V\|_{W^{r,\infty}} \|\eta \|_{W^{r + \half, \infty}} \\
\|R(\nabla \eta, \nabla \cdot V)\|_{W^{\half +, \infty}} &\lesssim \| \eta\|_{W^{r + \half, \infty}}\|V\|_{W^{r, \infty}}.
\end{align*}

For the local Sobolev estimates, by Proposition \ref{localparalinearize},
\begin{align*}
\|w_{x_0, \lambda}&S_\mu R(\eta)V\|_{H^{s' - \half}}  \leq \FF(M(t))Z(t)(\lambda^{\half}\mu^{-\half} + Z(t) + \lambda^{\half}\mu^{-\half} \LL(t, x_0, \xi_0, \lambda)).
\end{align*}

By Proposition \ref{localparaproduct},
\begin{align*}
\|w_{x_0, \mu} S_\mu T_{\nabla \cdot V} \nabla \eta \|_{H^{s' - \half}} \lesssim \|\nabla \cdot V\|_{L^\infty}( \|w_{x_0, \mu} S_\mu \nabla \eta \|_{H^{s' - \half}} + \|\nabla \eta\|_{H^{s - \half}}).
\end{align*}

Lastly, by Proposition \ref{bilinear}, 
\begin{align*}
\mu^{s' - \half}\|w_{x_0, \lambda}&S_\mu R(\nabla \eta, \nabla \cdot V)\|_{L^2} \\
&\lesssim \mu^{s' - \half}\|\nabla \eta\|_{LS^{0}_{x_0, \xi_0, \lambda, \mu}}\|\nabla \cdot V\|_{C_*^{0+}} + \|S_\lambda \nabla \eta\|_{C_*^{\half}}\mu^{s' - \half}\|w_{x_0, \lambda}S_{\xi_0, \lambda, \mu}\nabla \cdot V\|_{H^{-\half}} \\
&\lesssim Z(t)\lambda^{\half}\mu^{-\half}\LL(t, x_0, \xi_0, \lambda).
\end{align*}

We conclude
$$T_q V = L\nabla \eta + G_1$$
where 
\begin{align*}
\|G_1\|_{W^{\half +, \infty}} &\leq \FF(M(t))Z(t)(1 + \|\eta\|_{W^{r + \half, \infty}}),\\
\|w_{x_0, \lambda} S_\mu G_1\|_{H^{s' - \half}} &\leq \FF(M(t))Z(t)(\lambda^{\half}\mu^{-\half} + Z(t) + \lambda^{\half}\mu^{-\half}\LL(t, x_0, \xi_0, \lambda)).
\end{align*}
Following the above proof of the local Sobolev estimate, but using the standard Sobolev counterparts to all of the local Sobolev estimates, one also obtains the (global) Sobolev counterpart,
$$\|G_1\|_{H^{s - \half}} \leq \FF(M(t))Z(t).$$

\

\emph{Step 2. Inversion of $q$.} We write
$$V = T_{q^{-1}}L\nabla \eta + T_{q^{-1}}G_1 + (1 - T_{q^{-1}}T_q)V.$$

For the first error term $T_{q^{-1}}G_1$, $q^{-1}$ is a symbol of order $-1$ and is a smooth function of $\nabla \eta$, so that by (\ref{smoothholder}) and Sobolev embedding with $s - \half > \frac{d}{2}$,
\begin{equation}\label{qcontrol}
M_0^{-1}(q^{-1}) \leq \FF(\|\nabla \eta\|_{L^\infty}) \leq \FF(\|\nabla \eta\|_{H^{s - \half}})
\end{equation}
and similarly
$$M_{\half}^{-1}(q^{-1}) \leq \FF(\|\nabla \eta\|_{H^{s - \half}}) \|\nabla \eta\|_{C_*^\half}.$$
Then by (\ref{ordernorm}) and the estimate on $G_1$ in the previous step,
\begin{align*}
\|T_{q^{-1}}G_1\|_{W^{\frac{3}{2} +, \infty}} &\lesssim M_0^{-1}(q^{-1}) \|G_1\|_{W^{\half +, \infty}} \leq \FF(M(t))Z(t)(1 +  \|\eta\|_{W^{r + \half, \infty}}).
\end{align*}
For the local Sobolev counterpart, use instead Proposition \ref{localparadiff} with Step 1:
\begin{align*}
\| w_{x_0, \lambda}  S_\mu T_{q^{-1}} G_1 \|_{H^{s' + \frac{1}{2}}} &\lesssim M_0^{-1}(q^{-1})(\| w_{x_0, \lambda} S_\mu G_1 \|_{H^{s' - \frac{1}{2}}} + \lambda^{\frac{3}{8}} \mu^{-\frac{3}{8}}\| G_1\|_{H^{s' + \frac{1}{2} - 1 - \frac{1}{8}}}) \\
&\leq \FF(\|\eta\|_{H^{s + \half}})(\| w_{x_0, \lambda} S_\mu G_1 \|_{H^{s' - \frac{1}{2}}} + \lambda^{\frac{3}{8}} \mu^{-\frac{3}{8}}\| G_1\|_{H^{s - \frac{1}{2}}}) \\
&\leq \FF(M(t))Z(t)(\lambda^{\half}\mu^{-\half} + Z(t) + \lambda^{\half}\mu^{-\half}\LL(t, x_0, \xi_0, \lambda)).
\end{align*}

Using (\ref{holdercommutator}), we control the second error term by
\begin{align*}
\|(1 - T_{q^{-1}}T_q)V\|_{W^{\frac{3}{2} +, \infty}} &\lesssim \left(M_{\half}^{-1}(q^{-1}) M_{0}^{1}(q) + M_{0}^{-1}(q^{-1}) M_{\half}^{1}(q)\right)\|V\|_{W^{r, \infty}} \\
& \leq \FF(\|\eta\|_{H^{s + \half}}) \|\eta\|_{W^{r + \half, \infty}}\|V\|_{W^{1, \infty}}.
\end{align*}
Likewise, using instead Proposition \ref{localparadiff} for the local Sobolev estimate,
\begin{align*}
\|w_{x_0, \mu} S_\mu(1 - T_{q^{-1}}T_q)V\|_{H^{s' + \half}} &\lesssim \left(M_{\half}^{-1}(q^{-1}) M_{0}^{1}(q) + M_{0}^{-1}(q^{-1}) M_{\half}^{1}(q)\right) \\
&\quad \cdot (\| w_{x_0, \mu} S_\mu V \|_{H^{s'}} + \| \tilde{S}_\mu V\|_{H^{s}}) \\
& \leq \FF(\|\eta\|_{H^{s + \half}}, \|V\|_{H^s}) \|\eta\|_{W^{r + \half, \infty}}\| w_{x_0, \mu} S_\mu V \|_{H^{s'}}.
\end{align*}

We conclude
$$ V = T_{q^{-1}}L\nabla \eta + G_2$$
with
\begin{align*}
\|G_2\|_{W^{\frac{3}{2} +, \infty}} \leq \ &\FF(M(t))Z(t)(1 + \|\eta\|_{W^{r + \half, \infty}}), \\
\|w_{x_0, \lambda} S_\mu G_2\|_{H^{s' + \half}} \leq \ &\FF(M(t))Z(t)(\lambda^{\half}\mu^{-\half} + Z(t) + \lambda^{\half}\mu^{-\half}\LL(t, x_0, \xi_0, \lambda)).
\end{align*}
As before, it is easy to follow the local Sobolev proof to show additionally
$$\|G_2\|_{H^{s + \half}} \leq \FF(M(t))Z(t).$$

\

\emph{Step 3. Frequency truncation and differentiation.} Applying $\D_x^2 S_{\leq c_1\lambda}$ to both sides of our identity, we have
$$\D_x^2 V_\lambda = T_{q^{-1}}\D_x^2 S_{\leq c_1\lambda} L\nabla \eta + \D_x^2 S_{\leq c_1\lambda}G_2 + [\D_x^2 S_{\leq c_1\lambda}, T_{q^{-1}}] L\nabla \eta.$$

For the first error term, we have
\begin{align*}
\|\D_x^2 S_{\leq c_1\lambda}G_2\|_{L^\infty} &\lesssim \lambda^{\frac{1}{2} -}\|G_2\|_{W^{\frac{3}{2} +, \infty}} \lesssim \lambda^{\frac{1}{2} -}\FF(M(t))Z(t)^2,\\
\|w_{x_0, \lambda} S_\mu \D_x^2 G_2\|_{H^\half} &\lesssim \mu^{\frac{1}{2} -}\|w_{x_0, \lambda} S_\mu G_2\|_{H^{s' + \half}} \lesssim \lambda^{\frac{1}{2} -}\FF(M(t))Z(t)(Z(t) + \LL(t, x_0, \xi_0, \lambda)).
\end{align*}

To estimate the second error term, first rewrite, using the definition of the paradifferential calculus,
$$[\D_x^2 S_{\leq c_1\lambda}, T_{q^{-1}}]S_{\leq c \lambda} L\nabla \eta.$$
Using (\ref{etatoB}), write
\begin{equation}\label{dstructure}
L\nabla \eta = \nabla L \eta - \nabla V \cdot \nabla \eta = \nabla B - \nabla V \cdot \nabla \eta.
\end{equation}
Then by (\ref{holdercommutator}) and (\ref{qcontrol}),
\begin{align*}
\|[\D_x^2 S_{\leq c_1\lambda}, T_{q^{-1}}] S_{\leq c\lambda}L\nabla \eta\|_{L^\infty} &\lesssim \lambda^{\frac{1}{2} - }\left(M_{\half}^{-1}(q^{-1}) + M_{0}^{-1}(q^{-1})\right) \|\nabla B - \nabla V \cdot \nabla \eta\|_{C_*^r} \\
&\lesssim \lambda^{\frac{1}{2} - }\FF(\|\eta\|_{H^{s + \half}})\|\eta\|_{W^{r + \half, \infty}} \left(\| B \|_{C_*^r} + \|\nabla \eta \|_{L^\infty}\| \D V\|_{C_*^r} \right) \\
&\leq \lambda^{\frac{1}{2} -}\FF(\|\eta\|_{H^{s + \half}}) \|\eta\|_{W^{r + \half, \infty}}\|(V, B)\|_{C_*^r}.
\end{align*}
For the local Sobolev counterpart, use instead Proposition \ref{localparadiff}:
\begin{align*}
\|w_{x_0, \mu}S_\mu &[\D_x^2 S_{\leq c_1\lambda}, T_{q^{-1}}]L\nabla \eta\|_{H^\half} \\
&\lesssim \lambda^{\frac{1}{2} - }\left(M_{\half}^{-1}(q^{-1}) + M_{0}^{-1}(q^{-1})\right) \\
&\quad \cdot (\| w_{x_0, \mu} S_\mu (\nabla B - \nabla V \cdot \nabla \eta) \|_{H^{s' - 1}} + \| \tilde{S}_\mu (\nabla B - \nabla V \cdot \nabla \eta)\|_{H^{s - 1}}) \\
&\lesssim \lambda^{\frac{1}{2} - }\FF(\|\eta\|_{H^{s + \half}})\|\eta\|_{W^{r + \half, \infty}} \\
&\quad \cdot (\LL(t, x_0, \xi_0, \lambda) + \| w_{x_0, \mu} S_\mu (\nabla V \cdot \nabla \eta) \|_{H^{s' - 1}} + M(t)).
\end{align*}
For the remaining middle term, using Corollary \ref{localproduct},
\begin{align*}
\| w_{x_0, \mu} S_\mu (\nabla V \cdot \nabla \eta) \|_{H^{s' - 1}} \lesssim \ &(\|w_{x_0, \mu} \tilde{S}_\mu \nabla V \|_{H^{s' - 1}} + \|\nabla V \|_{H^{s - 1}})\|\nabla \eta\|_{L^\infty} \\
&+ \|\nabla V\|_{L^\infty}(\|w_{x_0, \mu} \tilde{S}_\mu \nabla \eta\|_{H^{s' - 1}} + \|\nabla \eta\|_{H^{s - 1}}) \\
&+ \|\nabla \eta\|_{C_*^{\frac{1}{8}}}\|\nabla V\|_{H^{s - 1}} \\
\leq \ &\FF(M(t))(Z(t) + \LL(t, x_0, \xi_0, \lambda)).
\end{align*}


We conclude
$$\D_x^2 V_\lambda = T_{q^{-1}}\D_x^2 S_{\leq c_1\lambda} L\nabla \eta + G_3$$
with 
\begin{align*}
\|G_3\|_{L^\infty} &\leq \lambda^{\frac{1}{2} -}\FF(M(t))Z(t)^2, \\
\|w_{x_0, \lambda} S_\mu G_3\|_{H^{\half}} &\leq\lambda^{\frac{1}{2} -}\FF(M(t))Z(t)(Z(t) + \LL(t, x_0, \xi_0, \lambda))
\end{align*}

\

\emph{Step 4. Vector field paralinearization.} We paralinearize the vector field in order to commute past it in the next step. Writing the paraproduct expansion
$$(V \cdot \nabla)\nabla \eta = (T_{V} \cdot \nabla)\nabla \eta + T_{\nabla (\nabla \eta)}\cdot V + R(V, \nabla(\nabla \eta)),$$
we have
$$\D_x^2 V_\lambda = T_{q^{-1}}\D_x^2 S_{\leq c_1\lambda}(\D_t + T_V \cdot \nabla)\nabla \eta + G_3 + T_{q^{-1}}\D_x^2 S_{\leq c_1\lambda}(R(V, \nabla(\nabla \eta))+ T_{\nabla (\nabla \eta)}\cdot V).$$
Then by (\ref{ordernorm}), (\ref{holderparaerror}), and (\ref{holderparaproduct}),
\begin{align*}
\|T_{q^{-1}}\D_x^2 S_{\leq c_1\lambda}&(R(V, \nabla(\nabla \eta)) + T_{\nabla (\nabla \eta)}\cdot V)\|_{L^\infty} \\
&\lesssim M_0^{1}(q^{-1}\xi^2) \|S_{\leq c_1\lambda}(R(V, \nabla(\nabla \eta))+ T_{\nabla (\nabla \eta)}\cdot V)\|_{W^{1, \infty}} \\
&\lesssim \lambda^{\frac{1}{2} -}\FF(\|\eta\|_{H^{s + \half}})\|\nabla \eta\|_{L^\infty}\|V\|_{W^{r, \infty}}\| \nabla^2\eta\|_{C_*^{-\half}} \\
&\lesssim \lambda^{\frac{1}{2} -}\FF(\|\eta\|_{H^{s + \half}})\|\eta\|_{H^{s + \half}} \|V\|_{W^{r, \infty}} \|\eta\|_{W^{\frac{3}{2}, \infty}}.
\end{align*}

For the local Sobolev estimate, first consider $R(V, \nabla(\nabla \eta))$. By Proposition \ref{localparadiff},
\begin{align*}
\| w_{x_0, \lambda} & S_\mu T_{q^{-1}}\D_x^2 R(V, \nabla(\nabla \eta)) \|_{H^\half} \\
&\lesssim M_0^{-1}(q^{-1})(\| w_{x_0, \lambda} S_\mu \D_x^2 R(V, \nabla(\nabla \eta)) \|_{H^{\half - 1}} + \lambda^{\frac{3}{8}} \mu^{-\frac{3}{8}}\| \tilde{S}_\mu \D_x^2 R(V, \nabla(\nabla \eta))\|_{H^{\half - 1 - \frac{1}{8}}}) \\
&\leq \FF(M(t))(\| w_{x_0, \lambda} S_\mu \D_x^2 R(V, \nabla(\nabla \eta)) \|_{H^{-\half}} + \lambda^{\frac{3}{8}}\| \tilde{S}_\mu R(V, \nabla(\nabla \eta))\|_{H^{1}}) \\
&\leq \FF(M(t))(\| w_{x_0, \lambda} S_\mu \D_x^2 R(V, \nabla(\nabla \eta)) \|_{H^{-\half}} + \lambda^{\half-}Z(t)).
\end{align*}
For the remaining first term, first note that 
\begin{align*}
\| (\D_x w_{x_0, \lambda}) S_\mu \D_x R(V, \nabla(\nabla \eta)) \|_{H^{-\half}} &\lesssim \mu^{-\half} \lambda^{\frac{3}{4}} \|S_\mu R(V, \nabla(\nabla \eta)) \|_{H^1} \\
&\lesssim  \lambda^{\frac{3}{8}} \|S_\mu R(V, \nabla(\nabla \eta)) \|_{H^1} \\
&\lesssim \lambda^{\half -} \FF(M(t)) Z(t).
\end{align*}
We have a similar situation (better than needed) when both $\D_x^2$ fall on $w$, so it remains to consider
$$\| w_{x_0, \lambda} S_\mu R(V, \nabla(\nabla \eta)) \|_{H^{\frac{3}{2}}}.$$
Applying Proposition \ref{bilinear},
\begin{align*}
\|w_{x_0, \lambda} S_\mu R(V, \nabla(\nabla \eta)) \|_{L^2} \lesssim \ &\|V\|_{LS^\half_{x_0, \xi_0, \lambda, \mu}}\|\nabla^2 \eta\|_{C_*^{r - \frac{3}{2}}} + \|V\|_{C_*^{r}}\|w_{x_0, \lambda}S_{\xi_0, \lambda, \mu}\nabla^2 \eta\|_{H^{-1}}
\end{align*}
and thus
$$\| w_{x_0, \lambda} S_\mu R(V, \nabla(\nabla \eta)) \|_{H^{\frac{3}{2}}} \leq \FF(M(t))Z(t) \lambda^{\half -} \LL(t, x_0, \xi_0, \lambda).$$

For the local Sobolev estimate of $T_{\nabla (\nabla \eta)}\cdot V$, the analysis is similar and easier, using Proposition \ref{localparaproduct} in the place of Proposition \ref{bilinear}.

We conclude that we may replace $V$ with $T_V$, yielding, for $G_4$ satisfying the same estimates as $G_3$,
$$\D_x^2 V_\lambda = T_{q^{-1}}\D_x^2 S_{\leq c_1\lambda}(\D_t + T_V \cdot \nabla)\nabla \eta + G_4.$$

\

\emph{Step 5. Vector field commutator estimate.} By the definition of the paradifferential calculus, we may freely exchange $T_{q^{-1}}\D_x^2 S_{\leq c_1\lambda}(\D_t + T_V \cdot \nabla)\nabla\eta$ for 
$$T_{q^{-1}}\D_x^2 S_{\leq c_1\lambda}(\D_t + T_V \cdot \nabla)\nabla S_{\leq c\lambda}\eta.$$

In turn, applying Proposition \ref{com} with $m = 1$, $r = 0$ and $\eps = 1$, we may exchange this for 
$$LT_{q^{-1}}\D_x^2 \nabla \eta_\lambda$$
with an error bounded in $L^\infty$ by (using that $q$ is a smooth function of $\nabla \eta$, and (\ref{dstructure}))
\begin{align*}
M_0^{1}&(q^{-1}\xi^2 ) \|V\|_{W^{1, \infty}} \|\nabla S_{\leq c\lambda} \eta\|_{B^{1}_{\infty, 1}} + M_0^1(Lq^{-1} \xi^2) \|\nabla S_{\leq c\lambda}\eta\|_{W^{1, \infty}} \\
\lesssim& \ \lambda^{\frac{1}{2} -}\FF(\|\eta\|_{H^{s + \half}})(\|\eta \|_{W^{1, \infty}} \|V\|_{W^{1, \infty}}\|\eta\|_{B^{\frac{3}{2} +}_{\infty, 1}} + \|\nabla B - \nabla \eta \cdot \nabla V\|_{L^\infty} \| \eta\|_{W^{\frac{3}{2}, \infty}}) \\
\leq& \ \lambda^{\frac{1}{2} -}\FF(\|\eta\|_{H^{s + \half}})\|(V, B)\|_{W^{1, \infty}}\|\eta\|_{W^{r + \half, \infty}} \leq \lambda^{\frac{1}{2} -} \FF(M(t))Z(t)(1 +  \|\eta\|_{W^{r + \half, \infty}}).
\end{align*}

For the local Sobolev estimate, apply instead Proposition \ref{localparadiff} to estimate 
\begin{align*}
\|w_{x_0, \mu} S_\mu[T_{q^{-1}}\D_x^2 S_{\leq c_1\lambda}, &\D_t + T_V \cdot \nabla] \nabla\eta(t)\|_{H^\half} \\
&\lesssim (M_0^1(q^{-1}\xi^2)\|V(t)\|_{W^{r, \infty}} + M_0^1((\D_t + V \cdot \nabla)q^{-1}\xi^2)) \\
&\quad \cdot (\|w_{x_0, \mu}  S_\mu \nabla\eta(t)\|_{H^{\half + 1}} + \|\nabla\eta(t)\|_{H^{\half + 1 - \frac{1}{4}}}) \\
&\leq \FF( \|\eta\|_{H^{s + \half}}) Z(t) (\mu^2\|w_{x_0, \mu}  S_\mu \eta(t)\|_{H^{\half}} + \mu^{\half-}) \\
&\leq \FF( \|\eta\|_{H^{s + \half}}) Z(t) \lambda^{\half -}(\LL(t, x_0, \xi_0, \lambda) + 1).
\end{align*}
Note that we also need to restore $V$, from $T_V$. First consider the error arising from the balanced terms. Note that since we have one term truncated to low frequencies, we may use use Corollary \ref{bilinearsimple} instead of Proposition \ref{bilinear}:
\begin{align*}
\|w_{x_0, \lambda} S_\mu R(V, \nabla T_{q^{-1}}\D_x^2 \nabla \eta_\lambda)\|_{H^{0+}} &\lesssim (\|V\|_{LS^{s'}_{x_0, \lambda}} + \lambda^{\frac{3}{4}}\mu^{-1} \|S_{\leq c\lambda} V\|_{H^{s'}})\|\nabla T_{q^{-1}}\D_x^2 \nabla \eta_\lambda\|_{C_*^{-\frac{3}{2}}} \\
&\leq \FF(\|\eta\|_{H^{s + \half}})(\|V\|_{LS^{s'}_{x_0, \lambda}} + \lambda^{\half}\mu^{-\half} \| V\|_{H^{s}})\|\eta\|_{C_*^{\frac{3}{2}}} \\
&\leq \FF(M(t))Z(t)\lambda^{\half}\mu^{-\half}(\LL(t, x_0, \xi_0, \lambda) + 1).
\end{align*}
The high-low terms $T_{\nabla T_{q^{-1}}\D_x^2 \nabla \eta_\lambda}V$ are similarly estimated, using Proposition \ref{localparaproduct} in the place of Proposition \ref{bilinear}.

We conclude, for $G_5$ satisfying the same estimates as $G_4$,
$$\D_x^2 V_\lambda = LT_{q^{-1}}\D_x^2 \nabla \eta_\lambda + G_5.$$

\

\emph{Step 6. Vector field truncation.} Lastly, we frequency truncate the vector field, using (\ref{ordernorm}):
\begin{align*}
\|((S_{> c\lambda}V)\cdot \nabla) T_{q^{-1}}\D_x^2 \nabla \eta_\lambda\|_{L^\infty} &\lesssim \|S_{> c\lambda}V\|_{L^\infty}M_0^{1}(q^{-1}\xi^2)\| \nabla \eta_\lambda \|_{W^{2, \infty}} \\
&\leq \FF(\|\eta\|_{H^{s + \half}})\lambda^{1 -}\|S_{> c\lambda}V\|_{W^{\half, \infty}}\| \nabla \eta_\lambda \|_{W^{r- \half, \infty}}\\
&\lesssim \lambda^{\frac{1}{2} -}\FF(\|\eta\|_{H^{s + \half}})\|V\|_{W^{r,\infty}}\| \nabla \eta_\lambda \|_{W^{r - \half, \infty}}.
\end{align*}

For the local Sobolev counterpart, note that it suffices to consider
$$w_{x_0, \lambda} S_{\mu}(\tilde{S}_{c_1\lambda} V) \nabla T_{q^{-1}}\D_x^2 \nabla \eta_\lambda,$$
which was essentially estimated at the end of Step 5.
\end{proof}

\subsection{Integrating the dispersive term}

In this subsection we integrate $\D_x^2 a_\lambda$, the coefficient of the dispersive term in our symbol $H$.

\begin{prop}\label{hamintegrate}
We have
$$\frac{\D_x^2 a_\lambda}{\sqrt{a_\lambda}} = \sqrt{a_\lambda} \D_x T_{q^{-1}} \D_x^3 \eta_\lambda + (\D_t + V_\lambda \D_x) F_a + G_a$$
where $q = \Lambda - i \nabla \eta \xi$ is a symbol of order 1, and $F_a, G_a$ satisfy
\begin{align*}
\|F_a\|_{L^2} &\leq \lambda^{\frac{5}{8}-} \FF(M(t)), \\
\|F_a\|_{H^\half} &\leq \lambda^{\frac{5}{8}-} \FF(M(t)) Z(t), \\
\|w_{x_0, \lambda} S_\mu \sqrt{a_\lambda} \D_x F_a\|_{H^\half} &\leq \lambda^{\frac{3}{2} -}\FF(M(t))Z(t)(Z(t) + \LL(t, x_0, \xi_0, \lambda)), \\
\|G_a\|_{L^\infty} &\leq \lambda^{1 -} \FF(M(t)) Z(t)^2, \\
\|w_{x_0, \lambda} S_\mu G_a\|_{H^\half} &\leq \lambda^{1 -} \FF(M(t))Z(t)(Z(t) + \LL(t, x_0, \xi_0, \lambda)).
\end{align*}
\end{prop}

\begin{proof}
We will focus on the local Sobolev estimates, as the H\"older estimates are easier, using the appropriate H\"older counterparts. Throughout, we use identities from Section \ref{vfid}. Beginning with the identity
$$G(\eta)B = -\D_x V$$
and the paralinearization
$$G(\eta)B = |D|B + R(\eta)B,$$
write
$$B = |D|^{-1} |D|B = |D|^{-1} (G(\eta)B - R(\eta)B) = -|D|^{-1}\D_x V - |D|^{-1}R(\eta)B$$

Then apply $L$ to both sides, so that using the identities
$$LB = a - g,\quad LV = -a\nabla \eta,$$
we obtain
\begin{align*}
a - g &= -L|D|^{-1}\D_x V - L|D|^{-1}R(\eta)B \\
&= |D|^{-1}\D_x (a \nabla \eta) - L|D|^{-1}R(\eta)B - [V, |D|^{-1}\D_x] \D_x V.
\end{align*}
Applying $|D|$ to both sides, using the product rule, and applying a paradifferential decomposition,
\begin{equation*}
(|D| - T_{\nabla \eta}\D_x) a  = a \D_x^2 \eta + ((\nabla \eta) - T_{\nabla \eta})\D_x a - |D|L|D|^{-1}R(\eta)B - |D|[V, |D|^{-1}\D_x] \D_x V.
\end{equation*}

Observe that on the left hand side, we have $T_q a$. Lastly, we apply $\D_x^2 S_{\leq c_1\lambda} T_{q^{-1}}$ to both sides of the identity, to obtain
$$\D_x^2 a_\lambda = \D_x^2 S_{\leq c_1\lambda} T_{q^{-1}} a \D_x^2 \eta + E$$
where $E$ denotes error terms,
\begin{align}\label{tayloreqn}
E &= \D_x^2 S_{\leq c_1\lambda} T_{q^{-1}}(((\nabla \eta) - T_{\nabla \eta})\D_x a - |D|L|D|^{-1}R(\eta)B - |D|[V, |D|^{-1}\D_x] \D_x V) \\
&\quad + \D_x^2 S_{\leq c_1\lambda} (T_{q^{-1}}T_q - 1) a. \nonumber
\end{align}

Except for some commuting and rearrangement, it remains to estimate the errors $E$ in $L^\infty$ or locally with the weight $w_{x_0,\lambda}S_\mu$ in $H^\half$. As usual, we observe that 
$$\|w_{x_0,\lambda}S_\mu f\|_{H^\sigma} \lesssim \|w_{x_0,\mu}S_\mu f\|_{H^\sigma}$$
so we may exchange $w_{x_0,\lambda}$ with $w_{x_0,\mu}$ when necessary.

\ 

\emph{Step 1.} First we consider from (\ref{tayloreqn}) the paraproduct error with the Taylor coefficient, 
$$((\nabla \eta) - T_{\nabla \eta})\D_x a = T_{\D_x a} \D_x \eta + R(\D_x \eta, \D_x a).$$

From the first of these terms, we estimate
$$w_{x_0, \lambda} S_\mu \D_x^2 T_{q^{-1}} T_{\D_x a} \D_x \eta$$
using successive application of Propositions \ref{localparaproduct} and \ref{localparadiff} to commute the localization under the paraproducts and paradifferential operators. Note that we have a total of 2 derivatives applied to $\eta$ (using that $q^{-1}$ is of order $-1$), but we have an additional $1/2$ from measuring $E$ in $H^\half$, and an additional $\half$ from measuring $\D_x a \in C_*^{-\half}$. Using that $s' + \half > 2$ and the estimate 
\begin{align*}
\|w_{x_0, \mu} S_\mu \eta\|_{H^{3}} \lesssim \mu^{1-}\|w_{x_0, \mu} S_\mu \eta\|_{H^{s' + \half}} \leq \mu^{1-}\LL(t, x_0, \xi_0, \lambda),
\end{align*}
we see that we have a net loss of $\mu^{1-}$, which is better than needed. (Throughout, we will use the fact that we may commute $w_{x_0, \lambda} S_\mu$ under $\D_x^2 T_{q^{-1}}$ and thus effectively treat it as a single low frequency derivative.)

For the second term, note that the bilinear estimate of Proposition \ref{bilinear} is not directly applicable here, because we do not have an estimate on $a$ in $LS^\alpha_{x_0, \xi_0, \lambda, \mu}$. Thus, we instead integrate, writing
$$R(\D_x \eta, \D_x a) = R(\D_x \eta, \D_x LB) = R(\D_x \eta, L\D_x B + (\D_x V) \D_x B).$$
The term with $(\D_x V) \D_x B$ is balanced and hence easily estimated. Indeed, we have (note we do not use the spatial localization)
\begin{align*}
\|S_\mu \D_x^2 T_{q^{-1}} R(\D_x \eta, (\D_x V) \D_x B)\|_{H^\half} &\leq \mu^{1 -}\FF(M(t)) \|S_\mu R(\D_x \eta,  (\D_x V) \D_x B)\|_{H^{\half+}} \\
&\leq \mu^{1 -}\FF(M(t)) \|\D_x \eta\|_{H^{s - \half}}\|\D_x V\|_{L^\infty} \|\D_x B\|_{L^\infty}.
\end{align*}
It thus remains to consider
$$R(\D_x \eta, L\D_x B) = \sum_{\kappa} (S_\kappa \D_x \eta)(S_\kappa L \D_x B) = L R(\D_x \eta,\D_x B) + \sum_{\kappa} [(S_\kappa \D_x \eta)S_\kappa, L] \D_x B.$$

Consider the commutator. It is easy to commute $S_\kappa$ with $L$ as discussed above, as this balances the derivatives onto the vector field $V$ of $L$. Thus we consider 
$$\sum_{\kappa} [(S_\kappa \D_x \eta), L] S_\kappa \D_x B = -\sum_{\kappa} (LS_\kappa \D_x \eta)( S_\kappa \D_x B).$$
Again, we may commute $L$ with a single derivative, so that using $L\eta = B$, we may exchange this for
$$-\sum_{\kappa} (S_\kappa \D_x B)( S_\kappa \D_x B) = -R(\D_x B, \D_x B)$$
which we may estimate using Proposition \ref{localparadiff} to commute with the paradifferential operator and Proposition \ref{bilinear} to handle the $R(\cdot, \cdot)$ term:
\begin{align*}
\|w_{x_0, \lambda} S_\mu \D_x^2 T_{q^{-1}}R(\D_x B, \D_x B)\|_{H^\half} &\leq \FF(M(t))(\|w_{x_0, \lambda} S_\mu R(\D_x B, \D_x B)\|_{H^{\half + 1}}\\
&\quad + \lambda^{\frac{3}{8}} \mu^{-\frac{3}{8}}\|S_\mu R(\D_x B, \D_x B)\|_{H^{\half + 1 - \frac{1}{8}}}) \\
&\leq \FF(M(t))( \mu^{\frac{3}{2}}\|w_{x_0, \lambda} S_\mu R(\D_x B, \D_x B)\|_{L^2} + \lambda^{\frac{3}{8}}\mu^{\frac{5}{8}-}Z(t)) \\
&\leq \FF(M(t))( \mu^{\frac{3}{2}}\|B\|_{C_*^r} \| B\|_{LS_{x_0, \xi_0, \lambda, \mu}^1} + \lambda^{1 -}Z(t)) \\
&\leq \lambda^{1-}\FF(M(t))Z(t)(\LL(t, x_0, \xi_0, \lambda) + 1).
\end{align*}
We conclude that the first error term from (\ref{tayloreqn}) satisfies the appropriate estimates, except for a term
$$L R(\D_x \eta,\D_x B),$$
which we consider later as part of $F_a$.

\

\emph{Step 2.} Next, consider the second error term from (\ref{tayloreqn}),
$$|D|L|D|^{-1}R(\eta)B.$$
As discussed in Step 1, we may easily commute $L$ with a single derivative, so we are left with a term
$$LR(\eta) B$$
which we also consider later as part of $F_a$.

\

\emph{Step 3.} For the third error term from (\ref{tayloreqn}),
$$|D|[V, |D|^{-1}\D_x] \D_x V,$$
we reduce to paraproducts. First, we have by Proposition \ref{localparadiff},
\begin{align*}
\|w_{x_0, \mu} S_\mu [T_V, |D|^{-1}\D_x] \D_x V\|_{H^{\frac{3}{2}+}} &\lesssim \|V\|_{C^1}(\|w_{x_0, \mu} S_\mu V\|_{H^{\frac{3}{2}+}} + \|V\|_{H^{\frac{3}{2} - \frac{1}{4}+}}) \\
&\lesssim \|V\|_{C^1}(\|V\|_{LS_{x_0, \lambda}^{s'}} + \|V\|_{H^s})
\end{align*}
which suffices. The other low-high terms are similarly estimated using Proposition \ref{localparadiff}. On the other hand, a typical balanced-frequency term may be estimated by Proposition \ref{bilinear},
\begin{align*}
\|w_{x_0, \lambda} S_\mu R(V, |D|^{-1} \D_x^2 V)\|_{H^{\frac{3}{2}+}} &\lesssim \mu^{s'}\|V\|_{C^r}\|w_{x_0, \lambda} S_\mu V\|_{LS_{x_0, \xi_0, \lambda, \mu}^0}.
\end{align*}

\
\emph{Step 4.} Lastly, we estimate the fourth error term from (\ref{tayloreqn}), using the commutator estimate of Proposition \ref{localparadiff} with the local Taylor coefficient estimate of Corollary \ref{taylorsmooth},
\begin{align*}
\|w_{x_0, \lambda} S_\mu (T_{q^{-1}}T_q - 1) a\|_{H^\frac{5}{2}} &\lesssim (M_{\half}^1(q) M_0^{-1}(q^{-1}) + M_0^1(q) M_{\half}^{-1}(q^{-1})) \\
&\quad \cdot (\| w_{x_0, \lambda} S_\mu a \|_{H^{2}} + \lambda^{\frac{3}{8}} \mu^{-\frac{3}{8}}\| \tilde{S}_\mu a\|_{H^{2 - \frac{1}{8}}}) \\
&\leq \FF(\|\eta\|_{H^{s + \half}})Z(t) (\mu^{1-}\| w_{x_0, \lambda} S_\mu a \|_{H^{s' - \half}} + \lambda^{\frac{3}{8}} \mu^{\frac{5}{8}-}\| \tilde{S}_\mu a\|_{H^{s - \half}}) \\
&\leq \FF(M(t))Z(t)\lambda^{1-}(Z(t) + \LL(t, x_0, \xi_0, \lambda)).
\end{align*}

\

We conclude that 
\begin{equation}\label{dispeqn}
\D_x^2 a_\lambda = \D_x^2 S_{\leq c_1\lambda} T_{q^{-1}} a \D_x^2 \eta + \D_x^2 S_{\leq c_1\lambda} T_{q^{-1}} L(R(\D_x \eta,\D_x B) + R(\eta)B) + G_1
\end{equation}
where $G_1$ satisfies the desired estimate.

\

\emph{Step 5.} In the first term on the right hand side of (\ref{dispeqn}), we need to commute the Taylor coefficient $a$ to the front. To do so, we check below that we may exchange $a$ for $T_a$, after which we can commute and apply the commutator estimate of Proposition \ref{localparadiff} as usual. Having done so, we may also commute the paradifferential operators to the desired arrangement,
$$T_a\D_xT_{q^{-1}} \D_x^3 \eta_\lambda.$$

To estimate 
$$\D_x^2 S_{\leq c_1\lambda} T_{q^{-1}} (a - T_a)\D_x^2 \eta,$$
we apply the same steps as in Step 1, ending instead with the integrated term
$$LR(B, \D_x^2 \eta).$$

After commuting the paradifferential operators to the desired arrangement above, it remains to estimate 
$$(a_\lambda - T_a)\D_xT_{q^{-1}} \D_x^3 \eta_\lambda = (a_\lambda - T_{a_\lambda})\D_xT_{q^{-1}} \D_x^3 \eta_\lambda.$$
The low-high term uses Proposition \ref{localparaproduct} as usual, while for the balanced term
$$R(a_\lambda, \D_xT_{q^{-1}} \D_x^3 \eta_\lambda),$$
we apply Corollary \ref{bilinearsimple} since both factors are frequency truncated. Finally, multiplying both sides of (\ref{dispeqn}), after the above modifications, by $(a_\lambda)^{-\half}$, we obtain
\begin{equation}\label{dispeqn2}
\frac{\D_x^2 a_\lambda}{\sqrt{a_\lambda}} = \sqrt{a_\lambda}\D_xT_{q^{-1}} \D_x^3 \eta_\lambda + (a_\lambda)^{-\half} \D_x^2 S_{\leq c_1\lambda} T_{q^{-1}} L(R(\D_x \eta,\D_x B) + R(B, \D_x^2 \eta) + R(\eta)B) + G_2
\end{equation}
Here, observe that we may use
$$(a_\lambda)^{-\half} \in H^{s - \half} \subseteq H^{\half+}, \quad (a_\lambda)^{-\half} \in W^{\half, \infty}$$
with a paraproduct decomposition and Proposition \ref{localparaproduct} to easily estimate $G_2 = (a_\lambda)^{-\half}G_1$ by the same right hand side as $G_1$.

\

\emph{Step 6.} It remains to commute $L$ to the front of the second term in (\ref{dispeqn2}), 
$$(a_\lambda)^{-\half}\D_x^2 S_{\leq c_1\lambda} T_{q^{-1}} L(R(\D_x \eta,\D_x B) + R(B, \D_x^2 \eta) + R(\eta)B).$$
To do so, as in Step 5, we exchange $L$ for a paradifferential counterpart, $\D_t + T_V \D_x$. We remark that the discussion in Step 1 about commuting $L$ applies here, in that the computations in this step do not require spatial localization and thus are relatively straightforward. We may commute the vector field using \cite[Lemma 2.15]{alazard2014cauchy} to obtain
$$(a_\lambda)^{-\half}(\D_t + T_V \D_x)\D_x^2 S_{\leq c_1\lambda} T_{q^{-1}}(R(\D_x \eta,\D_x B) + R(B, \D_x^2 \eta) + R(\eta)B).$$
After restoring $L$ from $\D_t + T_V \D_x$, we have
$$(a_\lambda)^{-\half}L\D_x^2 S_{\leq c_1\lambda} T_{q^{-1}}(R(\D_x \eta,\D_x B) + R(B, \D_x^2 \eta) + R(\eta)B).$$
Lastly, we may commute $L$ to the front by using the Taylor coefficient estimates of Proposition \ref{taylorbd}, 
$$L(a_\lambda)^{-\half} \in L^\infty \cap H^{s - 1},$$
again not requiring spatial localization. Likewise, we may exchange $L$ for $\D_t + V_\lambda \D_x$. 

\

\emph{Step 7.} Lastly, we check that 
$$F_a := (a_\lambda)^{-\half}\D_x^2 S_{\leq c_1\lambda} T_{q^{-1}}(R(\D_x \eta,\D_x B) + R(B, \D_x^2 \eta) + R(\eta)B)$$
satisfies the desired estimates. Using the algebra property of $H^{\half+}$, as well as the tame (linear in $Z(t)$) paralinearization estimate of Proposition \ref{DNparalineartame},
\begin{align*}
\|(a_\lambda)^{-\half}\D_x^2 S_{\leq c_1\lambda} &T_{q^{-1}}(R(\D_x \eta,\D_x B) + R(B, \D_x^2 \eta) + R(\eta)B)\|_{H^\half} \\
&\lesssim \|(a_\lambda)^{-\half}\|_{H^{\half+}}\|\D_x^2 S_{\leq c_1\lambda} T_{q^{-1}}(R(\D_x \eta,\D_x B) + R(B, \D_x^2 \eta) + R(\eta)B)\|_{H^{\half+}} \\
&\lesssim \FF(M(t))\|S_{\leq c_1\lambda} T_{q^{-1}}(R(\D_x \eta,\D_x B) + R(B, \D_x^2 \eta) + R(\eta)B)\|_{H^{s' + 1}} \\
&\leq \lambda^{\frac{5}{8}-}\FF(M(t))\|R(\D_x \eta,\D_x B) + R(B, \D_x^2 \eta) + R(\eta)B\|_{H^{s - \half}} \\
&\leq \lambda^{\frac{5}{8}-}\FF(M(t))Z(t)
\end{align*}
as desired. The $L^2$ estimate is similar, using instead Proposition \ref{DNparalinearsupertame}.

We also estimate $w_{x_0, \lambda} S_\mu \sqrt{a_\lambda} \D_x F_a$. When the derivative falls on $(a_\lambda)^{-\half}$, we have the previous estimate except with an additional loss,
\begin{align*}
\|\D_x (a_\lambda)^{-\half}\|_{H^\half} \leq \lambda^{5/8}\FF(M(t))
\end{align*}
which is better than needed. When the derivative falls on the rest of the term, we have 
\begin{align*}
\|w_{x_0, \lambda} S_\mu \D_x^3 T_{q^{-1}}(R(\D_x \eta,\D_x B) + R(B, \D_x^2 \eta) + R(\eta)B)\|_{H^\half}.
\end{align*}
By applying Proposition \ref{localparadiff} to commute under the paradifferential operator, it remains to estimate
\begin{align*}
\|w_{x_0, \lambda} S_\mu (R(\D_x \eta,\D_x B) + R(B, \D_x^2 \eta) + R(\eta)B)\|_{H^{5/2}}.
\end{align*}
To the first quadratic term (the second is similar), we apply Proposition \ref{bilinear},
\begin{align*}
\|w_{x_0, \lambda} S_\mu R(\D_x \eta,\D_x B)\|_{H^{5/2}} &\lesssim \mu^{2}(\|\D_x \eta\|_{C_*^{r - \half}} \|B\|_{LS_{x_0, \xi_0, \lambda, \mu}^1} + \|\D_x B\|_{C_*^{r - 1}}\|\eta\|_{LS_{x_0, \xi_0, \lambda, \mu}^{\frac{3}{2}}}) \\
&\leq \lambda\mu^{\half - }Z(t)\LL(t, x_0, \xi_0, \lambda)
\end{align*}
as desired. Likewise, by Proposition \ref{localparalinearize}, 
\begin{align*}
\|w_{x_0, \lambda} S_\mu R(\eta)B\|_{H^{s' - \half}}  \leq \FF(M(t))Z(t)(\lambda^{\half}\mu^{-\half} +Z(t) + \lambda^\half \mu^{-\half}\LL(t, x_0, \xi_0, \lambda))
\end{align*}
so that
\begin{align*}
\|w_{x_0, \lambda} S_\mu R(\eta)B\|_{H^{5/2}} &\leq \mu^{\frac{3}{2} -}\FF(M(t))Z(t)(\lambda^{\half}\mu^{-\half} +Z(t) + \lambda^\half \mu^{-\half}\LL(t, x_0, \xi_0, \lambda)) \\
&\leq \lambda^{\frac{3}{2} -}\FF(M(t))Z(t)(Z(t) + \LL(t, x_0, \xi_0, \lambda)).
\end{align*}
\end{proof}

\subsection{Integrating the symbol}

Combining Propositions \ref{integrate} and \ref{hamintegrate}, we obtain the integration (\ref{desiredint2}) of $\D_x^2 H$:

\begin{cor}\label{structurecor}
Let 
\begin{align*}
F_1(t, x) &= T_{q^{-1}} \D_x^3 \eta_\lambda, \quad F_\half(t, x) = F_a
\end{align*}
where $F_a$ is given in the proof of Proposition \ref{hamintegrate}. Then we may write
\begin{equation}\label{structurecoreqn}
\D_x^2 H = (\D_t + H_\xi \D_x - H_x \D_\xi)(F_1\xi + F_\half|\xi|^\half) + G_1\xi + G_\half|\xi|^\half + G_0|\xi|^{-1} \xi
\end{equation}
where
\begin{align*}
\|F_\alpha\|_{L^2} &\leq \lambda^{\frac{9}{8} - \alpha - } \FF(M(t)),\\
\|F_\alpha\|_{L^\infty} &\leq \lambda^{\frac{3}{2} - \alpha - } \FF(M(t))Z(t),\\
\|w_{x_0, \lambda} S_\mu F_\alpha\|_{H^\half} &\leq \lambda^{\frac{3}{2} - \alpha - } \FF(M(t))(Z(t) + \LL(t, x_0, \xi_0, \lambda)), \\
\|G_\alpha\|_{L^\infty} &\leq \lambda^{\frac{3}{2} - \alpha - } \FF(M(t))Z(t)^2, \\
\|w_{x_0, \lambda} S_\mu G_\alpha\|_{H^\half} &\leq \lambda^{\frac{3}{2} - \alpha - }\FF(M(t))Z(t)(Z(t) + \LL(t, x_0, \xi_0, \lambda)).
\end{align*}
\end{cor}

\begin{proof}

It is easy to check that $F_1$ satisfies the desired estimates using Proposition \ref{localparadiff}. Note that $F_\half$ is even easier, using Sobolev embedding and not requiring the spatial localization. It remains to compute the errors $G_i$ of the integration.

First compute
\begin{align*}
H_\xi &= V_\lambda + \half\sqrt{a_\lambda} |\xi|^{-\frac{3}{2}}\xi, \\
H_x &= \D_x V_\lambda \xi + \D_x \sqrt{a_\lambda} |\xi|^\half, \\
\D_x^2 H &= \D_x^2 V_\lambda \xi + \frac{1}{2}\left(-\half \frac{(\D_x a_\lambda)^2}{\sqrt{a_\lambda^3}} + \frac{\D_x^2 a_\lambda}{\sqrt{a_\lambda}} \right)|\xi|^\half.
\end{align*}
Consider the three terms constituting $\D_x^2 H$. For the first and last, recall the identities of Propositions \ref{integrate} and \ref{hamintegrate}, 
\begin{align*}
\D_x^2 V_\lambda &= (\D_t + V_\lambda \D_x) F_1 + G_V \\
\frac{\D_x^2 a_\lambda}{\sqrt{a_\lambda}} &= \sqrt{a_\lambda} \D_x F_1 + (\D_t + V_\lambda \D_x)F_\half + G_a.
\end{align*}
Substituting into the expression for $\D_x^2 H$, we see that the two terms containing $F_1$ combine to form 
$$(\D_t + H_\xi \D_x) F_1\xi,$$
so that
$$\D_x^2 H = (\D_t + H_\xi \D_x) F_1\xi + G_V\xi + \frac{1}{2}\left(-\half \frac{(\D_x a_\lambda)^2}{\sqrt{a_\lambda^3}} + (\D_t + V_\lambda \D_x)F_\half + G_a \right)|\xi|^\half.$$
On the right hand side, putting aside for now the term $(\D_t + V_\lambda \D_x)F_\half|\xi|^\half$, and adding and subtracting $H_x \D_\xi F_1\xi$, we have
\begin{align*}
(\D_t + H_\xi \D_x - H_x \D_\xi) F_1\xi + (G_V + (\D_x V_\lambda)F_1)\xi &+ \frac{1}{2}\left(-\half \frac{(\D_x a_\lambda)^2}{\sqrt{a_\lambda^3}} + G_a + (\D_x \sqrt{a_\lambda})F_1\right)|\xi|^\half \\
=: \ &(\D_t + H_\xi \D_x - H_x \D_\xi) F_1\xi + G_1\xi + G_\half |\xi|^\half.
\end{align*}

It remains to check the estimates on $G_i$. Note that $G_V, G_a$ satisfy the desired estimates by Propositions \ref{integrate} and \ref{hamintegrate}. We consider the remaining terms, focusing on the local Sobolev estimates since the H\"older estimates are easier, using the appropriate H\"older counterparts. First, for
$$(\D_x V_\lambda) F_1,$$
we use a paraproduct decomposition, 
$$(\D_x V_\lambda) F_1 = T_{\D_x V_\lambda} F_1 + T_{F_1} \D_x V_\lambda + R(\D_x V_\lambda, F_1).$$
The first and second terms are straightforward, using Proposition \ref{localparaproduct} with the estimates on $F_1 \in L^\infty \cap H^{\half}$ of the proposition. The third term uses Corollary \ref{bilinearsimple}.

The remaining two terms, 
$$-\half \frac{(\D_x a_\lambda)^2}{\sqrt{a_\lambda^3}} + (\D_x \sqrt{a_\lambda})F_1$$
are similarly treated using a paraproduct decomposition (using also the Taylor coefficient bounds of Proposition \ref{taylorbd} and the Taylor sign condition).

\

We assess the term 
$$\half(\D_t + V_\lambda \D_x)F_\half|\xi|^\half$$
similarly, adding errors as necessary to the $G_i$ terms. However, to form
$$\half(\D_t + H_\xi \D_x)F_\half|\xi|^\half,$$
we simply add and subtract
$$\frac{1}{4} \sqrt{a_\lambda} |\xi|^{-\frac{3}{2}} \xi \D_x F_\half |\xi|^\half = \frac{1}{4} \sqrt{a_\lambda} \D_x F_\half |\xi|^{-1}\xi.$$
By the corresponding estimate of Proposition \ref{hamintegrate}, this term satisfies the appropriate estimate to be absorbed into $G_0$. Also note that inserting the term $H_x \D_\xi F_\half |\xi|^\half$ is easier than before, in particular not using spatial localization. 
\end{proof}


\section{Regularity of the Hamilton Flow}\label{sec:regham}

Recall we have the Hamiltonian
$$H(t, x, \xi) = V_\lambda \xi + \sqrt{a_\lambda |\xi|}.$$
In this section we discuss the regularity properties of the associated Hamilton flow.

\subsection{The linearized equations}

Recall the Hamilton equations (\ref{hamilton}) associated to $H$,
\begin{equation*}\begin{cases}
\dot{x}(t) = H_\xi(t, x(t), \xi(t))\\
\dot{\xi}(t) = -H_x(t, x(t), \xi(t))\\
(x(s), \xi(s)) = (x, \xi),
\end{cases}\end{equation*}
and that we denote the solution $(x(t), \xi(t))$ to (\ref{hamilton}), at time $t \in I$ with initial data $(x, \xi)$ at time $s \in I$, variously by 
$$(x^t_s(x, \xi), \xi^t_s(x, \xi)) = (x^t(x, \xi), \xi^t(x, \xi)) = (x^t_s, \xi^t_s) = (x^t, \xi^t).$$

Denoting
$$p(t, x, \xi) = \begin{pmatrix}
x^t(x, \xi) \\
\xi^t(x, \xi)
\end{pmatrix},$$
we have the linearization of the Hamilton equations (\ref{hamilton}):
\begin{equation}\label{eqn:linearizedham}
\frac{d}{dt} \D_{x} p(t) = \begin{pmatrix}
H_{\xi x}(t, x^t, \xi^t) & H_{\xi \xi}(t, x^t, \xi^t) \\
-H_{xx}(t, x^t, \xi^t) & -H_{x \xi}(t, x^t, \xi^t)
\end{pmatrix} \D_{x} p(t).
\end{equation}

By a heuristic computation using the uncertainty principle and the dispersion relation $|\xi|^\half$, the scale on which wave packets cohere is 
$$\delta x \approx \lambda^{-\frac{3}{4}}, \quad \delta \xi \approx \lambda^{\frac{3}{4}}, \quad \delta t \approx 1.$$
As a result, it is natural to define 
$$P(t, x, \xi) = \begin{pmatrix}
X^t(x, \xi) \\
\Xi^t(x, \xi)
\end{pmatrix}
= \begin{pmatrix}
\lambda^{\frac{3}{4}} x^t(x, \xi) \\
\lambda^{-\frac{3}{4}} \xi^t(x, \xi)
\end{pmatrix},$$
and with a slight abuse of notation,
$$\tilde{H}_{\xi \xi} = \lambda^{\frac{3}{2}}H_{\xi \xi}, \quad \tilde{H}_{x \xi} = H_{x \xi}, \quad \tilde{H}_{xx} = \lambda^{-\frac{3}{2}}H_{xx}.$$
Then we have
\begin{equation}\label{eqn:linearizedhamscaled}
\frac{d}{dt} \D_{x} P(t) = \begin{pmatrix}
\tilde{H}_{\xi x}(t, x^t, \xi^t) & \tilde{H}_{\xi \xi}(t, x^t, \xi^t) \\
-\tilde{H}_{xx}(t, x^t, \xi^t) & -\tilde{H}_{x \xi}(t, x^t, \xi^t)
\end{pmatrix} \D_{x} P(t).
\end{equation}

For the remainder of this subsection we estimate the entries of the matrix of the linearized Hamilton flow \ref{eqn:linearizedhamscaled}.

\begin{lem}\label{basicHbd}
On $\{|\xi| \in [\lambda/2, 2\lambda]\}$,
\begin{align*}
\|\tilde{H}_{\xi\xi}(t)\|_{L^\infty_x} &\leq \FF(M(t)), \\
\|\tilde{H}_{x\xi}(t)\|_{L^\infty_x} &\leq \FF(M(t))Z(t), \\
\|w_{x_0, \lambda}\tilde{H}_{\xi\xi}(t)\|_{H^{\half+}_x} &\leq \lambda^{0+}\FF(M(t)), \\
\|w_{x_0, \lambda}\tilde{H}_{x\xi}(t)\|_{H^{\half+}_x} &\leq \lambda^{0+}\FF(M(t))(Z(t) + \LL(t, x_0, \xi_0, \lambda)).
\end{align*}
\end{lem}

\begin{proof}
We have 
$$H_{\xi\xi} = -\frac{1}{4}\sqrt{a_\lambda} |\xi|^{-\frac{3}{2}}$$
from which the first estimate is immediate on $\{|\xi| \in [\lambda/2, 2\lambda]\}$, using Proposition \ref{taylorbd}. We also have
$$H_{x\xi} = \D_x V_\lambda + \frac{1}{2} (\D_x \sqrt{a_\lambda}) |\xi|^{-\frac{3}{2}}\xi.$$
Then the second bound is obtained using Proposition \ref{taylorbd}, the Taylor sign condition, frequency localization, and the assumption on $\xi$. The third estimate on $H_{\xi\xi}$ is immediate as there is evidently sufficient regularity on $a$.

For the fourth estimate on $H_{x\xi}$, consider first the transport term $V_\lambda$. We have using Proposition \ref{localparadiff} to commute the localization under $\D_x$ (or simply by using the product rule),
$$\|w_{x_0, \mu} S_{\mu}\D_x V\|_{H^{\half+}} \lesssim \|w_{x_0, \mu} S_{\mu}V\|_{H^{\frac{3}{2}+}} + \|V\|_{H^{s-}}.$$
Further, we have 
$$\|w_{x_0, \lambda} S_{\lesssim \lambda^{\frac{3}{4}}}\D_x V\|_{H^{\half+}} \lesssim \lambda^{0+}\|\D_x V\|_{L^\infty}.$$
Summing this with logarithmically many values of $\mu$, we conclude 
$$\|w_{x_0, \lambda}\D_x V\|_{H^{\half+}} \leq \FF(M(t))(Z(t) + \LL(t, x_0, \xi_0, \lambda)).$$

From the dispersive term, we have
\begin{align*}
\|w_{x_0, \lambda} \D_x \sqrt{a_\lambda}\|_{H^{\half+}} &\lesssim  \|w_{x_0, \lambda} a_\lambda^{-\half} \D_x a_\lambda \|_{H^\half} \\
&\lesssim \| a_\lambda^{-\half}\|_{H^{\half+}}\| w_{x_0, \lambda}\D_x a_\lambda \|_{L^\infty} + \|a_\lambda^{-\half}\|_{L^\infty}\| w_{x_0, \lambda}\D_x a_\lambda \|_{H^{\half+}}.
\end{align*}
The first three terms are all easily estimated using Proposition \ref{taylorbd}, so it remains to consider the last term in $H^{\half+}$.


By Corollary \ref{taylorsmooth}, we have
$$\|w_{x_0, \lambda} S_{\gg \lambda^{\frac{3}{4}}} a_\lambda\|_{H^{\frac{3}{2}+}} \lesssim \sum_{\lambda^{\frac{3}{4}} \ll \mu \leq c_1\lambda} \mu^{-\eps}\|w_{x_0, \lambda} S_\mu a\|_{H^{s'}} \leq \lambda^\half \FF(M(t))(Z(t) + \LL(t, x_0, \xi_0, \lambda)).$$
Thus, it remains to estimate
$$\|w_{x_0, \lambda} S_{\lesssim \lambda^{\frac{3}{4}}}\D_x a\|_{H^{\half+}} + \|(\D_x w_{x_0, \lambda})S_{\gg \lambda^{\frac{3}{4}}} a_\lambda\|_{H^{\half+}}.$$
The first is easily estimated,
$$\|w_{x_0, \lambda} S_{\lesssim \lambda^{\frac{3}{4}}}\D_x a\|_{H^{\half+}} \lesssim \lambda^{\frac{3}{8}+}\|S_{\lesssim \lambda^{\frac{3}{4}}}\D_x a\|_{L^2} \lesssim  \lambda^{\half}\|a\|_{H^{s - \half}}.$$
The second is estimated using the algebra property of $H^{\half+}$,
\begin{align*}
\|(\D_x w_{x_0, \lambda})S_{\gg \lambda^{\frac{3}{4}}} a_\lambda\|_{H^{\half+}} &\lesssim \|\D_x w_{x_0, \lambda}\|_{H^{\half+}} \|S_{\gg \lambda^{\frac{3}{4}}} a_\lambda\|_{H^{\half+}} \\
&\lesssim \lambda^{\frac{3}{4}+}\|S_{\gg \lambda^{\frac{3}{4}}} a_\lambda\|_{H^{\half+}} \\
&\lesssim \lambda^{\half}\|a\|_{H^{s - \half}}.
\end{align*}

\end{proof}

A direct estimate on $\tilde{H}_{xx} = \lambda^{-\frac{3}{2}} H_{xx}$ as in Lemma \ref{basicHbd} would yield estimates that are too crude for a bilipschitz flow (the estimate would depend on a positive power of $\lambda$). However, as discussed in Section \ref{COVsec}, we can integrate $H_{xx}$ along $(x^t, \xi^t)$, with the following estimates on the integral:

\begin{cor}\label{finalstructure}
We may write
\begin{equation}
\D_x^2H(t, x^t, \xi^t) = \frac{d}{dt} F(t, x^t, \xi^t) + G(t, x^t, \xi^t)
\end{equation}
where on $\{|\xi| \in [\lambda/2, 2\lambda]\}$,
\begin{align*}
\|F(t, \cdot, \xi)\|_{L^2} &\leq \lambda^{\frac{9}{8} -} \FF(M(t)),\\
\|F(t, \cdot, \xi)\|_{L^\infty} &\leq \lambda^{\frac{3}{2} -} \FF(M(t))Z(t),\\
\|w_{x_0, \lambda} F(t, \cdot, \xi)\|_{H^{\half+}} &\leq \lambda^{\frac{3}{2} -} \FF(M(t))(Z(t) + \LL(t, x_0, \xi_0, \lambda)), \\
\|G(t, \cdot, \xi)\|_{L^\infty} &\leq \lambda^{\frac{3}{2} -} \FF(M(t))Z(t)^2, \\
\|w_{x_0, \lambda} G(t, \cdot, \xi)\|_{H^{\half+}} &\leq \lambda^{\frac{3}{2} -}\FF(M(t))Z(t)(Z(t) + \LL(t, x_0, \xi_0, \lambda)).
\end{align*}
\end{cor}

\begin{proof}
By Corollary \ref{structurecor},
$$\D_x^2 H = (\D_t + H_\xi \D_x - H_x \D_\xi)(F_1\xi + F_\half|\xi|^\half) + G_1\xi + G_\half|\xi|^\half + G_0|\xi|^{-1} \xi.$$
Then set 
$$F(t, x, \xi) = F_1\xi + F_\half|\xi|^\half, \quad G(t, x, \xi) = G_1\xi + G_\half|\xi|^\half + G_0|\xi|^{-1} \xi,$$
from which the desired identity is easily verified, along with the desired estimates on $\{|\xi| \in [\lambda/2, 2\lambda]\}$. We make some additional remarks regarding the local Sobolev estimates:
By Corollary \ref{structurecor}, we have
$$\|w_{x_0, \lambda} S_\mu F(t, \cdot, \xi)\|_{H^\half} \leq \lambda^{\frac{3}{2} -} \FF(M(t))(Z(t) + \LL(t, x_0, \xi_0, \lambda)).$$
Since the first term of $F$,
$$F_1\xi = T_{q^{-1}} \D_x^3 \eta_\lambda \xi,$$
is localized to frequencies $\leq c\lambda$, we can sum these estimates over $\lambda^{\frac{3}{4}} \ll \mu \leq c\lambda$ with logarithmic loss. The component $\lesssim \lambda^{\frac{3}{4}}$ is estimated directly as in the proof of Lemma \ref{basicHbd}.

The other terms of $F$ and $G$ are not localized to frequencies $\leq c\lambda$, but only due to a coefficient in $C^\half$. Precisely, they may be written as
$$AS_{\lesssim c_1\lambda} B$$
where $A \in C^\half$ (for instance, $A = (a_\lambda)^{-\half}$). Thus, we may estimate the high frequency component
$$S_{> c_1\lambda} AS_{\lesssim c_1\lambda} B$$
by absorbing half a derivative into $A$ instead of using the local Sobolev estimates of Corollary \ref{structurecor}. 

\end{proof}

We denote in the following, where $F, G$ are as in Corollary \ref{finalstructure},
$$\tilde{F} = \lambda^{-\frac{3}{2}} F, \quad \tilde{G} = \lambda^{-\frac{3}{2}} G$$
so that
$$\D_x^2\tilde{H}(t, x^t, \xi^t) = \frac{d}{dt} \tilde{F}(t, x^t, \xi^t) + \tilde{G}(t, x^t, \xi^t).$$

\subsection{The bilipschitz flow}

We can combine the integration of Corollary \ref{finalstructure} with Gronwall to obtain estimates on $p(t)$ which exhibit the sense in which $(x^t, \xi^t)$ is bilipschitz.

\begin{prop}\label{bilipprop}
There exists $s_0 \in I$ such that for $J \subseteq \R$ with $|J| \lesssim \lambda^{-\frac{3}{4}}$, and solutions $(x^t, \xi^t)$ to (\ref{hamilton}) with initial data $(x, \xi)$ satisfying $|\xi| \in [\lambda/2, 2\lambda]$, at initial time
\begin{enumerate}[i)]
\item $s = s_0 \in I$,
$$\|\D_x X^t - \lambda^{\frac{3}{4}}I \|_{L^\infty_t(I; L^\infty_x)} \ll \lambda^{\frac{3}{4}}.$$
\item $s \in I$,
$$\|\D_x X^t - \lambda^{\frac{3}{4}}I\|_{L_x^2(J)} + \|\D_x \Xi^{s_0}\|_{L_x^2(J)} \ll \lambda^{\frac{3}{8}}.$$
\end{enumerate}
\end{prop}

\begin{proof}
Recall (\ref{eqn:linearizedhamscaled}),
\begin{equation*}
\frac{d}{dt} \D_{x} P(t) = \begin{pmatrix}
\tilde{H}_{\xi x}(t, x^t, \xi^t) & \tilde{H}_{\xi \xi}(t, x^t, \xi^t) \\
-\tilde{H}_{xx}(t, x^t, \xi^t) & -\tilde{H}_{x \xi}(t, x^t, \xi^t)
\end{pmatrix} \D_{x} P(t).
\end{equation*}
\emph{Step 1.} First, we integrate to obtain the right object on which to apply Gronwall. Using Corollary \ref{finalstructure}, write (omitting the parameters $(t, x^t, \xi^t)$ for brevity)
\begin{align*}
\frac{d}{dt} \D_{x} P(t) &= \begin{pmatrix}
0 & 0 \\
-\dot{\tilde{F}} & 0
\end{pmatrix} \D_{x} P(t) +
 \begin{pmatrix}
\tilde{H}_{\xi x} & \tilde{H}_{\xi \xi} \\
-\tilde{G} & -\tilde{H}_{x \xi}
\end{pmatrix} \D_{x} P(t) \\
&= -\frac{d}{dt}\left( \begin{pmatrix}
0 & 0 \\
\tilde{F} & 0
\end{pmatrix} \D_{x} P(t) \right)+
\begin{pmatrix}
0 & 0 \\
\tilde{F} & 0
\end{pmatrix} \frac{d}{dt}\D_{x} P(t)
+ 
\begin{pmatrix}
\tilde{H}_{\xi x} & \tilde{H}_{\xi \xi} \\
-\tilde{G} & -\tilde{H}_{x \xi}
\end{pmatrix} \D_{x} P(t).
\end{align*}
Substituting (\ref{eqn:linearizedhamscaled}) into the right hand side and rearranging,
\begin{align*}
\frac{d}{dt} \left( \begin{pmatrix}
I & 0 \\
\tilde{F} & I
\end{pmatrix} \D_{x} P(t) \right) &= 
\begin{pmatrix}
\tilde{H}_{\xi x} & \tilde{H}_{\xi \xi} \\
\tilde{F}\tilde{H}_{\xi x}-\tilde{G} & \tilde{F}\tilde{H}_{\xi\xi}-\tilde{H}_{x \xi}
\end{pmatrix} \D_{x} P(t).
\end{align*}
For brevity, write this as
\begin{equation}\label{eqn:diffeqn}
\frac{d}{dt} (I_F  \D_{x} P)(t) = (A\D_{x} P)(t).
\end{equation}
Integrating and rewriting the right hand side to apply Gronwall,
\begin{equation}\label{eqn:integratedeqn}
(I_F  \D_{x} P)(t) = (I_F \D_{x} P)(s) + \int_{s}^t (AI_F^{-1})(r)(I_F\D_{x} P)(r) \, dr.
\end{equation}

\emph{Step 2.} As a preliminary step, we establish fixed-time estimates in $L_x^\infty$. Applying uniform norms and Gronwall,
$$\|(I_F  \D_{x} P)(t)\|_{L^\infty_x}\lesssim \|(I_F \D_{x} P)(s)\|_{L^\infty_x} \exp\left(\int_{s}^t \|(AI_F^{-1})(r)\|_{L^\infty_x} \, dr \right).$$
Noting that
$$I_F^{-1} = \begin{pmatrix}
I & 0 \\
-\tilde{F} & I
\end{pmatrix}$$
and
$$(I_F \D_{x} P)(s) = I_F(s) \begin{pmatrix}
\lambda^{\frac{3}{4}}I \\
0
\end{pmatrix} = \lambda^{\frac{3}{4}}\begin{pmatrix}
I \\
\tilde{F}(s)
\end{pmatrix},$$
we have
\begin{equation}\label{fixedtime}
\|\D_{x} P(t)\|_{L^\infty_x}\lesssim \lambda^{\frac{3}{4}}(1 + \|\tilde{F}(t)\|_{L^\infty_x})(1 + \|\tilde{F}(s)\|_{L^\infty_x}) \exp\left(\|AI_F^{-1}\|_{L_t^1(I; L^\infty_x)}\right).
\end{equation}

To estimate the right hand side, first consider the exponential term:
\begin{align*}AI_F^{-1} &= 
\begin{pmatrix}
\tilde{H}_{\xi x} & \tilde{H}_{\xi \xi} \\
\tilde{F}\tilde{H}_{\xi x}-\tilde{G} & \tilde{F}\tilde{H}_{\xi\xi}-\tilde{H}_{x \xi}
\end{pmatrix} \begin{pmatrix}
I & 0 \\
-\tilde{F} & I
\end{pmatrix} \\
&=\begin{pmatrix}
\tilde{H}_{\xi x} - \tilde{F}\tilde{H}_{\xi \xi} & \tilde{H}_{\xi \xi} \\
\tilde{F}\tilde{H}_{\xi x}-\tilde{G} - \tilde{F}\tilde{H}_{\xi\xi}\tilde{F} + \tilde{H}_{x \xi}\tilde{F} & \tilde{F}\tilde{H}_{\xi\xi}-\tilde{H}_{x \xi}
\end{pmatrix}
\end{align*}
Then using Corollary \ref{finalstructure}, Lemma \ref{basicHbd}, and Cauchy-Schwarz,
\begin{align}\label{expest}
\|AI_F^{-1}\|_{L_t^1(I; L^\infty_x)} \lesssim \ &T^{\half}(\|\tilde{H}_{\xi x} \|_{L_t^2(I; L^\infty_x)} + \|\tilde{H}_{\xi \xi} \|_{L_t^2(I; L^\infty_x)}) + \|\tilde{G} \|_{L_t^1(I; L^\infty_x)} \nonumber\\
&+ \|\tilde{F} \|_{L_t^2(I; L^\infty_x)}\left(\|\tilde{H}_{\xi \xi} \|_{L_t^2(I; L^\infty_x)} + \|\tilde{H}_{x \xi} \|_{L_t^2(I; L^\infty_x)}\right) \nonumber\\
&+ \|\tilde{F} \|_{L_t^2(I; L^\infty_x)}^2\|\tilde{H}_{\xi \xi} \|_{L_t^\infty(I; L^\infty_x)} \nonumber\\
\leq \ &(T^{\half} + \lambda^{0-}) \FF(T) \ll 1.
\end{align}
Here we have chosen $T$ sufficiently small, iterating the argument over multiple time intervals as necessary, and $\lambda$ sufficiently large. For smaller $\lambda$, the frequency localized Strichartz estimates can be proven directly with Sobolev embedding.

To estimate $\tilde{F}(s)$, we choose $s_0 \in I$ such that
\begin{align*}
\|\tilde{F}(s_0)\|_{L^\infty_x} \leq T^{-1}\|\tilde{F}\|_{L_t^1(I; L^\infty_x)} \leq T^{-\half}\|\tilde{F}\|_{L_t^2(I; L^\infty_x)}.
\end{align*}
Then using Corollary \ref{finalstructure} and choosing $\lambda$ sufficiently large as before,
\begin{equation}\label{specialtime}
\|\tilde{F}(s_0)\|_{L^\infty_x} \leq T^{-\half}\lambda^{0-}\FF(T) \ll T^{-\half}.
\end{equation}

Thus, by setting $s = s_0$ in (\ref{fixedtime}) and using Corollary \ref{finalstructure} on $\tilde{F}(t)$, we find
\begin{align}\label{eqn:firstbilip}
\|\D_{x} P(t)\|_{L^\infty_x} \leq \lambda^{\frac{3}{4}}T^{-\half}\FF(M(t))Z(t).
\end{align}

\emph{Step 3.} Next we improve upon Step 2 by showing that $\D_x X^t \approx \lambda^{\frac{3}{4}}I$ uniformly in $t$ (thus establishing the first estimate in part $(i)$). From the top row of (\ref{eqn:integratedeqn}),
$$\D_{x}X^t = \lambda^{\frac{3}{4}}I + \int_{s_0}^t \tilde{H}_{\xi x}(r) \D_{x}X^r + \tilde{H}_{\xi\xi}(r)\D_{x}\Xi^r \, dr.$$
Using (\ref{eqn:firstbilip}), Lemma \ref{basicHbd}, and Cauchy-Schwarz,
\begin{equation}\label{identitymatrix}
\|\D_{x}X^t - \lambda^{\frac{3}{4}}I\|_{L^\infty_x} \leq \int_{s_0}^t \|\tilde{H}_{\xi x}(r)\D_{x}X^r\|_{L^\infty_x} + \|\tilde{H}_{\xi\xi}(r)\D_{x}\Xi^r\|_{L^\infty_x} \, dr \leq \lambda^{\frac{3}{4}}T^{-\half}\FF(T)
\end{equation}
and in particular,
$$\|\D_{x}X^t\|_{L_t^\infty(I;L^\infty_x)} \leq \lambda^{\frac{3}{4}}T^{-\half}\FF(T).$$
Using this in place of (\ref{eqn:firstbilip}), revisit (\ref{identitymatrix}) and perform the estimate again to obtain
\begin{align*}
\|\D_{x}X^t - \lambda^{\frac{3}{4}}I\|_{L^\infty_x} &\leq T^{\half}\FF(T)(\lambda^{\frac{3}{4}}T^{-\half}\|\tilde{H}_{\xi x} \|_{L_t^2(I; L^\infty_x)} + \|\D_x\Xi^t \|_{L_t^2(I; L^\infty_x)}) \\
&\leq \lambda^{\frac{3}{4}}\FF(T). \nonumber
\end{align*}
Repeating this process one more time, we obtain
\begin{equation}\label{xbilip}
\|\D_{x}X^t - \lambda^{\frac{3}{4}}I\|_{L^\infty_x} \leq \lambda^{\frac{3}{4}}T^{\half}\FF(T) \ll \lambda^{\frac{3}{4}}.
\end{equation}

\emph{Step 4.} Lastly, we establish $L_x^2(J)$ estimates in $(ii)$. Recalling that
$$(I_F \D_{x} P)(s) = I_F(s) \begin{pmatrix}
\lambda^{\frac{3}{4}}I \\
0
\end{pmatrix} = \lambda^{\frac{3}{4}}\begin{pmatrix}
I \\
\tilde{F}(s)
\end{pmatrix} := \lambda^{\frac{3}{4}}I_x + \lambda^{\frac{3}{4}}\begin{pmatrix}
0 \\
\tilde{F}(s)
\end{pmatrix},$$
it is natural to write, in place of (\ref{eqn:integratedeqn}),
$$(I_F  \D_{x} P - \lambda^{\frac{3}{4}}I_x)(t) = \lambda^{\frac{3}{4}}\begin{pmatrix}
0 \\
\tilde{F}(s)
\end{pmatrix} + \int_{s}^t (AI_F^{-1})(r)(I_F\D_{x} P - \lambda^{\frac{3}{4}}I_x)(r) + \lambda^{\frac{3}{4}}(AI_F^{-1})(r)I_x \, dr.$$
Then apply $L_x^2(J)$ and Gronwall:
\begin{align*}
\|(I_F  \D_{x} P - &\lambda^{\frac{3}{4}}I_x)(t)\|_{L_x^2(J)} \\
&\lesssim \lambda^{\frac{3}{4}}(\|\tilde{F}(s)\|_{L_x^2(J)} + \|AI_F^{-1}\|_{L_t^1(I;L_x^2(J))}) \exp\left(\int_{s}^t \|(AI_F^{-1})(r)\|_{L^\infty_x} \, dr \right). \nonumber
\end{align*}
We estimate the right hand side using the (global) $L^2$ estimate of Corollary \ref{finalstructure}, (\ref{expest}), and Cauchy-Schwarz in space, concluding
\begin{equation}\label{prelocalx}
\|(I_F  \D_{x} P - \lambda^{\frac{3}{4}}I_x)(t)\|_{L_x^2(J)} \ll \lambda^{\frac{3}{8}}.
\end{equation}
From the top row, we obtain
\begin{equation}\label{goodlocalx}
\|\D_x X^t - \lambda^{\frac{3}{4}}I\|_{L_x^2(J)} \ll \lambda^{\frac{3}{8}}
\end{equation}
as desired. 

\

It remains to establish the estimate on $\D_x \Xi^{s_0}$. First, recalling,
$$I_F^{-1} = \begin{pmatrix}
I & 0 \\
-\tilde{F} & I
\end{pmatrix},$$
we have using Corollary \ref{finalstructure} with (\ref{prelocalx}),
\begin{equation}\label{badlocalx}
\| \D_{x} P(t)\|_{L_x^2(J)} \lesssim (1 + \|\tilde{F}(t)\|_{L^\infty_x})\|(I_F  \D_{x} P)(t)\|_{L_x^2(J)} \leq \lambda^{\frac{3}{8}}\FF(M(t))Z(t).
\end{equation}

Using again the top row of (\ref{eqn:integratedeqn}),
$$\D_{x}X^t = \lambda^{\frac{3}{4}}I + \int_{s_0}^t \tilde{H}_{\xi x}(r) \D_{x}X^r + \tilde{H}_{\xi\xi}(r)\D_{x}\Xi^r \, dr,$$
we apply $L_x^2(J)$, and use (\ref{goodlocalx}, (\ref{badlocalx}), Lemma \ref{basicHbd}, and Cauchy-Schwarz to obtain
\begin{align*}
\|\D_{x}X^t - \lambda^{\frac{3}{4}}I\|_{L_x^2(J)} &\lesssim \int_{s}^t \|\tilde{H}_{\xi x}(r)\|_{L^\infty} \|\D_{x}X^r\|_{L_x^2(J)} + \|\tilde{H}_{\xi\xi}(r)\|_{L^\infty} \|\D_{x}\Xi^r\|_{L_x^2(J)} \, dr \\
&\leq \lambda^{\frac{3}{8}}T^{\half}\FF(T).
\end{align*}
Combining this with (\ref{specialtime}),
\begin{equation}\label{switchlocalx}
\|\tilde{F}(s_0)(\D_{x}X^t - \lambda^{\frac{3}{4}}I)\|_{L_x^2(J)} \leq T^{-\half}\lambda^{0-}\lambda^{\frac{3}{8}}T^{\half}\FF(T) \ll \lambda^{\frac{3}{8}}.
\end{equation}


Returning next to the bottom row of (\ref{prelocalx}), we obtain
$$\|\tilde{F}(t)\D_x X^t + \D_x \Xi^t\|_{L_x^2(J)} \ll \lambda^{\frac{3}{8}}.$$
Setting $t = s_0$ and applying (\ref{switchlocalx}),
$$\|\lambda^{\frac{3}{4}}\tilde{F}(s_0) + \D_x \Xi^{s_0}\|_{L_x^2(J)} \ll \lambda^{\frac{3}{8}}.$$
Lastly, using the (global) $L^2$ estimate of Corollary \ref{finalstructure} (here using as well as the bilipschitz property (\ref{goodlocalx}) since $F(s_0) = F(s_0, x^{s_0}, \xi^{s_0})$), we obtain
$$\|\D_x \Xi^{s_0}\|_{L_x^2(J)} \ll \lambda^{\frac{3}{8}},$$
completing the two estimates of $(ii)$.


\end{proof}

We also require estimates on $\D_\xi P(t)$ to exhibit the spreading of the characteristics, which we later use to show dispersive estimates:

\begin{prop}\label{spreadprop}
Let $s_0 \in I$ be as in Proposition \ref{bilipprop}. Then for solutions $(x^t, \xi^t)$ to (\ref{hamilton}) with initial data $(x, \xi)$ satisfying $|\xi| \in [\lambda/2, 2\lambda]$, at initial time $s \in I$,
$$\|\D_\xi \Xi^{s_0} - \lambda^{-\frac{3}{4}}I\|_{L^\infty_x} \ll \lambda^{-\frac{3}{4}}.$$
Further, if $t \geq s$,
$$-\D_{\xi}X^t  \gtrsim \lambda^{-\frac{3}{4}}(t - s),$$
and if $s > t$,
$$\D_{\xi}X^t  \gtrsim \lambda^{-\frac{3}{4}}(s - t).$$
In either case,
$$\|\D_{\xi}X^t\|_{L^\infty_x} \leq \lambda^{-\frac{3}{4}}\FF(T)|t - s|.$$
\end{prop}

\begin{proof}
Similar to (\ref{eqn:linearizedhamscaled}), we have
\begin{equation*}
\frac{d}{dt} \D_{\xi} P(t) = \begin{pmatrix}
\tilde{H}_{\xi x}(t, x^t, \xi^t) & \tilde{H}_{\xi \xi}(t, x^t, \xi^t) \\
-\tilde{H}_{xx}(t, x^t, \xi^t) & -\tilde{H}_{x \xi}(t, x^t, \xi^t)
\end{pmatrix} \D_{\xi} P(t).
\end{equation*}

\emph{Step 1.} We first establish $L_x^\infty$ estimates, uniform in time. In analogy to (\ref{eqn:integratedeqn}), we can integrate $\tilde{H}_{xx}$ to write
\begin{align*}
(I_F  \D_{\xi} P)(t) &= (I_F \D_{\xi} P)(s) + \int_{s}^t (AI_F^{-1})(r)(I_F\D_{\xi} P)(r) \, dr.
\end{align*}
Observing that
$$(I_F \D_{\xi} P)(s) =\begin{pmatrix}
0 \\
\lambda^{-\frac{3}{4}}I
\end{pmatrix} := \lambda^{-\frac{3}{4}}I_\xi,$$
we may write
\begin{align}\label{integratedeqn2}
(I_F  \D_{\xi} P - \lambda^{-\frac{3}{4}}I_\xi)(t)
= \ &\int_{s}^t (AI_F^{-1})(r)(I_F\D_{\xi} P - \lambda^{-\frac{3}{4}}I_\xi)(r)+ \lambda^{-\frac{3}{4}}(AI_F^{-1})(r)I_\xi \, dr 
\end{align}

Applying uniform norms, Gronwall, and a subset of the computations ending in (\ref{expest}), we have
\begin{align*}
\|(I_F  \D_{\xi} P - \lambda^{-\frac{3}{4}}I_\xi)(t)\|_{L^\infty_x} &\lesssim \lambda^{-\frac{3}{4}}\|AI_F^{-1}I_\xi\|_{L_t^1(I; L^\infty_x)} \exp\left(\int_{s}^t \|(AI_F^{-1})(r)\|_{L^\infty_x} \, dr \right) \\
&\leq \lambda^{-\frac{3}{4}}T^{\half}\FF(T).
\end{align*}
From the top row, we have
\begin{equation*}
\|\D_\xi X^t\|_{L^\infty_x} \leq \lambda^{-\frac{3}{4}}T^{\half}\FF(T)
\end{equation*}
and from the bottom row,
\begin{equation}\label{spreadbdxi}
\|\tilde{F}(t)\D_\xi X^t + \D_\xi \Xi^t - \lambda^{-\frac{3}{4}}I\|_{L^\infty_x} \leq  \lambda^{-\frac{3}{4}}T^{\half}\FF(T) \ll \lambda^{-\frac{3}{4}}.
\end{equation}
In particular, combining these two estimates, setting $t = s_0$, and using (\ref{specialtime}), we find
$$\|\D_\xi \Xi^{s_0} - \lambda^{-\frac{3}{4}}I\|_{L^\infty_x} \leq \lambda^{-\frac{3}{4}}(T^{-\half}\lambda^{0-}T^{\half}\FF(T) + c) \ll \lambda^{-\frac{3}{4}}$$
as desired.

\emph{Step 2.} Next we show that $\D_\xi X^t$ is linear in the time $t$. From the top row of (\ref{integratedeqn2}),
$$\D_{\xi}X^t = \int_{s}^t \tilde{H}_{\xi x}(r) \D_{\xi}X^r + \tilde{H}_{\xi\xi}(r)\D_{\xi}\Xi^r \, dr.$$
Substituting the formula for $\D_\xi \Xi^t$ from (\ref{spreadbdxi}), we have
\begin{equation}\label{integratedeqn3}
\D_{\xi}X^t = \int_{s}^t (\tilde{H}_{\xi x} - \tilde{H}_{\xi\xi} \tilde{F})(r) \D_{\xi}X^r + \tilde{H}_{\xi\xi}(r)(\lambda^{-\frac{3}{4}}I + K(r)) \, dr
\end{equation}
where
$$\|K(t)\|_{L^\infty_x} \ll \lambda^{-\frac{3}{4}}.$$
Using Gronwall followed by Lemma \ref{basicHbd} and (\ref{expest}) (noting that $\tilde{F}\tilde{H}_{\xi\xi}-\tilde{H}_{x \xi}$ is one of the coefficients of $AI_F^{-1}$), we have
\begin{align*}
\|\D_\xi X^t\|_{L^\infty_x} &\lesssim |t - s|\|\tilde{H}_{\xi\xi}(t)(\lambda^{-\frac{3}{4}}I + K(t))\|_{L^\infty_t(I;L_x^\infty)} \exp \left(\int_{s}^t \|(AI_F^{-1})(r)\|_{L^\infty_x} \, dr \right) \\
&\leq \lambda^{-\frac{3}{4}}\FF(T)|t - s|
\end{align*}
as desired.

In turn, using this estimate in again (\ref{integratedeqn3}), we can write, using Corollary \ref{finalstructure}, Lemma \ref{basicHbd}, and Cauchy-Schwarz,
$$\|\D_{\xi}X^t - \int_s^t \tilde{H}_{\xi\xi}(r)(\lambda^{-\frac{3}{4}}I + K(r)) \, dr\|_{L^\infty_x} \leq T^\half \lambda^{-\frac{3}{4}}\FF(T)|t - s|\ll \lambda^{-\frac{3}{4}}|t - s|.$$
Lastly, by the Taylor sign condition and Lemma \ref{freqpres},
$$-\tilde{H}_{\xi\xi} = \frac{1}{4}\lambda^{\frac{3}{2}}\sqrt{a_\lambda} |\xi^t|^{-\frac{3}{2}} \gtrsim \sqrt{a_{min}} > 0$$
so that if $t \geq s$,
$$- \int_s^t \tilde{H}_{\xi\xi}(r) \, dr > \sqrt{a_{min}}(t - s)$$
and if $s > t$,
$$\int_s^t \tilde{H}_{\xi\xi}(r) \, dr > \sqrt{a_{min}}(s - t).$$

We conclude by the triangle inequality that if $t > s$,
$$-\D_{\xi}X^t  \gtrsim \lambda^{-\frac{3}{4}}(t - s),$$
and if $s > t$,
$$\D_{\xi}X^t  \gtrsim \lambda^{-\frac{3}{4}}(s - t).$$
\end{proof}

\subsection{Geometry of the characteristics}

In this subsection we discuss the geometry of the characteristics $(x^t, \xi^t)$. Throughout, $s_0 \in I$ denotes the $s_0$ discussed in Propositions \ref{bilipprop} and \ref{spreadprop}.

First we observe that if two characteristics intersect at time $t$ with similar frequencies, then they must have had similar initial frequencies.

\begin{prop}\label{charoverlap}
For solutions $(x^t, \xi^t)$ to (\ref{hamilton}) with initial data $(x_i, \xi_i)$ at initial time $s_0 \in I$, if
$$|x^t(x_1, \xi_1) - x^t(x_2, \xi_2)| \leq \lambda^{-\frac{3}{4}},$$
then 
$$|\xi_1 - \xi_2| \leq 2|\xi^t(x_1, \xi_1) - \xi^t(x_2, \xi_2)| + \lambda^{\frac{3}{4}}.$$
\end{prop}

\begin{proof}
Write $\xi^{s_0} = \xi_t^{s_0}$ and
\begin{equation}\label{startxi}
\xi_1 - \xi_2 = \xi^{s_0}(x^t(x_1, \xi_1), \xi^t(x_1, \xi_1)) - \xi^{s_0}(x^t(x_2, \xi_2), \xi^t(x_2, \xi_2)).
\end{equation}
Using the fundamental theorem of calculus and $(ii)$ in Proposition \ref{bilipprop} with inital time $s = t$, we have (assuming without loss of generality that $x^t(x_1, \xi_1) \leq x^t(x_2, \xi_2)$)
\begin{align*}
|\xi^{s_0}(x^t(x_2, \xi_2), \xi^t(x_2, \xi_2)) - \xi^{s_0}(x^t(x_1, \xi_1), \xi^t(x_2, \xi_2))| &\leq \int_{x^t(x_1, \xi_1)}^{x^t(x_2, \xi_2)} |\D_x \xi^{s_0}(z, \xi^t(x_2, \xi_2))| \, dz \\
&\ll \lambda^{\frac{3}{4}}.
\end{align*}
On the other hand, using Proposition \ref{spreadprop},
\begin{align*}
|\xi^{s_0}(x^t(x_1, \xi_1), \xi^t(x_1, \xi_1)) - \xi^{s_0}(x^t(x_1, \xi_1), \xi^t(x_2, \xi_2))| &\leq \int^{\xi^t(x_2, \xi_2)}_{\xi^t(x_1, \xi_1)} |\D_\xi \xi^{s_0} (x^t(x_1, \xi_1), \eta)| \, d\eta \\
&\leq 2|\xi^t(x_1, \xi_1) - \xi^t(x_2, \xi_2)|.
\end{align*}
Using these two estimates with (\ref{startxi}), we obtain
$$|\xi_1 - \xi_2| \leq 2|\xi^t(x_1, \xi_1) - \xi^t(x_2, \xi_2)| + \lambda^{\frac{3}{4}}.$$
\end{proof}

Next, we observe that two characteristics which intersect at two distant times must have similar initial frequencies:

\begin{prop}\label{char2pts}
For solutions $(x^t, \xi^t)$ to (\ref{hamilton}) with initial data $(x_i, \xi_i)$ at initial time $s_0 \in I$, if
$$|x^s(x_1, \xi_1) - x^s(x_2, \xi_2)| \leq \lambda^{-\frac{3}{4}}, \quad |x^t(x_1, \xi_1) - x^t(x_2, \xi_2)| \leq \lambda^{-\frac{3}{4}}$$
for $t, s \in I$,
then 
$$|\xi_1 - \xi_2| \lesssim \lambda^{\frac{3}{4}}|t - s|^{-1}.$$
\end{prop}

\begin{proof}
Write $x^r = x^r_{s_0}$ unless otherwise indicated, and
\begin{equation}\label{start2pts}
x^t(x_1, \xi_1) - x^t(x_2, \xi_2) = x^t_s(x^s(x_1, \xi_1), \xi^s(x_1, \xi_1)) - x^t_s(x^s(x_2, \xi_2), \xi^s(x_2, \xi_2)).
\end{equation}

Using the fundamental theorem of calculus and $(ii)$ in Proposition \ref{bilipprop},
\begin{align*}
|x^t_s(x^s(x_2, \xi_2), \xi^s(x_2, \xi_2)) - x^t_s(x^s(x_1, \xi_1), \xi^s(x_2, \xi_2))| &\leq \int_{x^s(x_1, \xi_1)}^{x^s(x_2, \xi_2)} |\D_x x^t_s(z, \xi^s(x_2, \xi_2))| \, dz \\
&\lesssim \lambda^{-\frac{3}{4}}.
\end{align*}
On the other hand, considering without loss of generality the case $t \geq s$ and $\xi^s(x_2, \xi_2) \geq \xi^s(x_1, \xi_1)$, we may use Proposition \ref{spreadprop} to estimate
\begin{align*}
x^t_s(x^s(x_1, \xi_1), \xi^s(x_1, \xi_1)) - x^t_s(x^s(x_1, \xi_1), \xi^s(x_2, \xi_2))  &= -\int_{\xi^s(x_1, \xi_1)}^{\xi^s(x_2, \xi_2)} \D_\xi x^t_s(x^s(x_1, \xi_1), \eta) \, d\eta \\
&\gtrsim\lambda^{-\frac{3}{2}}(\xi^s(x_2, \xi_2) - \xi^s(x_1, \xi_1))(t - s).
\end{align*}
Using these two estimates with (\ref{start2pts}), we obtain
$$|x^t(x_1, \xi_1) - x^t(x_2, \xi_2)| + \lambda^{-\frac{3}{4}} \gtrsim \lambda^{-\frac{3}{2}}(\xi^s(x_2, \xi_2) - \xi^s(x_1, \xi_1))(t - s).$$

We have an upper bound on the left hand side by assumption, and a lower bound on the right hand side by Proposition \ref{charoverlap}:
$$2\lambda^{-\frac{3}{4}} \gtrsim \half\lambda^{-\frac{3}{2}} (|\xi_1 - \xi_2| - \lambda^{\frac{3}{4}})(t - s).$$
Rearranging yields the claim.

\end{proof}

Conversely, we see that if two packets  intersect at a given time, then they intersect at nearby times:

\begin{prop}\label{charnearpts}
For solutions $(x^t, \xi^t)$ to (\ref{hamilton}) with initial data $(x_i, \xi_i)$ at initial time $s_0 \in I$,
$$|x^t(x_1, \xi_1) - x^t(x_2, \xi_2)| \lesssim |x^s(x_1, \xi_1) - x^s(x_2, \xi_2)| + \lambda^{-\half}\FF(T)|t - s|.$$
\end{prop}

\begin{proof}
Write $x^r = x^r_{s_0}$ unless otherwise indicated, and
\begin{equation}\label{startnearpts}
x^t(x_1, \xi_1) - x^t(x_2, \xi_2) = x^t_s(x^s(x_1, \xi_1), \xi^s(x_1, \xi_1)) - x^t_s(x^s(x_2, \xi_2), \xi^s(x_2, \xi_2)).
\end{equation}

Using the fundamental theorem of calculus and $(ii)$ in Proposition \ref{bilipprop},
\begin{align*}
x^t_s(x^s(x_2, \xi_2), \xi^s(x_2, \xi_2)) - x^t_s(x^s(x_1, \xi_1), \xi^s(x_2, \xi_2)) &= \int_{x^s(x_1, \xi_1)}^{x^s(x_2, \xi_2)} \D_x x^t_s(z, \xi^s(x_2, \xi_2)) \, dz \\
&\approx  x^s(x_2, \xi_2) - x^s(x_1, \xi_1).
\end{align*}
On the other hand, we may use the last estimate of Proposition \ref{spreadprop} to obtain
\begin{align*}
|x^t_s(x^s(x_1, \xi_1), \xi^s(x_1, \xi_1)) - x^t_s(x^s(x_1, \xi_1), \xi^s(x_2, \xi_2))| &\leq \int_{\xi^s(x_1, \xi_1)}^{\xi^s(x_2, \xi_2)} |\D_\xi x^t_s(x^s(x_1, \xi_1), \eta)| \, d\eta \\
&\leq \lambda^{-\half}\FF(T)|t - s|.
\end{align*}
Using these two estimates with (\ref{startnearpts}), we obtain
$$|x^t(x_1, \xi_1) - x^t(x_2, \xi_2)| \lesssim |x^s(x_1, \xi_1) - x^s(x_2, \xi_2)| + \lambda^{-\half}\FF(T)|t - s|.$$

\end{proof}

\subsection{Characteristic local smoothing}

The analysis so far has addressed the motion of the characteristics $(x^t, \xi^t)$ in physical space, and hence the motion of the wave packets we construct in Section \ref{sec:parametrix}. However, for the square summability of the wave packets in our parametrix construction, we will need control of an extra half derivative on the phase of the packets. We obtain this gain via local smoothing estimates on the characteristics.



Before establishing the local smoothing estimate, we establish a bilipschitz property for the Hamilton characteristics, but relative to initial data at a time $s_T \in I$ amenable to local smoothing:

\begin{lem}\label{bilipproplocal}
Let $T = (x_0, \xi_0)$ with $|\xi_0| \in [\lambda/2, 2\lambda]$. There exists $s_T \in I$ such that for solutions $(x^t, \xi^t)$ to (\ref{hamilton}) with initial data $(x, \xi)$ at initial time $s_T$, and satisfying $|\xi| \in [\lambda/2, 2\lambda]$,
$$\|w_{x^{s_T}, \lambda} (\D_x X^t - \lambda^{\frac{3}{4}}I)\|_{L^\infty_t(I; L^\infty_x)} \ll \lambda^{\frac{3}{4}}, \quad \|w_{x^{s_T}, \lambda}\tilde{F}(s_T)\|_{H_x^{\half+}} \ll T^{-\half}$$
where we denote $x^{s_T} = x^{s_T}_{s_0}(x_0, \xi_0)$.
\end{lem}

\begin{proof}

The proof is similar to the proof of Proposition \ref{bilipprop}, except we multiply by $w_{x^s, \lambda}$ before applying $L_x^\infty$, with $s$ to be chosen as follows: When estimating
$$(I_F \D_{x} P)(s) = I_F(s) \begin{pmatrix}
\lambda^{\frac{3}{4}}I \\
0
\end{pmatrix} = \lambda^{\frac{3}{4}}\begin{pmatrix}
I \\
\tilde{F}(s)
\end{pmatrix},$$
we use that
$$\|w_{x^{s}, \lambda} \tilde{F}(s)\|_{L_x^\infty} \lesssim \|w_{x^{s}, \lambda} \tilde{F}(s)\|_{H_x^{\half+}}.$$
Then we choose $s = s_T$ such that
\begin{align*}
\|w_{x^{s_T}, \lambda}\tilde{F}(s_T)\|_{H_x^{\half+}} \leq T^{-1}\|w_{x^{t}, \lambda}\tilde{F}\|_{L_t^1(I; H_x^{\half+})} \leq T^{-\half}\|w_{x^{t}, \lambda}\tilde{F}\|_{L_t^2(I; H_x^{\half+})}.
\end{align*}
Then using Corollary \ref{finalstructure}, combined with the fact that by Corollaries \ref{wwlsspaceeasy} and \ref{wwlsspace},
$$\|\LL(t, x^t, \xi^t, \lambda)\|_{L_t^2} \leq \lambda^{0+} \FF(T),$$
we have
\begin{equation*}
\|w_{x^{s_T}, \lambda}\tilde{F}(s_T)\|_{H_x^{\half+}} \leq T^{-\half}\lambda^{0-}\FF(T) \ll T^{-\half}.
\end{equation*}
\end{proof}

\begin{prop}\label{localprop}
Let $T = (x_0, \xi_0)$ with $|\xi_0| \in [\lambda/2, 2\lambda]$. There exists $s_T \in I$ such that for solutions $(x^t, \xi^t)$ to (\ref{hamilton}) with initial data $(x, \xi)$ at initial time $s_T$, and satisfying $|\xi| \in [\lambda/2, 2\lambda]$,
\begin{align*}
\|\chi_{x^{s_T}, \lambda}(\D_{x}X^t - \lambda^{\frac{3}{4}}I)\|_{H^{\half+}} &\ll \lambda^{\frac{3}{4}}, \\
\|\chi_{x^{s_T}, \lambda}\D_x \Xi^t\|_{L_t^2(I; H_x^{\half+})} &\ll \lambda^{\frac{3}{4}}, \\
\|\chi_{x^{s_T}, \lambda}\D_x \Xi^t\|_{L_t^\infty(I; L^2_x)} &\ll \lambda^{\frac{3}{8}-}
\end{align*}
where we denote $x^{s_T} = x^{s_T}_{s_0}(x_0, \xi_0)$.
\end{prop}

\begin{proof}
We begin by multiplying (\ref{eqn:integratedeqn}) by $\chi_{x^{s_T}, \lambda}(x)$, applying $H^{\half+}$ norms, and using the algebra property,
\begin{align*}
\|\chi_{x^{s_T}, \lambda}(I_F  \D_{x} P)(t)\|_{H^{\half+}} &\lesssim \|\chi_{x^{s_T}, \lambda}(I_F \D_{x} P)(s_T)\|_{H^{\half+}} \\
&\quad + \int_{s_T}^t \|\tilde{\chi}_{x^{s_T}, \lambda}(AI_F^{-1})(r)\|_{H^{\half+}}\|\chi_{x^{s_T}, \lambda}(I_F\D_{x} P)(r)\|_{H^{\half+}} \, dr.
\end{align*}
Then applying Gronwall,
\begin{align}\label{startingeqn}
\|\chi_{x^{s_T}, \lambda}(I_F  \D_{x} P)(t)\|_{H^{\half+}} &\lesssim \|\chi_{x^{s_T}, \lambda}(I_F \D_{x} P)(s_T)\|_{H^{\half+}} \exp\left(\int_{s_T}^t \|\tilde{\chi}_{x^{s_T}, \lambda}(AI_F^{-1})(r)\|_{H^{\half+}} \, dr\right).
\end{align}

We consider the exponential term. Consider a typical term in the matrix $AI_F^{-1}$,
$$\|\chi_{x^{s_T}, \lambda}(x)\tilde{H}_{x\xi}(r, x^r, \xi^r)\|_{H^{\half+}}.$$
In turn, $\tilde{H}_{x\xi}$ is a sum of a transport term and a dispersive term. We consider the transport term,
$$\|\chi(\lambda^{\frac{3}{4}}(x - x^{s_T}))(\D_x V_\lambda)(r, x^r(x, \xi))\|_{H^{\half+}}.$$
Using the bilipschitz property of Lemma \ref{bilipproplocal}, this may be rewritten with an inserted weight,
$$\|\chi(\lambda^{\frac{3}{4}}(x - x^{s_T}))\tilde{\chi}(\lambda^{\frac{3}{4}}(x^r(x, \xi) - x^r(x^{s_T}, \xi))) (\D_x V_\lambda)(r, x^r(x, \xi))\|_{H^{\half+}}.$$
Then using the algebra property of $H^{\half+}$, it suffices to estimate
$$\|\tilde{\chi}(\lambda^{\frac{3}{4}}(x^r(x, \xi) - x^r(x^{s_T}, \xi))) (\D_x V_\lambda)(r, x^r(x, \xi))\|_{H^{\half+}}.$$
We remark that measuring the weight in $H^{\half+}$ nets a loss of $\lambda^{0+}$. We may choose $T \leq \lambda^{0-}$ to negate this, iterating the argument as necessary.

Using again the bilipschitz property of Lemma \ref{bilipproplocal} with the diffeomorphism estimate Proposition \ref{compositionlem}, it suffices to estimate
$$\|\tilde{\chi}(\lambda^{\frac{3}{4}}(x - x^r(x^{s_T}, \xi))) (\D_x V_\lambda)(r, x)\|_{H^{\half+}} = \|\tilde{\chi}_{x^r, \lambda} \D_x V_\lambda\|_{H^{\half+}} \lesssim \|w_{x^r, \lambda} \D_x V_\lambda\|_{H^{\half+}}.$$
Then using the last estimate of Lemma \ref{basicHbd} (more accurately, its proof), use again the assumption $T \leq \lambda^{0-}$ to conclude that 
$$\|\chi_{x^{s_T}, \lambda}(x)(\D_x V_\lambda)(r, x^r)\|_{L^1(I;H^{\half+})} \leq T^{\half} \lambda^{0+}\FF(T) \ll 1.$$

Similar analyses apply to the other terms of $AI_F^{-1}$, with repeated use of the algebra property of $H^{\half+}$. We remark that other terms also have $\xi^r$ factors, which require (after possibly applying a Moser composition estimate)
\begin{equation}\label{xilemma}
\|\chi(\lambda^{\frac{3}{4}}(x - x^{s_T})) \xi^r(x, \xi)\|_{H^{\half+}} \lesssim \lambda^{1+},
\end{equation}
which we establish in Lemma \ref{xiregularitylem} below.

We conclude from (\ref{startingeqn}), using the estimate on $F$ from Lemma \ref{bilipproplocal},
$$\|\chi_{x^{s_T}, \lambda}(I_F  \D_{x} P)(t)\|_{H^{\half+}} \lesssim \lambda^{\frac{3}{4}}T^{-\half}.$$
Further, recalling 
$$I_F^{-1} = \begin{pmatrix}
I & 0 \\
-\tilde{F} & I
\end{pmatrix}$$
and using the same analysis as above, combined with (the proof of) Corollary \ref{finalstructure},
\begin{align*}
\|\chi_{x^{s_T}, \lambda}(\D_{x} P)(t)\|_{H^{\half+}} &\lesssim (1 + \|\tilde{\chi}_{x^{s_T}, \lambda}\tilde{F}(t)\|_{L^\infty_x})\|\chi_{x^{s_T}, \lambda}(I_F  \D_{x} P)(t)\|_{H^{\half+}} \\
&\leq \lambda^{\frac{3}{4}}T^{-\half}\FF(M(t))(Z(t) + \LL(t, x^t, \xi^t, \lambda)).
\end{align*}

\

Using an argument analogous to Step 3 in the proof of Proposition \ref{bilipprop}, we can boost this estimate for $\D_{x}X^t$ to 
\begin{equation}\label{xbiliplocal}
\|\chi_{x^{s_T}, \lambda}(\D_{x}X^t - \lambda^{\frac{3}{4}}I)\|_{H^{\half+}} \ll \lambda^{\frac{3}{4}}.
\end{equation}

We also adapt a similar argument to obtain improvements for $\D_x \Xi^t$. From the bottom row of (\ref{eqn:integratedeqn}), 
$$-\tilde{F}(t)\D_x X^t + \D_x \Xi^t = \lambda^{\frac{3}{4}}\tilde{F}(s_T) + \int_{s_T}^t (\tilde{F} \tilde{H}_{\xi x} - \tilde{G})(r)\D_x X^r + (\tilde{F} \tilde{H}_{\xi\xi} - \tilde{H}_{x \xi})(r) \D_x \Xi^r \, dr.$$
Multiplying by $\chi_{x^{s_T}, \lambda}(x)$, applying $H^{\half+}$ norms, using the algebra property, and applying Gronwall, 
\begin{align}
\label{finalinequality} \|\chi_{x^{s_T}, \lambda}\D_x \Xi^t\|_{H^{\half+}} \lesssim \ &(\|\chi_{x^{s_T}, \lambda}(\lambda^{\frac{3}{4}}\tilde{F}(s_T) + \tilde{F}(t)\D_x X^t)\|_{H^{\half+}} \\
&+ \|\chi_{x^{s_T}, \lambda}(\tilde{F} \tilde{H}_{\xi x} - \tilde{G})(r)\D_x X^r\|_{L_r^1(I;H^{\half+})}) \nonumber \\
&\cdot \exp\left(\int_{s_T}^t \|\tilde{\chi}_{x^{s_T}, \lambda}(\tilde{F} \tilde{H}_{\xi\xi} - \tilde{H}_{x \xi})(r)\|_{H^{\half+}_x} \, dr \right). \nonumber
\end{align}
The exponential term is estimated similar to the exponential term before. The first and second terms are also similar, using additionally the uniform estimate on $\D_{x}X^t$ from (\ref{xbiliplocal}). We conclude
$$\|\chi_{x^{s_T}, \lambda}\D_x \Xi^t\|_{H^{\half+}} \leq \lambda^{\frac{3}{4}}(cT^{-\half} + \lambda^{-\eps}\FF(M(t))(Z(t) + \LL(t, x^t, \xi^t, \lambda))).$$
Integrating in time, we obtain
$$\|\chi_{x^{s_T}, \lambda}\D_x \Xi^t\|_{L_t^2(I; H_x^{\half+})} \ll \lambda^{\frac{3}{4}}$$
as desired. 

\

The third estimate is similar to the second, but using $L^2_x$ in the place of $H_x^\half$. In particular, in estimating the right hand side of (\ref{finalinequality}), we use the (global) $L^2$ estimate of Corollary \ref{finalstructure},
$$\|\tilde{F}(t)\|_{L^2} \leq \lambda^{-\frac{3}{8} -} \FF(M(t)) \leq \lambda^{-\frac{3}{8} -} \FF(T)$$
which provides the relative gain, and further is uniform in time.
\end{proof}

It remains to establish the following local estimate (\ref{xilemma}) on $\xi^t$:

\begin{lem}\label{xiregularitylem}
Let $T = (x_0, \xi_0)$ with $|\xi_0| \in [\lambda/2, 2\lambda]$. There exists $s_T \in I$ such that for solutions $(x^t, \xi^t)$ to (\ref{hamilton}) with initial data $(x, \xi)$ at initial time $s_T$, and satisfying $|\xi| \in [\lambda/2, 2\lambda]$,
\begin{align*}
\|\chi_{x^{s_T}, \lambda}\xi^t\|_{H^{\half+}} \lesssim \lambda^{1+}
\end{align*}
where we denote $x^{s_T} = x^{s_T}_{s_0}(x_0, \xi_0)$.
\end{lem}

\begin{proof}
Beginning with
$$\xi^t = \xi_0 - \int_{s_T}^t H_x(r) \, dr = \xi - \int_{s_T}^t (\D_x V_\lambda)(r) \xi^r + (\D_x \sqrt{a_\lambda})(r) |\xi^r|^\half \, dr,$$
multiply by $\chi_{x^{s_T}, \lambda}(x)$, apply $H^{\half+}$ norms, and use the algebra property to obtain
\begin{align*}
\|\chi_{x^{s_T}, \lambda}\xi^t\|_{H^{\half+}} \lesssim \xi\|\chi_{x^{s_T}, \lambda}\|_{H^{\half+}} + \int_{s_T}^t &(\|\tilde{\chi}_{x^{s_T},\lambda}(\D_x V_\lambda)(r)\|_{H^{\half+}} \\
&\quad + \|\tilde{\chi}_{x^{s_T}, \lambda}(\D_x \sqrt{a_\lambda})(r)\|_{H^{\half+}} \lambda^{-\half})\|\chi_{x^{s_T}, \lambda} \xi^r\|_{H^{\half+}}  \, dr.
\end{align*}
Applying Gronwall, we obtain
\begin{align*}
\|\chi_{x^{s_T}, \lambda}\xi^t\|_{H^{\half+}} &\lesssim \lambda^{1+} \exp\left( \int_{s_T}^t (\|\tilde{\chi}_{x^{s_T},\lambda}(\D_x V_\lambda)(r)\|_{H^{\half+}}  + \|\tilde{\chi}_{x^{s_T}, \lambda}(\D_x \sqrt{a_\lambda})(r)\|_{H^{\half+}} \lambda^{-\half}) \, dr \right).
\end{align*}

We estimate the right hand side as in the proof of Proposition \ref{localprop} to obtain the lemma.
\end{proof}

\subsection{The eikonal equation}

In the following, we let $y$ denote the spatial variable, so that in particular we may let $T = (x, \xi)$ without conflict. Further let $s_T \in I$ be chosen as in Lemma \ref{bilipproplocal}, and
$$(x^{s_T}, \xi^{s_T}) = (x^{s_T}_{s_0}(x, \xi), \xi^{s_T}_{s_0}(x, \xi))$$
where $s_0 \in I$ is chosen as in Proposition \ref{bilipprop}. Consider the eikonal equation,
\begin{equation}\label{eikonal}
\D_t \psi_{x, \xi}(t, y) + \tilde{H}(t, y, \D_y \psi_{x, \xi}(t, y)) = 0, \qquad \psi_{x, \xi}(s_T, y) = \xi^{s_T}(y - x^{s_T}),
\end{equation}
whose characteristics are given by the Hamilton flow. The fact that the Hamilton flow is bilipschitz corresponds to two derivatives on the eikonal solution:

\begin{prop}\label{eikonalbd}
Let $T = (x, \xi)$ with $|\xi| \in [\lambda/2, 2\lambda]$. Then for a solution $\psi_{x, \xi}$ to (\ref{eikonal}),
\begin{align*}
\|\chi_{x^t_{s_0}(x, \xi), \lambda} \D_y^2 \psi_{x, \xi}\|_{L^2_t(I;H_y^{\half+})} \ll \lambda^{\frac{3}{2}}, \\
\|\chi_{x^t_{s_0}(x, \xi), \lambda} \D_y^2 \psi_{x, \xi}\|_{L^\infty_t(I;L_y^2)} \ll \lambda^{\frac{9}{8}-}.
\end{align*}
\end{prop}
\begin{proof}
Using the characteristics for $\psi_{x, \xi}$, write
$$\xi^t(z, \xi^{s_T}) = (\D_y \psi_{x, \xi})(t, x^t(z, \xi^{s_T}))$$
so that
$$\D_x \xi^t(z, \xi^{s_T}) = (\D_y^2 \psi_{x, \xi})(t, x^t(z, \xi^{s_T})) (\D_{x} x^t)(z, \xi^{s_T}).$$
By Lemma \ref{bilipproplocal}, $|(\D_{x} x^t)(z, \xi^{s_T})| > \half$ on the support of  $\chi_{x^{s_T}, \lambda}(z)$, so we may write
\begin{align*}
\chi_{x^{s_T}, \lambda}(z)(\D_y^2 \psi_{x, \xi})(t, x^t(z, \xi^{s_T})) = \chi_{x^{s_T}, \lambda}(z)(\D_x \xi^t)(z, \xi^{s_T})(\tilde{\chi}_{x^{s_T}, \lambda}(z)(\D_{x} x^t)(z, \xi^{s_T}))^{-1}.
\end{align*}
Using the algebra property of $H^{\half+}$, as well as a Moser estimate on the second term, we obtain by Proposition \ref{localprop},
$$\|\chi_{x^{s_T}, \lambda}(z)(\D_y^2 \psi_{x, \xi})(t, x^t(z, \xi^{s_T}))\|_{L^2_t(I;H_z^{\half+})} \ll \lambda^{\frac{3}{2}}.$$
It is convenient to exchange $\chi$ with $\tilde{\chi}$, with the straightforward modifications to the proof:
$$\|\tilde{\chi}_{x^{s_T}, \lambda}(z)(\D_y^2 \psi_{x, \xi})(t, x^t(z, \xi^{s_T}))\|_{L^2_t(I;H_z^{\half+})} \ll \lambda^{\frac{3}{2}}.$$

Similar to the proof of Proposition \ref{localprop}, we may use the bilipschitz property to insert a weight, followed by the algebra property to obtain
\begin{align*}
\|\chi_{x^t(x^{s_T}, \xi^{s_T}), \lambda}& (x^t(z, \xi^{s_T}))(\D_y^2 \psi_{x, \xi})(t, x^t(z, \xi^{s_T}))\|_{L^2_t(I;H_z^{\half+})} \\
&= \|\chi_{x^t(x^{s_T}, \xi^{s_T}), \lambda} (x^t(z, \xi^{s_T}))\tilde{\chi}_{x^{s_T}, \lambda}(z) (\D_y^2 \psi_{x, \xi})(t, x^t(z, \xi^{s_T}))\|_{L^2_t(I;H_z^{\half+})} \\
&\ll \lambda^{\frac{3}{2}}. 
\end{align*}
Lastly, setting $y = x^t(z, \xi^{s_T})$ and combining the diffeomorphism estimate Lemma \ref{compositionlem} with the bilipschitz property of Lemma \ref{bilipproplocal}, we obtain the first estimate of the proposition.

\ 

The proof of the $L^2$ estimate is similar, using the corresponding third estimate of Proposition \ref{localprop}.
\end{proof}

As a consequence, we observe that $\D_y \psi_{x, \xi}$ and $\xi^t(x,\xi)$ are comparable when following a characteristic on the $\lambda^{-\frac{3}{4}}$ spatial scale:

\begin{cor}\label{xidifference}
Let $T = (x, \xi)$ with $|\xi| \in [\lambda/2, 2\lambda]$. Then for solutions $\psi_{x, \xi}$ to (\ref{eikonal}),
$$\|\chi_{x^t_{s_0}(x, \xi), \lambda}(y)(\D_y \psi_{x, \xi}(t, y) - \xi^t_{s_0}(x,\xi))\|_{L_y^\infty} \ll \lambda^{\frac{3}{4}-}.$$
\end{cor}

\begin{proof}
We write $(x^t, \xi^t) = (x^t_{s_0}, \xi^t_{s_0})$. Using the fundamental theorem of calculus, write
$$\D_y \psi_{x, \xi}(t, y) - \xi^t(x, \xi) = \D_y \psi_{x, \xi}(t, y) - \D_y \psi_{x, \xi}(t, x^t(x, \xi))= \int_{x^t}^y \D_y^2 \psi_{x, \xi}(t, z) \, dz$$
so that using Cauchy-Schwarz followed by Proposition \ref{eikonalbd},
$$|\D_y \psi_{x, \xi}(t, y) - \xi^t(x,\xi)| \leq |y - x^t(x, \xi)|^\half \|\D_y^2 \psi_{x, \xi}\|_{L_y^2([x^t, y])} \ll \lambda^{\frac{9}{8}-}|y - x^t(x, \xi)|^\half .$$
Then restricting to $|y - x^t(x, \xi)| \approx \lambda^{-\frac{3}{4}}$ via the cutoff $\chi_{x^t, \lambda}$, we have the desired estimate.
\end{proof}

\section{Wave Packet Parametrix}\label{sec:parametrix}

In this section we construct the wave packet parametrix. Define the index set
$$\TT = \{T = (x, \xi) \in \lambda^{-\frac{3}{4}}\Z \times \lambda^{\frac{3}{4}}\Z : |\xi| \in [\lambda/2, 2\lambda]\}.$$
Then a \emph{wave packet} $u_T = u_{x, \xi}$ centered at $(x, \xi)$ is a function of the form
$$u_T(t, y) = \lambda^{\frac{3}{8}}\chi_T(t, y)e^{i\psi_{T}(t, y)}$$
where 
$$\chi_T(t, y) = \chi(\lambda^{\frac{3}{4}}(y - x^t_{s_0}(x, \xi)))$$
and
$$\psi_T(t, y) = \psi_{x, \xi}(t, y)$$
solves the eikonal equation (\ref{eikonal}). Also note that we write
$$\chi'_T = (\chi')_T = (\chi')(\lambda^{\frac{3}{4}}(y - x^t(x, \xi)))$$
so that in particular
$$\lambda^{\frac{3}{8}}\chi'_T(t, y)e^{i\psi_{T}(t, y)}$$
is a wave packet.

\subsection{Approximate solution}

A wave packet is an approximate solution to
$$(D_t + H(t, y, D))u = 0$$
in the follow sense:

\begin{prop}\label{approxsoln}
Let $u_T$ be a wave packet. Then we may write $(D_t + H(t, y, D))S_\lambda u_T$ as a sum of terms, each taking one of the following forms:
\begin{enumerate}[i)]
\item $A(t, y)p_T(D)v_T(t, y)e^{i\psi_T(t, y)}$ where $\|A(t)\|_{L_y^\infty} \leq \FF(T)$, $|p_T^{(N)}(\eta)| \lesssim \lambda^{-N}$, $v_T = v_T\tilde{\chi}_T$, and 
\begin{align*}
\|v_T(t)\|_{L^2_y} \leq \FF(M(t))Z(t), \quad \|\D_y v_T(t)\|_{L^2_y} &\leq \lambda^{\frac{3}{4}}\FF(M(t))Z(t).
\end{align*}

\item $A(t, y)p_T(D)v_T(t, y)e^{i\psi_T(t, y)}$ as before except $v_T = v_T\tilde{\chi}_T$ instead satisfies
\begin{align*}
\|v_T(t)\|_{L^2_y} \ll 1, \quad \||D|^\half v_T(t)\|_{L^2_t(I;L_y^2)} &\ll \lambda^{\frac{3}{8}}.
\end{align*}

\item $[S_\lambda,  V_\lambda]v_Te^{i\psi_T}$ with $v_T = v_T\tilde{\chi}_T$ satisfying
\begin{align*}
\|v_T(t)\|_{L^2_y} \lesssim \lambda, \quad \|\D_y v_T(t)\|_{L^2_y} &\lesssim \lambda^{1 + \frac{3}{4}}.
\end{align*}

\item $[S_\lambda,  \sqrt{a_\lambda}]v_Te^{i\psi_T}$ with $v_T = v_T\tilde{\chi}_T$ satisfying
\begin{align*}
\|v_T(t)\|_{L^2_y} \lesssim \lambda^{\half}, \quad \|\D_y v_T(t)\|_{L^2_y} &\lesssim \lambda^{\frac{1}{2} + \frac{3}{4}}.
\end{align*}
\end{enumerate}

\end{prop}

\begin{proof}

\emph{Step 1.} First we compute and arrange the error terms. By a direct computation using (\ref{hamilton}) and (\ref{eikonal}),
\begin{align*}
(D_t + H(t, y, D))S_\lambda u_T = \ &(H(t, y, D)S_\lambda  - S_\lambda H(t, y, \D_y \psi_T(t, y)) \\
&- S_\lambda H_\xi(t, x^t, \xi^t)(D_y - \D_y \psi_T(t, y)))u_T.
\end{align*}

To organize the right hand side, first we exchange the coefficient $H_\xi(t, x^t, \xi^t)$ for 
$$H_\xi(t, y, \xi^t)$$
with the goal of separating the spatial variable $y$ of $H$ from the Taylor expansion in $\xi$. Writing
$$v_{T, 1} := H_\xi(t, y, \xi^t) - H_\xi(t, x^t, \xi^t),$$
we have
\begin{align*}
(D_t + H(t, y, D))S_\lambda u_T = \ &(H(t, y, D)S_\lambda  - S_\lambda H(t, y, \D_y \psi_T(t, y)) \\
&- S_\lambda H_\xi(t, y, \xi^t)(D_y - \D_y \psi_T(t, y)))u_T - i\lambda^{\frac{3}{4}}S_\lambda v_{T, 1} \lambda^{\frac{3}{8}}\chi_T'e^{i\psi_T}.
\end{align*}

The right hand side may be viewed as a symbol expansion of $H(t, y, \eta)$ around $\D_y \psi_T(t, y)$. However, this center is $y$-dependent. To remedy this, we can shift this center to $\xi^t$ by defining the error
$$V_{T, 2}u_T := (H(t, y, \xi^t) + H_\xi(t, y, \xi^t)(\D_y \psi_T(t, y)- \xi^t) - H(t, y, \D_y \psi_T(t, y)))u_T$$
and writing
\begin{align*}
(D_t + H(t, y, D))S_\lambda u_T = \ &(H(t, y, D)S_\lambda  - S_\lambda H(t, y, \xi^t) - S_\lambda H_\xi(t, y, \xi^t)(D_y - \xi^t))u_T \\
& - i\lambda^{\frac{3}{4}}S_\lambda  v_{T, 1}\lambda^{\frac{3}{8}}\chi_T' e^{i\psi_T} + S_\lambda V_{T, 2}u_T.
\end{align*}

Observe that on the right hand side, we now have a symbol expansion of $H(t, y, \eta)$ in $\eta$ to second order, modulo commutators with $S_\lambda$, and in particular, the transport term $V_\lambda \xi$ of $H$ vanishes. In turn, this allows us to separate variables in $H$. More precisely, writing the dispersive term with the notation
$$\sqrt{a_\lambda |\eta|} = \sqrt{a_\lambda}p(\eta),$$
we may write, using the Lagrange remainder and observing that $p''(\eta) = |\eta|^{-\frac{3}{2}}$,
\begin{align*}
(H(t, y, D)S_\lambda - S_\lambda H(t, y, \xi^t) &- S_\lambda H_\xi(t, y, \xi^t)(D_y - \xi^t))u_T \\
= \ &\sqrt{a_\lambda(y)}(p(D) - p(\xi^t) - p'(\xi^t)(D_y - \xi^t))S_\lambda u_T \\
&- [S_\lambda,  H(t, y, \xi^t) + H_\xi(t, y, \xi^t)(D_y - \xi^t)]u_T\\
= \ &\sqrt{a_\lambda(y)}S_\lambda |q_T(D)|^{-\frac{3}{2}}(D - \xi^t)^2u_T \\
&- [S_\lambda,  H(t, y, \xi^t) + H_\xi(t, y, \xi^t)(D_y - \xi^t)]u_T.
\end{align*}
By a routine computation, 
\begin{align*}
(D - \xi^t)^2u_T = \ &-\lambda^{\frac{3}{2}}\lambda^{\frac{3}{8}}\chi_T''e^{i\psi_{T}} - 2i\lambda^{\frac{3}{4}}(\D_y \psi_T -\xi^t)\lambda^{\frac{3}{8}}\chi_T' e^{i\psi_{T}} + (\D_y \psi_T -\xi^t)^2 \lambda^{\frac{3}{8}}\chi_T e^{i\psi_{T}} \\
&- i\D_y^2 \psi_T \lambda^{\frac{3}{8}}\chi_Te^{i\psi_{T}}.
\end{align*}

Since $V_{T, 2}$ also takes the form of a Taylor expansion, an analogous analysis applies. However, it is convenient to instead use the integral form of the remainder:
\begin{align*}
V_{T, 2} &= \sqrt{a_\lambda(y)} (p(\xi^t) + p'(\xi^t)(\D_y \psi_T(t, y)- \xi^t) - p(\D_y \psi_T(t, y))) \\
&= -\sqrt{a_\lambda(y)}\int_{\xi^t}^{\D_y \psi_T(t, y)} |\eta|^{-\frac{3}{2}} (\D_y \psi_T(t, y) - \eta) \, d\eta \\
&=: \sqrt{a_\lambda(y)}v_{T, 2}.
\end{align*}

Lastly, write $p_T(\eta) = \lambda^{\frac{3}{2}}|q_T(\eta)|^{-\frac{3}{2}}$, observing that on the support the symbol of $S_\lambda$, we have $|p_T^{(N)}(\eta)| \lesssim \lambda^{-N}$. We conclude
\begin{align}
\label{onepaxerror}
(D_t + H(t, y, D))S_\lambda u_T = \ &\sqrt{a_\lambda(y)}S_\lambda p_T(D) (-\lambda^{\frac{3}{8}}\chi_T''e^{i\psi_{T}} - 2i\lambda^{-\frac{3}{4}}(\D_y \psi_T -\xi^t)\lambda^{\frac{3}{8}}\chi_T' e^{i\psi_{T}} \\
&+ \lambda^{-\frac{3}{2}}(\D_y \psi_T -\xi^t)^2 \lambda^{\frac{3}{8}}\chi_T e^{i\psi_{T}} - i\lambda^{-\frac{3}{2}}\D_y^2 \psi_T \lambda^{\frac{3}{8}}\chi_Te^{i\psi_{T}}) \nonumber \\
&- i\lambda^{\frac{3}{4}} S_\lambda v_{T, 1} \lambda^{\frac{3}{8}}\chi_T'e^{i\psi_T} + \sqrt{a_\lambda(y)}S_\lambda v_{T, 2}u_T \nonumber \\
&- [S_\lambda,  H(t, y, \xi^t) + H_\xi(t, y, \xi^t)(D_y - \xi^t)]u_T \nonumber\\
&+ [S_\lambda, \sqrt{a_\lambda(y)}]v_{T, 2}u_T. \nonumber
\end{align}

\emph{Step 2.} Next, we check that the non-commutator terms on the right hand side of (\ref{onepaxerror}) may be put into the desired form with estimates.

We show the first three terms on the right hand side of (\ref{onepaxerror}) (from the first and second rows) may be placed in case $(i)$. We bound $\sqrt{a_\lambda(y)}$ using Proposition \ref{taylorbd}. Using Corollary \ref{xidifference}, we have
$$\|(\D_y \psi_T -\xi^t)\tilde{\chi}_T\|_{L^\infty_y} \ll \lambda^{\frac{3}{4}},$$
and we easily have 
$$\|\lambda^{\frac{3}{8}}(\chi_T,\chi_T',\chi''_T)\|_{L^2_y} \lesssim 1.$$
It is easy to check that the coefficients satisfy the $L^2$ estimates desired of $v_T$ by using the widened cutoff. For instance, we have
$$\|(\D_y \psi_T -\xi^t)\lambda^{\frac{3}{8}}\chi_T'\|_{L^2_y} \lesssim \|(\D_y \psi_T -\xi^t)\tilde{\chi}_T\|_{L^\infty_y}\|\lambda^{\frac{3}{8}}\chi_T'\|_{L^2_y} \ll \lambda^{\frac{3}{4}}.$$
To estimate the first derivatives, use also Proposition \ref{eikonalbd},
$$\|(\D_y^2 \psi_T) \chi_T\|_{L^2_y} \ll \lambda^{9/8},$$
and likewise
$$\|\D_y \lambda^{\frac{3}{8}}(\chi_T,\chi_T',\chi''_T)\|_{L^2_y} \lesssim \lambda^{\frac{3}{4}}.$$
Thus we also obtain the $L^2$ estimates desired of $\D_y v_T$. 

Consider the fourth term on the right hand side of (\ref{onepaxerror}) (from the second row), with $\D_y^2 \psi_T$. The two estimates of Proposition \ref{eikonalbd} yield both estimates desired of $v_T$ for case $(ii)$.

Next, we place the fifth term of (\ref{onepaxerror}), with 
$$v_{T, 1} = H_\xi(t, y, \xi^t) - H_\xi(t, x^t, \xi^t),$$
in case $(i)$. Using the fundamental theorem of calculus,
$$H_\xi (t, y, \xi^t) - H_\xi(t, x^t, \xi^t) = \int^y_{x^t} H_{x\xi}(t, z, \xi^t) \, dz.$$
Restricting to the support of $\chi_T$ and thus $|y - x^t| \lesssim \lambda^{-\frac{3}{4}}$, and using the uniform bound on $H_{x\xi}$ from Lemma \ref{basicHbd}, we have 
\begin{align*}
\|(H_\xi (t, y, \xi^t) - H_\xi(t, x^t, \xi^t))\chi'_T\|_{L^\infty_y} \leq \lambda^{-\frac{3}{4}}\FF(M(t))Z(t)
\end{align*}
and thus, using a widened cutoff as before, the $L^2$ estimate desired of $v_T$. The first derivative is easier, using directly Lemma \ref{basicHbd}. 

We place the the sixth term of (\ref{onepaxerror}), with 
$$v_{T, 2} = \int_{\xi^t}^{\D_y \psi_T(t, y)} |\eta|^{-\frac{3}{2}} (\D_y \psi_T(t, y) - \eta) \, d\eta,$$
in case $(i)$. Since
$$\xi^t(z, \xi) = (\D_y \psi_T)(t, x^t(z, \xi)),$$
by Lemma \ref{freqpres}, we may restrict to $\eta \approx \lambda$. Also restricting to the support of $\chi_T$ and thus $|y - x^t| \lesssim \lambda^{-\frac{3}{4}}$, we have using Corollary \ref{xidifference} twice,
$$\|v_{T, 2}\chi_T\|_{L^\infty_y} \ll \lambda^{-\frac{3}{2}}\lambda^{\frac{3}{4}}\lambda^{\frac{3}{4}} = 1$$
and thus the desired $L^2_y$ estimate using a widened cutoff. It remains to estimate the first derivative. By a direct compuation,
$$\D_y v_{T, 2} = \int_{\xi^t}^{\D_y \psi_T(t, y)}|\eta|^{-\frac{3}{2}} \D_y^2\psi_T(t, y) \, d\eta.$$
Then (restricting as usual to $|y - x^t| \lesssim \lambda^{-\frac{3}{4}}$) using Proposition \ref{eikonalbd} to estimate $\D_y^2\psi_T$ in $L^2_y$, and Corollary \ref{xidifference} to estimate the width of the limit of integration,
$$\|(\D_y v_{T, 2})\lambda^{\frac{3}{8}}\chi_T\|_{L^2} \ll \lambda^{-\frac{3}{2}}\lambda^{9/8}\lambda^{\frac{3}{4}}\lambda^{\frac{3}{8}} = \lambda^{\frac{3}{4}}$$
as desired.

\emph{Step 3.} It remains to consider the commutator terms on the right hand side of (\ref{onepaxerror}). There are a total of five, as the Hamiltonian contains two terms.

First, we have
$$[S_\lambda,  V_\lambda \xi^t]u_T = [S_\lambda,  V_\lambda]\xi^t\lambda^{\frac{3}{8}}\chi_T e^{i\psi_T}.$$
Then $v_T = \xi^t\lambda^{\frac{3}{8}}\chi_T$ satisfies, by Lemma \ref{freqpres},
$$\|v_T\|_{L^2} \lesssim \lambda, \quad \|\D_y v_T\|_{L^2} \lesssim \lambda \lambda^{\frac{3}{4}}.$$

Second, we have
$$[S_\lambda,  \sqrt{a_\lambda}|\xi^t|^\half]u_T = [S_\lambda, \sqrt{a_\lambda}]|\xi^t|^\half\lambda^{\frac{3}{8}}\chi_T e^{i\psi_T}$$
so that $v_T = |\xi^t|^\half \lambda^{\frac{3}{8}}\chi_T$ satisfies 
$$\|v_T\|_{L^2} \lesssim \lambda^{\half}, \quad \|\D_y v_T\|_{L^2} \lesssim \lambda^{\half}\lambda^{\frac{3}{4}}.$$

Third,
\begin{align*}
[S_\lambda,  V_\lambda(D_y - \xi^t)]u_T &= [S_\lambda,  V_\lambda](D_y - \xi^t)(\lambda^{\frac{3}{8}}\chi_T e^{i\psi_T}) \\
&= [S_\lambda,  V_\lambda](-i\lambda^{\frac{3}{4}}\lambda^{\frac{3}{8}}\chi_T' e^{i\psi_T} + \lambda^{\frac{3}{8}}\chi_T (\D_y \psi_T - \xi^t) e^{i\psi_T}).
\end{align*}
An analysis similar to that in Step 2 shows that these are better than needed for case $(iii)$. The fourth term,
\begin{align*}
[S_\lambda,  \sqrt{a_\lambda} |\xi^t|^{-\frac{3}{2}} \xi^t (D_y - \xi^t)]u_T = [S_\lambda, \sqrt{a_\lambda}]|\xi^t|^{-\frac{3}{2}} \xi^t(&-i\lambda^{\frac{3}{4}}\lambda^{\frac{3}{8}}\chi_T' e^{i\psi_T}\\
&+ \lambda^{\frac{3}{8}}\chi_T (\D_y \psi_T - \xi^t) e^{i\psi_T})
\end{align*}
is similarly better than needed.

Lastly, the fifth term with $v_{T, 2}$ is similar to the analysis of the corresponding term in Step 2, and again is better than needed.

\end{proof}

\subsection{Orthogonality}

In this subsection we observe that a collection of wave packets is orthogonal in an appropriate sense. First, we establish the orthogonality of functions in a form as described in the following proposition:

\begin{prop}\label{orthogprelim}
Let $\{U_T\}$ be functions of the form
$$U_T(t, y) = A(t, y)p_T(D)v_T(t, y) e^{iy\xi^t(x, \xi)}$$
where $T = (x, \xi) \in \TT$, $\|A(t)\|_{L_y^\infty} \leq \FF(T)$, $|p_T^{(N)}(\eta)| \lesssim \lambda^{-N}$, and $v_T = v_T\tilde{\chi}_T$. Then
\begin{align*}
\|\sum_{T \in \TT} U_T(t)\|_{L_{y}^2}^2 \lesssim \lambda^{-\frac{3}{4}}(\log \lambda)\FF(T)\sum_{T \in \TT} \|v_T\|_{H^\half_y}^2.
\end{align*}
\end{prop}
\begin{proof}

First note that $A$ is independent of $T$ and thus may be factored out and estimated immediately. As a result, we may assume $A \equiv 1$ below.

\emph{Step 1.} First we reduce the sum over $T = (x, \xi) \in \TT$ to fixed $x$. Consider $\xi \in \{\lambda^{\frac{3}{4}}\Z : |\xi| \in [\lambda/2, 2\lambda]\}$ and $k \in \lambda^{-\frac{3}{4}}\Z$. Consider two packets $(x_1, \xi), (x_2, \xi) \in \TT$ intersecting $(t, k)$. By Proposition \ref{bilipprop}, we have
$$\lambda^{-\frac{3}{4}} \geq |x^t(x_1, \xi) - x^t(x_2, \xi)| = |\D_x x^t(z, \xi)(x_1 - x_2)| \approx |x_1 - x_2|.$$
Since $x_i \in \lambda^{-\frac{3}{4}}\Z$, we conclude that among packets with frequency $\xi$, there is at most an absolute number whose supports intersect $(t, k)$. For simplicity, we assume there is at most one such packet. Then we may index the packets $\TT$ by $(k, \xi)$.

Write
$$\sum_{T \in \TT} U_T(t) =: \sum_{k} \sum_\xi U_T(t) =: \sum_{k} U^k(t).$$
Note that the $U^k(t)$ have essentially finite overlap. Thus, using polynomial weights and Cauchy-Schwarz,
$$\|\sum_T U_T(t)\|_{L_{y}^2}^2 \lesssim \sum_{k} \| U^k(t)\|_{L_{y}^2}^2$$
so that it suffices to show the orthogonality with fixed $k$. In other words, we may assume $\TT$ consists of packets intersecting $(t, k)$, each with distinct $\xi$.

\emph{Step 2.} We abuse notation by equating $\xi = (x, \xi) = T$ as indices, and denote $\xi^t = \xi^t(x, \xi)$. By Plancherel's and Cauchy-Schwarz,
\begin{align*}
\|\sum_\xi U_T(t)\|_{L_y^2}^2 &= \|\sum_\xi \widehat{U}_T(t)\|_{L_\eta^2}^2 =  \int |\sum_\xi \widehat{U}_T(t)\langle\eta - \xi^t\rangle^\half \langle\eta - \xi^t\rangle^{-\half}|^2 \, d\eta \\
&\lesssim  \int \sum_\xi |\widehat{U}_T(t)|^2\langle\eta - \xi^t\rangle \sum_\xi\langle \eta - \xi^t \rangle^{-1} \, d\eta \\
&\lesssim \left(\sum_\xi \int| \widehat{U}_T(t)|^2\langle\eta - \xi^t\rangle \, d\eta\right) \left(\sup_\eta \sum_\xi\langle\eta - \xi^t\rangle^{-1} \right).
\end{align*}

It remains to estimate the two terms.

\emph{Step 3.} First consider the supremum over $\eta$. Fix any $\eta \in [\lambda/2, 2\lambda]$ and write
$$(k, \eta) = (x^t(z, \zeta), \xi^t(z, \zeta)).$$
Then since
$$|x^t(z, \zeta) - x^t(x, \xi)| = |k - x^t(x, \xi)| \leq \lambda^{-\frac{3}{4}},$$
we may apply Proposition \ref{charoverlap},
$$|\zeta - \xi| \leq 2|\xi^t(z, \zeta) - \xi^t(x, \xi)| + \lambda^{\frac{3}{4}} = 2|\eta - \xi^t| + \lambda^{\frac{3}{4}}.$$

We conclude, using that there are at most $\lambda^{\frac{1}{4}}$ frequencies $\xi$ in $\{|\xi| \in [\lambda/2, 2\lambda]\}$,
$$\sum_\xi \langle \eta - \xi^t \rangle^{-1} \lesssim \sum_\xi \langle \zeta - \xi \rangle^{-1} \lesssim \lambda^{-\frac{3}{4}}\log \lambda.$$
Then take the supremum over $\eta$ to obtain 
$$\sup_\eta \sum_\xi\langle\eta - \xi^t\rangle^{-1} \lesssim \lambda^{-\frac{3}{4}}\log \lambda.$$

\emph{Step 4.} For the sum over $\xi$, we have, using a change of variables and Plancherel's,
\begin{align*}
\int| \widehat{U}_T(t)|^2\langle\eta - \xi^t\rangle  \, d\eta &= \int| p_T(\eta + \xi^t)\widehat{v_T}(t)|^2\langle \eta \rangle \, d\eta \\
&= \|\langle D \rangle^{\half}p_T(D + \xi^t) v_T(t)\|_{L^2_y}^2 \\
&\lesssim  \|\langle D \rangle^{\half} v_T(t)\|_{L^2_y}^2.
\end{align*}

\

Combining the above estimates, we conclude
$$\|\sum_\xi U_T(t)\|_{L_y^2}^2 \leq \lambda^{-\frac{3}{4}}(\log \lambda)\FF(T)\sum_\xi \|v_T\|_{H^\half_y}^2$$
as desired. 
\end{proof}

We then have the following orthogonality of wave packets:

\begin{cor}\label{orthog}
Let $\{U_T\}$ be functions of the form
$$U_T(t, y) = A(t, y)p_T(D)v_T(t, y) e^{i\psi_T(t, y)}$$
where $\|A(t)\|_{L_y^\infty} \leq \FF(T)$, $|p_T^{(N)}(\eta)| \lesssim \lambda^{-N}$, and $v_T = v_T\tilde{\chi}_T$. Then
\begin{align*}
\|\sum_{T \in \TT} U_T(t)\|_{L_{y}^2}^2 \lesssim (\log \lambda)\FF(T)\sum_{T \in \TT} \left(\|v_T(t)\|_{L_y^2}^2 + \lambda^{-\frac{3}{4}}\||D|^\half v_T(t)\|_{L^2_y}^2\right),
\end{align*}
and
\begin{align*}
\|\sum_{T \in \TT} U_T(t)\|_{L_{y}^2}^2 \lesssim (\log \lambda)\FF(T)\sum_{T \in \TT} \left(\|v_T(t)\|_{L_y^2}^2 + \lambda^{-\frac{3}{4}}\|\D_y v_T(t)\|_{L^2_y}\|v_T(t)\|_{L_y^2}\right).
\end{align*}
\end{cor}

\begin{proof}

Apply Proposition \ref{orthogprelim} with, in the place of $v_T$,
$$v_T e^{i(\psi_T - y\xi^t)}.$$

Then write
\begin{align*}
\|v_T e^{i(\psi_T - y\xi^t)}\|_{H^{\half}_y}^2 = \int &(\langle D \rangle v_T e^{i(\psi_T - y\xi^t)} )(\bar{v}_T e^{-i(\psi_T - y\xi^t)} ) \, dy \\
&= \int ([\langle D \rangle, \tilde{\chi}_Te^{i(\psi_T - y\xi^t)}]v_T)(\bar{v}_T e^{-i(\psi_T - y\xi^t)} ) \, dy + \int (\langle D \rangle v_T)\bar{v}_T \, dy.
\end{align*}
The first integral may be estimated using Cauchy-Schwarz and Corollary \ref{xidifference}:
\begin{align*}
\|[\langle D \rangle, \tilde{\chi}_T e^{i(\psi_T - y\xi^t)}]v_T\|_{L^2_y} \|v_T\|_{L^2_y} &\lesssim \|\tilde{\chi}_T(\D_y \psi_T - \xi^t)e^{i(\psi_T - y\xi^t)} + \lambda^{\frac{3}{4}}\tilde{\chi}_T'e^{i(\psi_T - y\xi^t)}\|_{L^\infty} \|v_T\|_{L^2_y}^2 \\
&\lesssim \lambda^{\frac{3}{4}} \|v_T\|_{L^2_y}^2.
\end{align*}
The second integral may be estimated by
$$\|\langle D \rangle^{\half}v_T\|_{L^2_y}^2 \lesssim \|v_T\|_{L^2_y}^2 + \||D|^\half v_T\|_{L^2_y}^2.$$

Combining the above estimates, we conclude
$$\|\sum_\xi U_T(t)\|_{L_y^2}^2 \leq (\log \lambda)\FF(T)\sum_\xi \left(\|v_T(t)\|_{L_y^2}^2 + \lambda^{-\frac{3}{4}}\||D|^\half v_T\|_{L^2_y}^2\right)$$
as desired. 

For the second estimate of the proposition, we instead estimate the second integral above via
$$\langle \langle D \rangle v_T, v_T \rangle \lesssim \|v_T\|_{H^1_y}\|v_T\|_{L^2} \lesssim \|v_T\|_{L^2}^2 + \|\D_y v_T\|_{L^2}\|v_T\|_{L^2}.$$

\end{proof}

Combining Proposition \ref{approxsoln} with Corollary \ref{orthog}, we obtain

\begin{cor}\label{finalapproxsoln}
Let $\{c_T\}_{T \in \T} \in \ell^2(\TT)$ and
$$u = \sum_{T \in \TT} c_T u_T$$
where $u_T$ are wave packets. Then
$$\|(D_t + H(t, y, D))S_\lambda u \|_{L^2(I;L_x^2)}^2 \leq \lambda^{0+}\FF(T) \sum_{T \in \TT} |c_T|^2$$
and
$$\|(D_t + H(t, y, D))S_\lambda u \|_{L^1(I;L_x^2)}^2 \ll  \sum_{T \in \TT} |c_T|^2.$$
\end{cor}

\begin{proof}
The second estimate is immediate from the first, having chosen $T \leq \lambda^{0-}$. The first is obtained by using the estimates on the four types of terms from Proposition \ref{approxsoln} with the matching orthogonality result of Corollary \ref{orthog}. The last two types of terms also require straightforward commutator estimates, to absorb a factor of $\lambda$ into a derivative on $V$ and $\lambda^{\half}$ into half of a derivative on $a$. 
\end{proof}

\subsection{Matching the initial data}

To conclude the parametrix construction, it remains to verify that we may use a linear combination of wave packets to match general initial data. In order to achieve this, it is convenient to further specify our choice of cutoff $\chi$ to satisfy
$$\sum_{m \in \Z} \chi(y - m)^2 = 1$$
so that
$$v_T(y) = \lambda^{\frac{3}{8}}\chi_T(y) e^{iy\xi}$$
forms a tight frame, in that for $f \in L^2(\R)$ with frequency support $\{|\xi| \in [\lambda/2, 2\lambda]\}$, 
$$f = \sum_{T \in \TT} c_T v_T, \quad c_T = \int f(y) \overline{v_T}(y) \, dy.$$
However, as we shall see below, it is convenient to instead define
$$v_T(y) = \lambda^{\frac{3}{8}}\chi_T(y) e^{i(\psi_T(s_0, x) + \xi(y - x))} = \lambda^{\frac{3}{8}}\chi_T(y) e^{i(\psi_T(s_0, x) + (\D_y\psi_T)(s_0, x)(y - x))}$$
which still forms a tight frame.

\begin{prop}\label{matchdata}
Given $u_0 \in L_y^2$ with frequency support $\{|\xi| \in [\lambda/2, 2\lambda]\}$, there exists $\{a_T\}_{T \in \TT} \in \ell^2(\TT)$ such that
$$u_0 = \sum_{T \in \TT} a_T u_T(s_0)$$
and
$$\sum_{T \in \TT} |a_T|^2 \lesssim \|u_0\|_{L_y^2}^2.$$
\end{prop}

\begin{proof}
Write using the above frame,
$$u_0 = \sum_{T \in \TT} c_T v_T.$$
Using these coefficients, construct the linear combination of wave packets
$$\tilde{u} = \sum_{T \in \TT} c_T u_T.$$
We consider the difference $u_0 - \tilde{u}(s_0)$, first observing
$$e^{i(\psi_T(s_0, x) + \xi(y - x))} - e^{i\psi_T(s_0, y)} = e^{i(\psi_T(s_0, x) + \xi(y - x))}(1 - e^{i(\psi_T(s_0, y) - \psi_T(s_0, x) - \xi(y - x))}).$$
Then we apply Proposition \ref{orthogprelim} with, in the place of $v_T$,
$$e^{i(\psi_T(s_0, x) - x\xi)} \chi_T (1 - e^{i(\psi_T(s_0, y) - \psi_T(s_0, x) - (\D_y\psi_T)(s_0, x)(y - x))}).$$

This is small in $L^2_y$ by Proposition \ref{eikonalbd} with a Taylor expansion:
$$\|\chi_T(1 - e^{i(\psi_T(s_0, y) - \psi_T(s_0, x) - (\D_y\psi_T)(s_0, x)(y - x))})\|_{L^2_y} \ll \lambda^{-\frac{3}{8}-}.$$
Similarly, its first derivative is small in $L^2_y$ by using the $L_y^\infty$ estimate of Corollary \ref{xidifference}. We conclude that
$$\|u_0 - \tilde{u}(s_0)\|_{L_y^2}^2 \ll \sum_{T \in \TT} |c_T|^2 = \|u_0\|_{L^2}^2.$$
We then obtain the claim by iterating. 

\end{proof}

In the next subsection, we will use a Duhammel argument to match the source term. To do so, we will need to be able to match initial data at arbitrary time $s \in I$:

\begin{cor}\label{matchanydata}
Given $u_0 \in L_y^2$ with frequency support $\{|\xi| \in [\lambda/2, 2\lambda]\}$ and $s \in I$, there exists $\{a_T\}_{T \in \TT} \in \ell^2(\TT)$ such that
$$u_0 = S_\lambda \sum_{T \in \TT} a_T u_T(s) =: S_\lambda \tilde{u}(s)$$
and
$$\sum_{T \in \TT} |a_T|^2 \lesssim \|u_0\|_{L_y^2}^2.$$
\end{cor}

\begin{proof}

Consider the exact solution $u$ to
$$(\D_t + H)u = 0, \qquad u(s) = u_0.$$
Using Proposition \ref{matchdata}, we may construct 
$$\tilde{u} = \sum_{T \in \TT} a_T u_T$$
satisfying
$$\tilde{u}(s_0) = u(s_0).$$

Using energy estimates with Corollary \ref{finalapproxsoln}, we have
\begin{align*}
\|S_\lambda(\tilde{u} - u)\|_{L_t^\infty(I;L_y^2)} &\lesssim \|(\D_t + H(t, y, D))S_\lambda (\tilde{u} - u)\|_{L_t^1(I;L_y^2)} \\
&\quad + \|[\D_x, V_\lambda] + [|D|^\half, \sqrt{a_\lambda}] S_\lambda(\tilde{u} - u)\|_{L_t^1(I;L_y^2)} \\
&\lesssim T^{\half}(\|(\D_t + H(t, y, D))S_\lambda (\tilde{u} - u)\|_{L_t^2(I;L_y^2)} \\
&\quad + \|[\D_x, V_\lambda] + [|D|^\half, \sqrt{a_\lambda}] S_\lambda(\tilde{u} - u)\|_{L_t^2(I;L_y^2)}) \\
&\lesssim T^{\half}(\lambda^{0+}\FF(T)\|u_0\|_{L_y^2} + \|[H, S_\lambda] u\|_{L_t^2(I;L_y^2)} \\
&\quad + (\|V\|_{L^2_t(I;C^1)} + \|a\|_{L^2_t(I; C^\half)})\|S_\lambda(\tilde{u} - u)\|_{L_t^\infty(I;L_y^2)}) \\
&\leq T^{\half}\lambda^{0+}\FF(T)(\|u_0\|_{L_y^2} + \|u\|_{L_t^\infty(I;L_y^2)} + \|S_\lambda(\tilde{u} - u)\|_{L_t^\infty(I;L_y^2)}) \\
&\leq T^{\half}\lambda^{0+}\FF(T)(\|u_0\|_{L_y^2} + \|S_\lambda(\tilde{u} - u)\|_{L_t^\infty(I;L_y^2)}).
\end{align*}
Choosing $T^\half \leq \lambda^{0-}$ so that $T^{\half}\lambda^{0+}\FF(T) \ll 1$, we conclude
$$\|S_\lambda\tilde{u}(s) - u_0\|_{L_y^2} \ll \|u_0\|_{L_y^2}.$$
Iterating, we obtain the claim.

\end{proof}

\subsection{Matching the source}

We use a Duhammel formula with an iteration argument to match the source term in Proposition \ref{redstrich}:

\begin{prop}\label{finalparametrix}
Consider the solution $u_\lambda$ to
$$(\D_t + H(t, y, D))u_\lambda = f, \qquad u_\lambda(s_0) = u_0$$
where $u_\lambda(t, \cdot)$ has frequency support $\{|\xi| \in [\lambda/2, 2\lambda]\}$. We may write
$$u_\lambda = \tilde{u} + \int_{s_0}^t \tilde{u}_s(t, y) \, ds$$
where $\tilde{u}$ is the construction in Proposition \ref{matchdata}, and 
$$\tilde{u}_s = \sum_{T \in \TT} a_{T,s} u_T$$
with
$$\sum_{T \in \TT} |a_{T,s}|^2 \lesssim \|f(s)\|_{L_y^2}^2 + \|u_0\|_{L_y^2}^2.$$
\end{prop}

\begin{proof}

Let $f \in L_t^1L_y^2$. We construct the function
$$Tf(t, y) = \int_{s_0}^t \tilde{u}_s(t, y) \, ds$$
where $\tilde{u}_s$ is the function constructed in Corollary \ref{matchanydata} with $f(s)$ in the place of $u_0$, thus satisfying
$$S_\lambda \tilde{u}_s(s) = f(s)$$
and
$$\sum_{T \in \TT} |a_{T,s}|^2 \lesssim \|f(s)\|_{L_y^2}^2.$$
Then we see
$$(\D_t + H(t, y, D))S_\lambda Tf = f(t, y) + \int_{s_0}^t (\D_t + H(t, y, D))S_\lambda \tilde{u}_s(t, y) \, ds$$
so that
$$\|(\D_t + H(t, y, D))S_\lambda Tf - f\|_{L_t^1(I; L_y^2)} \ll \|f\|_{L_t^1(I;L_y^2)}.$$
Thus, iterating, we see we may write the solution $u_\lambda$ to
$$(\D_t + H(t, y, D))u_\lambda = f, \qquad u_\lambda(s_0) = 0$$
in the form
$$u_\lambda = Tf = \int_{s_0}^t \tilde{u}_s(t, y) \, ds.$$

Using Proposition \ref{matchdata}, we can repeat the above argument to write the solution $u_\lambda$ to
$$(\D_t + H(t, y, D))u_\lambda = f, \qquad u_\lambda(s_0) = u_0$$
in the form
$$u_\lambda = \tilde{u} + Tf.$$

\end{proof}

\section{Strichartz Estimates}\label{sec:strichartz}

In this section we establish Strichartz estimates on sums of wave packets, as defined in Section \ref{sec:parametrix}.

\subsection{Packet overlap}

We record some estimates on the overlap of the packets.

First, we have a trivial overlap bound:
\begin{prop}\label{1ptoverlap}
Let $t \in I$ and $y \in \R$. There are $\lesssim \lambda^{\frac{1}{4}}$ packets $T = (x, \xi) \in \TT$ intersecting $(t, y)$. 
\end{prop}

\begin{proof}
Consider two packets $(x_1, \xi), (x_2, \xi) \in \TT$ intersecting $(t, y)$. By Proposition \ref{bilipprop}, we have
$$\lambda^{-\frac{3}{4}} \geq |x^t(x_1, \xi) - x^t(x_2, \xi)| = |\D_x x^t(z, \xi)(x_1 - x_2)| \approx |x_1 - x_2|.$$
Since $x_i \in \lambda^{-\frac{3}{4}}\Z$, there are at most an absolute number of such packets. Since there are $\approx \lambda^{\frac{1}{4}}$ frequencies $\xi$, we obtain the claim.
\end{proof}

Second, we have a bound on the number of packets which intersect at two times:
\begin{cor}\label{2ptoverlap}
Let $t, s \in I$ and $y, z \in \R$. There are $\lesssim |t - s|^{-1}$ packets $T = (x, \xi) \in \TT$ intersecting both $(t, y)$ and $(s, z)$. 
\end{cor}

\begin{proof}
Consider two packets $(x_1, \xi_1), (x_2, \xi_2) \in \TT$ intersecting both $(t, y)$ and $(s, z)$, so that we have
$$|x^s(x_1, \xi_1) - x^s(x_2, \xi_2)| \leq \lambda^{-\frac{3}{4}}, \quad |x^t(x_1, \xi_1) - x^t(x_2, \xi_2)| \leq \lambda^{-\frac{3}{4}}.$$

Thus we may apply Proposition \ref{char2pts},
$$|\xi_1 - \xi_2| \lesssim \lambda^{\frac{3}{4}}|t - s|^{-1}.$$
Since $\xi_i \in \{\lambda^{\frac{3}{4}}\Z : |\xi| \in [\lambda/2, 2\lambda]\}$, we obtain the claim.
\end{proof}

\subsection{Counting argument}

We have the following Strichartz estimate on sums of wave packets. Note that we do not consider phase cancellation, so the estimate is established by analyzing packet overlap only. 

\begin{prop}
For each $T = (x, \xi) \in \TT$, let the $u_T = u_T(t, y)$ denote the wave packet centered at $T$. Then
$$\| \sum_{T \in \TT} c_T u_T\|_{L^2(I;L^\infty)}^2 \leq \lambda^{\frac{3}{4}}\FF(T)(\log \lambda)^4\sum_{T \in \TT} |c_T|^2.$$
\end{prop}

\begin{proof}
Normalizing, we may assume $\sum |c_T|^2 = 1$. 

\emph{Step 1. Dyadic pidgeonholing.} By Proposition \ref{charnearpts}, two packets overlap for at least time $\lambda^{-\frac{1}{4}}$ (this should depend on $\FF(T)$ but to simplify our exposition we omit it). As a result, it suffices to estimate the following Riemann sum in $t$, over $\lambda^{\frac{1}{4}}$ times $t_j \in I$ separated by $\lambda^{-\frac{1}{4}}$, and $y_j \in \R$ arbitrary:
$$\sum_{j \in J} \lambda^{-\frac{1}{4}} (\sum_{T \in \TT} |c_T| \chi_T(t_j, y_j))^2 \lesssim (\log \lambda)^4.$$

By the trivial overlap bound Proposition \ref{1ptoverlap},
$$\sum_{T \in \TT} |\chi_T(t_j, y_j)| \lesssim \lambda^{\frac{1}{4}},$$
so we may restrict the sum over $\TT$ to $T$ such that $|c_T| \geq \lambda^{-\frac{1}{4}}$. Since $|c_T| \leq 1$ trivially, we may choose dyadic values $c \in [\lambda^{-\frac{1}{4}}, 1]$ and partition the sum over $\TT$ by grouping the $T$ such that $|c_T| \approx c$. Accepting the logarithmic loss, it suffices to consider the one member of the partition (denoting the subcollection by $\TT_c$):
$$\sum_{j \in J} \lambda^{-\frac{1}{4}} (\sum_{T \in \TT_c} |c_T| \chi_T(t_j, y_j))^2 \lesssim (\log \lambda)^2.$$

On the other hand, again by Proposition \ref{1ptoverlap}, we may choose dyadic values $L \in [1, \lambda^{\frac{1}{4}}]$ and partition the sum over $J$ by grouping the $j$ such that $\approx L$ packets in $\TT_c$ intersect $(t_j, y_j)$. Accepting the logarithmic loss, it suffices to show (denoting the subcollection by $J_L$):
$$\sum_{j \in J_L} \lambda^{-\frac{1}{4}} (\sum_{T \in \TT_c} |c_T| \chi_T(t_j, y_j))^2 \lesssim \log \lambda.$$

If we denote the number of points in $J_L$ by $M$,
$$\sum_{j \in J_L} \lambda^{-\frac{1}{4}} (\sum_{T \in \TT_c} |c_T| \chi_T(t_j, y_j))^2 \lesssim M\lambda^{-\frac{1}{4}}(cL)^2.$$
It thus suffices to show
$$M(cL)^2 \lesssim \lambda^{\frac{1}{4}} \log \lambda.$$

Further, if we denote the number of packets in $\TT_c$ by $N$,
$$N c^2 \approx \sum_{T \in T_c} c_T^2 \leq \sum_{T \in T} c_T^2 = 1$$
so that $N \lesssim c^{-2}$. It thus suffices to show
\begin{equation}\label{countingproblem}
ML^2 \lesssim \lambda^{\frac{1}{4}}N \log \lambda.
\end{equation}

\emph{Step 2. Double counting.} For each packet $T \in \TT_c$, denote by $n_T$ the number of points contained in the support of $u_T$. Note that 
\begin{equation}\label{simplecount}
ML \approx \sum_{T \in \TT_c} n_T = \sum_{n_T = 1} n_T + \sum_{n_T \geq 2} n_T \leq N + \sum_{n_T \geq 2} n_T.
\end{equation}
If 
$$\sum_{n_T \geq 2} n_T \leq N,$$
then we conclude $ML \lesssim N$. Using also the trivial overlap bound Proposition \ref{1ptoverlap} on one copy of $L$, we obtain an estimate even better than (\ref{countingproblem}).

Thus we may assume
$$\sum_{n_T \geq 2} n_T > N.$$
Then (\ref{simplecount}) with Cauchy-Schwarz gives
\begin{equation}\label{doublecount}
M^2 L^2 \lesssim (\sum_{n_T \geq 2} n_T)^2 \lesssim N \sum_{n_T \geq 2} n_T^2 =: NK.
\end{equation}

We estimate $K$ by a double counting. We claim that $K$ counts the triples $(i, j, T) \in J_L \times J_L \times \TT_c$, $i \neq j$, such that the packet $u_T$ intersects $(t_i, y_i)$ and $(t_j, y_j)$. Indeed, if $n_T = 1$, $T$ doesn't contribute to $K$, and otherwise, $T$ contributes $\approx n_T^2$ to $K$. 

On the other hand, by Corollary \ref{2ptoverlap}, each pair of points $(t_i, y_i)$ and $(t_j, y_j)$ is covered by $\lesssim |t_i - t_j|^{-1}$ packets. Thus
$$K \lesssim \sum_{1 \leq i \neq j \leq M} |t_i - t_j|^{-1}.$$
The sum is maximzed when the $t_j$ are as close as possible, as consecutive multiples of $\lambda^{-\frac{1}{4}}$:
$$K \lesssim \lambda^{\frac{1}{4}}\sum_{1 \leq i \neq j \leq M} |i - j|^{-1} \lesssim \lambda^{\frac{1}{4}} M \log M \lesssim M\lambda^{\frac{1}{4}} \log \lambda.$$
Substituting in (\ref{doublecount}), we obtain (\ref{countingproblem}).

\end{proof}

Combined with Proposition \ref{finalparametrix}, we obtain Proposition \ref{redstrich}.

\appendix

\section{H\"older Estimates}\label{holdersec}

We provide an appendix of H\"older estimates on unknowns appearing in the water waves system. Here we provide only the results; details and proofs are provided in \cite{alazard2014strichartz} and \cite{ai2017low}.

Throughout this section, let
$$s > \frac{d}{2} + \half, \qquad 0 < r - 1 < s - \frac{d}{2} - \half, \qquad r \leq \frac{3}{2}, \qquad I = [-1, 0].$$
Also recall that we denote $\nabla = \nabla_x$ and $\Delta = \Delta_x$, and have defined the following spaces for $J \subseteq \R$:
\begin{align*}
X^\sigma(J) &= C_z^0(I;H^{\sigma}) \cap L_z^2(I;H^{\sigma + \half})  \\
Y^\sigma(J) &= L_z^1(I;H^{\sigma}) + L_z^2(I;H^{\sigma - \half}) \\
U^\sigma(J) &= C_z^0(I;C_*^{\sigma}) \cap L_z^2(I;C_*^{\sigma + \half}).  \\
\end{align*}

\subsection{Paralinearization of the Dirichlet to Neumann map}\label{holderonDN}

Recall the Dirichlet to Neumann map is given by solving
\begin{equation}
\Delta_{x, y} \theta = 0, \qquad \theta|_{y = \eta(x)} = f, \qquad \D_n \theta |_\Gamma = 0
\end{equation}
and setting
$$(G(\eta)f)(x) = \sqrt{1 + |\nabla \eta|^2} (\D_n \theta) |_{y = \eta(x)} = ((\D_y - \nabla \eta \cdot \nabla)\theta) |_{y = \eta(x)}.$$
Recall that we write $\Lambda$ for the principal symbol of the Dirichlet to Neumann map. We can estimate the paralinearization error of the Dirichlet to Neumann map in H\"older norm:

\begin{prop}\label{DNparalinear}
Write
$$\Lambda(t, x, \xi) = \sqrt{(1 + |\nabla \eta|^2)|\xi|^2 - (\nabla \eta \cdot \xi)^2}.$$
Then
$$\|G(\eta)f - T_\Lambda f\|_{W^{\half, \infty}} \leq \FF(\|\eta\|_{H^{s + \half}}, \|f\|_{H^s})(1 + \|\eta\|_{W^{r + \half, \infty}})(1 + \|\eta\|_{W^{r + \half, \infty}}+ \|f\|_{W^{r, \infty}}).$$
\end{prop}

We also require the following Sobolev estimates:
\begin{prop}\cite[Theorem 1.4]{alazard2014strichartz}
\label{DNparalineartame}
We have
\begin{align}
\|G(\eta)f - T_\Lambda f\|_{H^{s - \half}} \leq \ &\FF(\|\eta\|_{H^{s + \half}}, \|f\|_{H^s})(1 + \|\eta\|_{W^{r + \half, \infty}}+ \|f\|_{W^{r, \infty}}).
\end{align}
\end{prop}

\begin{prop}\cite[Proposition 3.13]{alazard2014cauchy}
\label{DNparalinearsupertame}
We have
\begin{align}
\|G(\eta)f - T_\Lambda f\|_{H^{s - 1}} \leq \ &\FF(\|\eta\|_{H^{s + \half}}, \|f\|_{H^s}).
\end{align}
\end{prop}

\begin{rem}
The Sobolev estimate here was stated for $s > \frac{d}{2} + \frac{3}{4}$, but using sharper elliptic estimates, this can be reduced to $s > \frac{d}{2} + \half$.
\end{rem}

\subsection{General bottom estimates}\label{bottomsection}

In this section we recall errors that arise due to the presence of a general bottom to our fluid domain. We recall the following identities from Propositions 4.3 and 4.5 in \cite{alazard2014cauchy}:
$$G(\eta)B = -\nabla \cdot V - \Gamma_y$$
$$L \nabla \eta = G(\eta) V + \nabla \eta G(\eta) B + \Gamma_x + \nabla \eta \Gamma_y$$
where
$$\|\Gamma_x\|_{H^{s - \half}} + \|\Gamma_y\|_{H^{s - \half}} \leq \FF(\|\eta\|_{H^{s + \half}}, \|(\psi, V, B)\|_{H^\half}).$$
Here, $\Gamma_{x_i}$ is defined as follows: let $\theta_i$ be the solution to
$$\Delta_{x, y} \theta_i = 0, \qquad \theta_i|_{y = \eta(x)} = V_i, \qquad \D_n \theta_i|_{\Gamma} = 0.$$
Then 
\begin{equation}\label{gammadef}
\Gamma_{x_i} = ((\D_y - \nabla \eta \cdot \nabla) (\D_i \phi - \theta_i)) |_{y = \eta(x)}.
\end{equation}
Note that by the definition of $V$, $\D_i \phi = \theta_i$ if $\Gamma = \emptyset$, so $\Gamma_{x_i}$ is only nonzero in the presence of a bottom. $\Gamma_y$ is defined in the analogous way with $B$ in place of $V$.

We can estimate $\Gamma_x$ and $\Gamma_y$ in H\"older norm:

\begin{prop}\label{roughbottomest}
Consider $\Gamma_x, \Gamma_y$ as defined in (\ref{gammadef}). Then
$$\|\Gamma_x\|_{W^{\half,\infty}} + \|\Gamma_y\|_{W^{\half,\infty}} \leq \FF(\|\eta\|_{H^{s + \half}}, \|(\psi, V, B)\|_{H^\half})(1 + \|\eta\|_{W^{r + \half, \infty}}).$$
\end{prop}

We also require the following Sobolev estimate:
\begin{prop}\cite[Propositions 4.3, 4.5]{alazard2014cauchy}\label{roughbottomesttame}
We have
$$\|\Gamma_y\|_{H^{s - \half}} + \|\Gamma_x\|_{H^{s - \half}} \leq \FF(\|\eta\|_{H^{s + \half}}, \|(\psi, V, B)\|_{H^\half}).$$
\end{prop}

\subsection{Estimates on the Taylor coefficient}

We recall H\"older estimates on the Taylor coefficient:
\begin{prop}[{\cite[Proposition C.1]{alazard2014strichartz}}]\label{taylorbd}
Remain in the setting of Proposition \ref{prop:paralinearization}. Let $0 < \eps < \min(r - 1, s - \frac{d}{2} - \frac{3}{4})$. Then for all $t \in [0, T]$,
$$\|a(t) - g\|_{L^\infty} \lesssim \|a(t) - g\|_{H^{s - \half}} \leq \FF(M(t)),$$
$$\|a(t)\|_{W^{\half + \eps, \infty}} + \|La(t)\|_{W^{\eps, \infty}} \leq \FF(M(t))Z(t).$$
\end{prop}

\begin{rem}
The Sobolev estimate here was stated for $s > \frac{d}{2} + 1$, but the proof is actually valid for $s > \frac{d}{2} + \half$. The H\"older estimates also hold for $s > \frac{d}{2} + \half$, but require sharper arguments.
\end{rem}

We have the following straightforward consequence:
\begin{cor}\label{gammabd}
Remain in the setting of Proposition \ref{taylorbd} and fix multi-index $\beta$. Then for all $t \in [0, T]$, uniformly on $\{|\xi| = 1\}$,
$$\|\D_\xi^\beta\gamma(t, \cdot, \xi)\|_{L^\infty} \leq \FF(M(t)),$$
$$\|\D_\xi^\beta \gamma(t, \cdot, \xi)\|_{W^{\half + \eps, \infty}} \leq \FF(M(t))Z(t).$$
\end{cor}

\subsection{Vector field commutator estimate}

We record a version of Lemma 2.17 in \cite{alazard2014cauchy} for H\"older spaces:

\begin{prop}\label{com}
For any $t \in I$, $r \geq 0$, and $\eps > -r$, 
\begin{align*}
\|\left((\D_t + V \cdot \nabla )T_p - T_p(\D_t + T_V \cdot \nabla)\right)&u(t)\|_{W^{r, \infty}} \\
\lesssim\ &M_0^m(p(t)) \|V(t)\|_{W^{1, \infty}} \|u(t)\|_{B^{r + m}_{\infty ,1}} \\
&+ M_0^m(p(t))\|V(t)\|_{W^{r  + \eps, \infty}}\|u(t)\|_{B_{\infty, 1}^{1 + m - \eps}} \\
&+ M_0^m(\D_t p(t) + V \cdot \nabla p(t)) \|u(t)\|_{W^{r + m,\infty}}.
\end{align*}
\end{prop}

\section{Paradifferential Calculus}

For the reader's convenience, we provide an appendix of notation and estimates from Bony's paradifferential calculus. This is a subset of the appendix in \cite{alazard2014strichartz}.

\subsection{Notation}\label{paracalcnotation}

For $\rho = k + \sigma, \ k \in \N, \ \sigma \in (0, 1)$, denote by $W^{\rho, \infty}(\R^d)$ the space of functions whose derivatives up to order $k$ are bounded and uniformly H\"older continuous with exponent $\sigma$.

\begin{definition}
Given $\rho \in [0, 1]$ and $m \in \R$, let $\Gamma_\rho^m(\R^d)$ denote the space of locally bounded functions $a(x, \xi)$ on $\R^d \times (\R^d \backslash 0)$, which are $C^\infty$ functions of $\xi$ away from the origin and such that, for any $\alpha \in \N^d$ and any $\xi \neq 0$, the function $x \mapsto \D_\xi^\alpha a(x, \xi)$ is in $W^{\rho, \infty}(\R^d)$ and there exists a constant $C_\alpha$ such that on $\{|\xi| \geq \half\}$,
$$\|\D_\xi^\alpha a( \cdot, \xi)\|_{W^{\rho, \infty}(\R^d)} \leq C_\alpha (1 + |\xi|)^{m - |\alpha|}.$$
For $a \in \Gamma_\rho^m$, we define
$$M_\rho^m(a) = \sup_{|\alpha| \leq 1 + 2d + \rho} \sup_{|\xi| \geq 1/2} \|(1 + |\xi|)^{|\alpha| - m} \D_\xi^\alpha a(\cdot, \xi)\|_{W^{\rho, \infty}(\R^d)}.$$
\end{definition}

Given a symbol $a \in \Gamma_\rho^m(\R^d)$, define the (inhomogeneous) paradifferential operator $T_a$ by
$$\widehat{T_a u}(\xi) = (2\pi)^{-d}\int \chi(\xi - \eta, \eta)\widehat{a}(\xi - \eta, \eta) \psi(\eta) \widehat{u}(\eta) \, d\eta,$$
where $\widehat{a}(\theta, \xi)$ is the Fourier transform of $a$ with respect to the first variable, and $\chi$ and $\psi$ are two fixed $C^\infty$ functions satisfying, for small $0 < \eps_1 < \eps_2$,
\begin{equation*}
\begin{cases}
\psi(\eta) = 0 \quad \text{on } \{|\eta| \leq 1\} \\
\psi(\eta) = 1 \quad \text{on } \{|\eta| \geq 2\}, \\
\end{cases}
\end{equation*}
\begin{equation*}
\begin{cases}
\chi(\theta, \eta) = 1 \quad \text{on } \{|\theta| \leq \eps_1 |\eta|\} \\
\chi(\theta, \eta) = 0 \quad \text{on } \{|\theta| \geq \eps_2 |\eta|\}, \\
\end{cases}
\qquad |\D_\theta^\alpha \D_\eta^\beta \chi(\theta, \eta)| \leq c_{\alpha, \beta} (1 + |\eta|)^{-|\alpha| - |\beta|}.
\end{equation*}
The cutoff function $\chi$ can be chosen so that $T_a$ coincides with the usual definition  of a paraproduct (in terms of a Littlewood-Paley decomposition), where the symbol $a$ depends only on $x$.

\subsection{Symbolic calculus}

We shall use results from \cite{metivier2008differential} about operator norm estimates for the pseudodifferential symbolic calculus.

\begin{definition} 
Consider a dyadic decomposition of the identity:
$$I = S_0 + \sum_{\lambda = 1} S_\lambda.$$
If $s \in \R$, the Zygmund class $C_*^s(\R^d)$ is the space of tempered distributions $u$ such that
$$\|u\|_{C_*^s} := \sup_\lambda \lambda^s\|S_\lambda u\|_{L^\infty} < \infty.$$
\end{definition}

\begin{rem}
If $s > 0$ is not an integer, then $C_*^s(\R^d) = W^{s, \infty}(\R^d).$
\end{rem}

\begin{definition}
Let $m \in \R$. We say an operator $T$ is of order $m$ if for every $\mu \in \R$, it is bounded from $H^\mu$ to $H^{\mu - m}$ and from $C_*^\mu$ to $C_*^{\mu - m}$.
\end{definition}

The main features of the symbolic calculus for paradifferential operators are given by the following proposition:

\begin{prop}
Let $m \in \R$ and $\rho \in [0, 1]$.

i) If $a \in \Gamma_0^m(\R^d)$, then $T_a$ is of order $m$. Moreover, for all $\mu \in \R$,
\begin{equation}\label{ordernorm}
\|T_a\|_{H^\mu \rightarrow H^{\mu - m}} \lesssim M_0^m(a), \quad \|T_a\|_{C_*^\mu \rightarrow C_*^{\mu - m}} \lesssim M_0^m(a).
\end{equation}
ii) If $a \in \Gamma_\rho^m(\R^d)$ and $b \in \Gamma_\rho^{m'}(\R^d)$ then $T_a T_b - T_{ab}$ is of order $m + m' - \rho$. Moreover, for all $\mu \in \R$,
\begin{align}
\label{sobolevcommutator}
\|T_a T_b - T_{ab}\|_{H^\mu \rightarrow H^{\mu - m - m' + \rho}} \lesssim M_\rho^m(a) M_0^{m'}(b) + M_0^m(a) M_\rho^{m'}(b), \\
\label{holdercommutator}\|T_a T_b - T_{ab}\|_{C_*^\mu \rightarrow C_*^{\mu - m - m' + \rho}} \lesssim M_\rho^m(a) M_0^{m'}(b) + M_0^m(a) M_\rho^{m'}(b).
\end{align}
\end{prop}

We also need to consider paradifferential operators with negative regularity. As a consequence, we need to extend our previous definition.

\begin{definition}
For $m \in \R$ and $\rho < 0$, $\Gamma_\rho^m(\R^d)$ denotes the space of distributions $a(x, \xi)$ on $\R^d \times (\R^d \backslash 0)$ which are $C^\infty$ with respect to $\xi$ and such that, for all $\alpha \in \N^d$ and all $\xi \neq 0$, the function $x \mapsto \D_\xi^\alpha a(x, \xi)$ belongs to $C_*^\rho(\R^d)$ and there exists a constant $C_\alpha$ such that on $\{|\xi| \geq \half\}$,
$$\|\D_\xi^\alpha a(\cdot, \xi)\|_{C_*^\rho} \leq c_\alpha (1 + |\xi|)^{m - |\alpha|}.$$
For $a \in \Gamma_\rho^m$, we define
$$M_\rho^m(a) = \sup_{|\alpha| \leq \frac{3d}{2} + \rho + 1} \sup_{|\xi| \geq 1/2} \|(1 + |\xi|)^{|\alpha| - m} \D_\xi^\alpha a(\cdot, \xi)\|_{C_*^\rho(\R^d)}.$$
\end{definition}

We recall Proposition 2.12 in \cite{alazard2014cauchy} which is a generalization of (\ref{ordernorm}).

\begin{prop}
Let $\rho < 0, \ m \in \R$, and $a \in \dot{\Gamma}_\rho^m$. Then the operator $T_a$ is of order $m - \rho$:
\begin{equation}\label{negativeop}
\|T_a\|_{H^s \rightarrow H^{s - (m - \rho)}} \lesssim M_\rho^m(a), \quad \|T_a\|_{C_*^s \rightarrow C_*^{s - (m - \rho)}} \lesssim M_\rho^m(a).
\end{equation}
\end{prop}

\subsection{Paraproducts and product rules}

We recall here some properties of paraproducts, $T_a$ where $a(x, \xi) = a(x)$. A key feature is that one can define paraproducts for rough functions $a$ which do not belong to $L^\infty(\R^d)$ but merely $C_*^{-m}(\R^d)$ with $m > 0$.

\begin{definition}
Given two functions $a, b$ defined on $\R^d$, we define the remainder
$$R(a, u) = au - T_a u - T_u a.$$
\end{definition}

We record here various estimates about paraproducts (see Chapter 2 in \cite{bahouri2011fourier}).

\begin{prop}
i) Let $\alpha, \beta \in \R$. If $\alpha + \beta > 0$ then
\begin{align}
\|R(a, u)\|_{H^{\alpha + \beta - \frac{d}{2}}} &\lesssim \|a\|_{H^\alpha} \|u\|_{H^\beta} \\
\label{holderparaerror}\|R(a, u)\|_{C_*^{\alpha + \beta}} &\lesssim \|a\|_{C_*^\alpha} \|u\|_{C_*^\beta} \\
\label{sobolevparaerror}\|R(a, u)\|_{H^{\alpha + \beta}} &\lesssim \|a\|_{C_*^\alpha} \|u\|_{H^\beta}.
\end{align}
ii) Let $m > 0$ and $s \in \R$. Then
\begin{align}
\label{sobolevparaproduct}\|T_a u\|_{H^{s - m}} &\lesssim \|a\|_{C_*^{-m}} \|u\|_{H^s} \\
\label{holderparaproduct}\|T_a u\|_{C_*^{s - m}} &\lesssim \|a\|_{C_*^{-m}} \|u\|_{C_*^s} \\
\label{holderparaproduct0}\|T_a u\|_{C_*^s} &\lesssim \|a\|_{L^\infty} \|u\|_{C_*^s}.
\end{align}
\end{prop}

We have the following product estimates (for references and proofs, see \cite{alazard2014strichartz}):

\begin{prop}
i) Let $s \geq 0$. Then
\begin{align}
\label{sobolevproduct}\|u_1 u_2\|_{H^s} &\lesssim \|u_1\|_{H^s} \|u_2\|_{L^\infty} + \|u_1\|_{L^\infty} \|u_2\|_{H^s} \\
\label{holderproduct}\|u_1 u_2\|_{C_*^s} &\lesssim \|u_1\|_{C_*^s} \|u_2\|_{L^\infty} + \|u_1\|_{L^\infty} \|u_2\|_{C_*^s}.
\end{align}
ii) Let 
$$s_1 + s_2 > 0, \quad s_0 \leq s_1, s_2, \quad s_0 < s_1 + s_2 - \frac{d}{2}.$$
Then
\begin{equation}\label{specialsobolevproduct}
\|u_1 u_2\|_{H^{s_0}} \lesssim \|u_1\|_{H^{s_1}}\|u_2\|_{H^{s_2}}.
\end{equation}
iii) Let $\beta > \alpha > 0$. Then
\begin{align}\label{specialholderproduct}
\|u_1 u_2\|_{C_*^{-\alpha}} \lesssim \|u_1\|_{C_*^\beta} \|u_2\|_{C_*^{-\alpha}}.
\end{align}
iv) Let $s_1 > d/2$ and $s_2 \geq 0$, and consider $F \in C^\infty(\C^N)$ such that $F(0) = 0$. Then there exists a non-decreasing function $\FF: \R_+ \rightarrow \R_+$ such that
\begin{equation}\label{smoothholder}
\|F(U)\|_{H^{s_1}} \leq \FF(\|U\|_{L^\infty}) \|U\|_{H^{s_1}}, \quad \|F(U)\|_{C_*^{s_2}} \leq \FF(\|U\|_{L^\infty}) \|U\|_{C_*^{s_2}}.
\end{equation}
\end{prop}

We also need the following elementary composition estimate (\cite[Lemma 3.2]{alazard2011strichartz}):

\begin{prop}\label{compositionlem}

Let $m \in \N$ and $0 \leq \sigma \leq m$. Consider a diffeomorphism $\kappa: \R \rightarrow \R$. Then
$$\|u \circ \kappa\|_{H^\sigma} \leq \FF(\|\kappa'\|_{W^{p - 1, \infty}}) \|\D_x (\kappa^{-1})\|_{L^\infty} \|u\|_{H^{\sigma}}.$$

\end{prop}

\bibliography{allbib}
\bibliographystyle{alpha}

\end{document}